\documentclass[a4paper,10pt]{article}

\usepackage{amssymb,latexsym, amsfonts, amsmath, amsthm}
\usepackage{graphicx}
\usepackage{setspace}
\usepackage{enumerate}
\usepackage{wrapfig}
\usepackage{mathdots}
\usepackage{lscape}
\usepackage[paperwidth=210mm, paperheight=297mm, top=1in, bottom=1in, left=1in, right=1in, asymmetric, marginparwidth=0.6in]{geometry}
\usepackage{rotating}
\usepackage{epstopdf}
\usepackage[backgroundcolor=green!40, linecolor=green!40]{todonotes}
\usepackage{tikz}
\usepackage{todonotes}
\usepackage{MnSymbol}

\newcommand{\Lie}{\text{Lie}}

\usetikzlibrary{arrows,calc,matrix,fit}

\theoremstyle{definition}
\newtheorem{theorem}[equation]{Theorem}
\newtheorem{proposition}[equation]{Proposition}
\newtheorem{conjecture}[equation]{Conjecture}
\newtheorem{notation}[equation]{Notation}

\newtheorem{definition}[equation]{Definition}
\newtheorem{lemma}[equation]{Lemma}
\newtheorem{corollary}[equation]{Corollary}
\newtheorem{remark}[equation]{Remark}

\usepackage[utf8]{inputenc}
\newcommand{\mc}{\mathcal}

\numberwithin{equation}{section}

\newcommand{\be}{\begin{enumerate}}
\newcommand{\ee}{\end{enumerate}}
\newcommand{\bi}{\begin{itemize}}
\newcommand{\ei}{\end{itemize}}
\newcommand{\beq}{\begin{equation}}
\newcommand{\eeq}{\end{equation}}
\newcommand{\mf}{\mathfrak}
\newcommand{\ra}{\rightarrow }

\def\End{\operatorname{End}}
\def\Aut{\operatorname{Aut}}

\def\Irr{\operatorname{Irr}}

\def\dim{\operatorname{dim}}

\def\id{\operatorname{id}}

\def\Gal{\operatorname{Gal}}
\def\GL{\operatorname{GL}}

\def\supp{\operatorname{supp}}

\def\Sp{\operatorname{Sp}}
\def\Res{\operatorname{Res}}

\def\I{\operatorname{I}}

\def\bbZ{\mathbb{Z}}
\def\bbR{\mathbb{R}}

\def\C{{\mathcal{C}}}

\def\P{{\rm P}}

\def\a{\alpha}

\def\>{\geqslant}
\def\<{\leqslant}

\def\black#1{\textcolor{black}{#1}}

\def\({\left(}
\def\){\right)}

\def\presuper#1#2%
  {\mathop{}%
   \mathopen{\vphantom{#2}}^{#1}%
   \kern-\scriptspace%
   #2}
\makeatletter
\tikzstyle{notestyleraw}=[
  draw=\@todonotes@currentbordercolor,
  fill=\@todonotes@currentbackgroundcolor,
  line width=0.5pt,
  text width=\@todonotes@textwidth+9ex,
  inner sep=0.8ex,
  rounded corners=4pt
]
\makeatother

\def\C{{\rm C}}
\newcommand{\ti}[1]{\tilde{#1}}
\newcommand{\trd}{\text{trd}}

\newcommand{\Seq}{\text{Seq}}
\def\loccit{\emph{loc.cit.}}

\title{Semisimple characters for inner forms I:~$\GL_m(D)$}
\date{March 2017}
\author{D. Skodlerack\footnote{Daniel Skodlerack, Shanghaitech University, Institute of Mathematica; Sciences, Middle Ring Road 393, 201210 Pudong New District, China, dskodlerack@shanghaitech.edu.cn}}

\begin{document}

\maketitle

\begin{abstract}
The article is about the representation theory of an inner form~$G$ of a general linear group over a non-Archimedean local field. 
We introduce semisimple characters for~$G$ whose intertwining classes describe conjecturally via the Local Langlands correspondence the behaviour on wild inertia.
These characters also play a potential role to understand the classification of irreducible smooth representations of inner forms of classical groups. 
We prove the intertwining formula for semisimple characters and an intertwining implies conjugacy like theorem. 
Further we show that endo-parameters for~$G$, i.e. invariants consisting of simple endo-classes and a numerical part, classify the intertwining classes of semisimple 
characters for~$G$. They should be the counter part for restrictions of Langlands-parameters to wild inertia under the Local Langlands correspondence. 
MSC2010  [11E95] [20G05] [22E50]
\end{abstract}

%

\section{Introduction}
Semisimple characters occur extensively on the automorphic side of the Local Langlands correspondence, as in the classification of cuspidal irreducible representations of~$p$-adic classical groups of odd residue characteristic or in the description of the restrictions of Langlands parameters to wild inertia on the Galois side (see \cite{KSS:21}).
For the classification of all  irreducible smooth representations of~$\GL_m(F)$ via type theory Bushnell and Kutzko developed the theory of simple characters  over a non-Archimedean local field~$F$, see~\cite{bushnellKutzko:93}. These characters are essential in the construction of types, i.e. certain irreducible representations on open compact subgroups which describe the blocks of the Bernstein decomposition. This strategy has been successfully generalized to~$\GL_m(D)$ in a series of papers
(\cite{grabitz:07,secherreStevensIV:08,broussousSecherreStevens:12,secherreStevensVI:12}) for a central division algebra of finite degree over~$F$, by S\'echerre, Stevens, Broussous and Grabitz.
In the case of~$p$-adic classical groups the simple characters were not sufficient enough for reasonable classification results as the collection of non-split tori which arise from a field extension of~$F$ are not enough to construct all irreducible cuspidal representations. So more general non-split tori were needed, more precisely the ones which arise from a product of fields. This leads to the theory of semisimple characters which has been extensively studied by Stevens and the author in~\cite{stevens:05} and~\cite{skodlerackStevens:18}, and it succeeded in the classification of irreducible cuspidal representations for~$p$-adic classical groups
with odd residual characteristic, see~\cite{stevens:08} and~\cite{KSS:21}. 

The aim of this paper is to generalize the theory of semisimple characters to~$\GL_m(D)$. It will have the potential of four major applications:
\begin{enumerate}
 \item At first it provides the foundation to generalize to semisimple characters for inner forms of~$p$-adic classical groups which leads to the classification of their 
 irreducible cuspidal representations in odd residue characteristic. 
 \item The endo-parameters which classify intertwining classes of semisimple characters have conjecturally a concrete connection to the Galois side
 of the Local Langlands correspondence. 
 \item They also have the potential to help for the study of inner forms of classical groups in even residue characteristic.
 \item For the study of the representation theory of~$\GL_m(D)$, semisimple characters provide a way to decompose and help to describe the category of smooth representations.
\end{enumerate}

Semisimple characters are complex characters defined on compact open subgroups~$H(\Delta)$
which are generated by congruence subgroups depending on arithmetic data~$\Delta$ (so called semisimple strata). 
Part of the data~$\Delta$ is an element~$\beta$ of the Lie algebra of~$G:=\GL_D(V)$ which generates a mod centre anisotropic torus~$F[\beta]^\times$ of a Levi subgroup of~$G$. 
The construction of a semisimple character is done by an inductive procedure using complex characters which factorize through reduced norms of certain
centralizers and a character which factorizes through a ``$\beta$-twisted'' reduced trace of~$\End_D(V)$. Thus, the ``essential'' data for the construction are the 
characters on the tori and the generator~$\beta$ of the torus of~$\Delta$. Semisimple characters which are defined with the same essential data but may be on different 
central algebras over~$F$ are called transfers and all semisimple characters which intertwine up to transfer create a class called an endo-class. 
These endo-classes can be broken into simple endo-classes (the associated torus is the multiplicative group of a field). And conversely given a collection of finitely many 
simple endo-classes one can unify them to a semisimple endo-class. An endo-parameter for~$G$ is a finite collection of simple endo-classes where to every endo-class
is attached a non-negative integer. 

\begin{theorem}[\textbf{2nd Main Theorem~\ref{thmEndoParameter}}]
The set of intertwining classes of full semisimple characters is in canonical bijection with the set of endo-parameters for~$G$. 
\end{theorem}

\newcommand{\bigslant}[2]{{\raisebox{.2em}{$#1$}\left/ \raisebox{-.2em}{$#2$}\right.}}
\newcommand{\Rep}{\text{Rep}}
The Local Langlands correspondence attaches to an irreducible representation of~$G$ (automorphic side) a Weil-Deligne representation~$\phi$ (Langlands parameter) of the Weil 
group~$W_F$ in the absolute Galois group of~$F$. Work on~$\Sp_{2n}(F)$ by Blondel, Henniart and Stevens \cite{blondelHenniartStevens:18} suggests that the endo-parameters should describe the 
collection of possible restrictions of Langlands parameters to the wild inertia group~$P_F$. 
For~$G$ we can define the following diagram: 
\begin{equation}\begin{matrix}\label{diagEndoparLLC}
 \Irr(G) & \stackrel{LLC}{\longrightarrow} & \text{LP}_G(W_F)\\
 \downarrow & & \downarrow \\
 \mathcal{E}_{par}(F)& \stackrel{B.-H.}{\longrightarrow} &  \bigslant{\Rep(P_F)}{W_F} \\
\end{matrix}\end{equation}

where the above horizontal line is the Local Langlands correspondence, for a description see~\cite{hiragaSaito:12}, and the bottom line is the map resulting from Bushnell-Henniart's first ramification theorem~\cite[Theorem 8.2]{bushnellHenniart:03}. The right map is the restriction map from the set~$\text{LP}_G(W_F)$ of Langlands parameters for~$G$ 
to the set of smooth representations of~$P_F$, up to~$W_F$-conjugacy, and the left map is given in the following way:
Given an irreducible representation~$\pi$ of~$G$ it should contain a full semisimple character~$\theta$ 
and two such semisimple characters must intertwine and they therefore define the same endo-parameter, where~$\mathcal{E}_{par}(F)$ is the set of endo-parameters.  
The conjecture has the following form: 

\begin{conjecture}
 Diagram~\eqref{diagEndoparLLC} is commutative.  
 \end{conjecture}
 

A common strategy for the classification of all irreducible representations of~$\mathbb{G}(F)$ for a reductive group~$\mathbb{G}$ defined over~$F$ starts with
the Bernstein decomposition of the category~$\Rep(\mathbb{G}(F))$ of all smooth representations of~$\Rep(\mathbb{G}(F))$. 
\[\Rep(\mathbb{G}(F))=\prod_{\mf{s}}\Rep(\mathbb{G}(F))_{\mf{s}},\]
where~$\mf{s}$ passes through all inertia classes of~$\mathbb{G}(F)$. We call~$\Rep(\mathbb{G}(F))_{\mf{s}}$ the Bernstein block of~$\mf{s}$.   
Now the classification can be broken into three parts:
\begin{enumerate}
 \item \label{enumClassification.i} Exhaustiveness: Find a list of types to cover all Bernstein blocks.
 \item \label{enumClassification.ii} Rigidity: Describe the relation between two types which give the same Berstein block. 
 \item Description of the  Bernstein block (using spherical Hecke algebras). 
\end{enumerate}

The construction of types for~$G$ starts with a semisimple character, say defined on a group called~$H^1=H(\Delta)$
and it is extended to an irreducible representation~$\kappa$ of a compact open group~$J\supseteq H^1$ and tensored with the inflation~$\tau$ of a cuspidal irreducible 
representation of a finite quotient~$J/J^1$. It is a similar construction in the case of~$p$-adic classical groups.  
Beautiful descriptions of this construction can be found in~\cite{bushnellKutzko:93},~\cite{secherreI:04} and~\cite{stevens:08}.

Semisimple characters are important for parts~\ref{enumClassification.i} and~\ref{enumClassification.ii}. 
Given two types which are contained in a cuspidal irreducible representation of~$G$, it is reasonable to ask if they are conjugate by an element of~$G$. 
As an example to study the non-cuspidal Bernstein blocks for~$G$, S\'echerre and Stevens constructed in~\cite{secherreStevensVI:12} (maximal) semisimple types but without 
using semisimple characters, see also the work of Bushnell and Kutzko~\cite{bushnellKutzko:99} for the split case. The advantage of the use of the theory of semisimple characters would be to approach the rigidity question. An example where this has been successful 
is the case of~$p$-adic classical groups~\cite{KSS:21}.
The base of these rigidities is an ``intertwining and conjugacy'' result for semisimple characters. 

We give a sketch of the ingredients and underline the ones which are new. 
Denote by~$\C(\Delta)$ the set of semisimple characters which are constructed with the arithmetic data~$\Delta$. 
Recall that~$\Delta$ consists of~$\beta\in\Lie(G)$ which generates a commutative semisimple algebra~$E$ over~$F$, a rational point~$\Lambda$ of the Bruhat--Tits building~$\mf{B}(G)$ of~$G$ 
which is in the image of~$\mf{B}(C_G(\beta))$,~$C_G(\beta)$ being the centralizer of~$\beta$,  by a map constructed by Broussous and Lemaire in~\cite{broussousLemaire:02}. The data~$\Delta$ consists further of two non-negative integers~$r\leq n$, where~$r$ 
indicates the biggest integer~$m$  such that the semisimple characters are defined on a subgroup of the~$1$-unit group~$P_{m+1}(\Lambda)$, and where~$n$ is the smallest integer~$m$ such that the semisimple characters are trivial on~$P_{m+1}(\Lambda)$. 
 The primitive idempotents of~$E$ split~$\theta\in\C(\Delta)$ into simple characters~$\theta_i$,~$i\in I$,  
for some general linear group~$\GL_D(V^i)$, say~$\theta_i\in \C(\Delta_i)$, and if two semisimple characters~$\theta$ and~$\theta'$ intertwine then the intertwining matches 
these simple characters to each other:

\begin{theorem}[\textbf{Matching Theorem~\ref{thmMatching}}]
 Suppose~$\theta\in\C(\Delta)$ and~$\theta'\in\C(\Delta')$  are intertwining semisimple characters, and suppose~$r=r'$ and~$n=n'$. Then there is a unique bijection~$\zeta:I\ra I'$ between the index sets 
 for~$E$ and~$E'$ such that 
 \begin{enumerate}
  \item $\dim_DV^i=\dim_DV^{\zeta(i)}$, and 
  \item there is a~$D$-linear isomorphism~$g_i$ from~$V^i$ to~$V^{\zeta(i)}$ such that~$g_i.\theta_i$ and~$\theta'_{\zeta(i)}$ intertwine by a~$D$-linear 
  automorphism of~$V^{\zeta(i)}$,
 for all indexes~$i\in I$. 
 \end{enumerate}
\end{theorem}

The statement is a generalization of the split case~$D=F$, see~\cite{skodlerackStevens:18}, but the proof heavily needs a fine study of embeddings. For the intertwining to force conjugacy we need an extra new ingredient. 
Two intertwining semisimple characters~$\theta\in\C(\Delta)$ and~$\theta'\in\C(\Delta')$ induce a canonical bijection~$\bar{\zeta}$ between the 
residue algebras~$\kappa_E$ of~$E$ and~$\kappa_{E'}$ of~$E'=F[\beta']$. When these two characters are conjugate by an element of~$G$ then the conjugation has to induce
this bijection~$\bar{\zeta}$. This leads to the ``intertwining and conjugacy'' theorem. To avoid too much notation we make the assumption that~$\Delta$ and~$\Delta'$
have the same point of~$\mf{B}(G)$. 

\begin{theorem}[\textbf{1st Main Theorem: Intertwining and Conjugacy~\ref{thmIntertwiningImplConjugacy}}]
 Suppose that~$\theta\in\C(\Delta)$ and~$\theta'\in\C(\Delta')$ are two semisimple characters such that~$n=n'$,~$r=r'$ and~$\Lambda=\Lambda'$. 
 If there is an element~$t\in G$ which normalizes~$\Lambda$ and maps the splitting of~$V$ with respect to~$\beta$ to the one with respect to~$\beta'$ such that the  conjugation with~$t$ induces the map~$\bar{\zeta}$ then there is an element~$g$ in the normalizer of~$\Lambda$ such that~$g.\theta_i=\theta'_{\zeta(i)}$, for all~$i\in I$, and such that~$g^{-1}t$ fixes~$\Lambda$. In particular~$g.\theta$ and~$\theta'$ coincide.
\end{theorem}

%
%

The article is structured as follows: 
In Section~\ref{secLattices} we collect results and study the points in the Bruhat--Tits building of~$G$ and some of its centralizers. 
It is followed by the technical heart of the article: Section~\ref{secSemisimpleStrata} about semisimple strata~$\Delta$. 
The highlights of this section is the ``intertwining and conjugacy''-result for semisimple strata (\ref{thmIntConjSesiTiG}). 
The section is followed by Section~\ref{secSemisimpleCharacters}, where the intertwining formula for transfers (Proposition~\ref{propTransferAndInt}),
the matching theorem (\ref{thmMatching}) and the ``intertwining and conjugacy'' theorem for semisimple characters (\ref{thmIntertwiningImplConjugacy}) are proven. 
The final section, Section~\ref{secEndoPar}, gives the first application of the study of semisimple characters of~$\GL_m(D)$: The theory of  endo-parameters.

I have to thank a lot Imperial College London and the Engineering and Physical Sciences Research Council (EP/M029719/1) for their financial support, in particular David Helm, who also showed me aspects of possible applications. 
Secondly the University of East Anglia and in particular Shaun Stevens, for his interest and feedback, and Vincent S\'echerre, who has pointed to me a subtlety which makes this 
theory so interesting and difficult at the same time. 

\textbf{Data Availability Statement:} Data sharing not applicable to this article as no datasets were generated or analysed during the current study.
%
%
%
%
%
%

\section{Notation}\label{secNot}

Let~$F$ be a non-Archimedean local field. We denote by~$G$ the group~$\GL_D(V)$ of invertible~$D$-linear automorphisms 
on a finite dimensional right~$D$-vector space~$V$, where~$D$ is a central division algebra over~$F$ of finite degree~$d$. We denote the $D$-dimension of~$V$ by~$m$. 
We denote by~$\nu_D,\ o_D,\ \mf{p}_D,$ and~$\kappa_D$ the valuation (with image~$\bbZ\cup\{\infty\})$, valuation ring, valuation ideal and the residue field of~$D$ and we use similar notation for non-Archimedean local fields.
Let~$L|F$ be a maximal unramified field extension in~$D$, i.e. of degree~$d$, and~$\pi_D$ be a uniformizer of~$D$ which normalizes~$L$. The conjugation by~$\pi_D$ is a 
generator~$\tau$ of the Galois group~$G(L|F)$, i.e.~$\tau(x)=\pi_D x\pi_D^{-1}$ for all elements~$x\in L$. 
By~$A$ we denote the ring of~$D$-endomorphisms of~$V$. 
We fix an additive character~$\psi_F$ of~$F$ of level~$1$, i.e.~$\psi_F|_{o_F}$ is non-trivial and factorizes through the residue field. 
We denote the character~$\psi_F\circ\trd_{A|F}$ by~$\psi_A$. For an element~$c\in A$ we define
the map~$\psi_c: A\ra \mathbb{C}$ via~$\psi_c(x):=\psi_A(c(x-1))$. 
Given a reel number~$x$ we write~$\lfloor x\rfloor$ for the greatest integer which is not greater than~$x$. We also put~$\lfloor \infty\rfloor:=\infty$ and~$\lceil x\rceil:=-\lfloor -x\rfloor$.

\section{Bruhat--Tits buildings and Embeddings}\label{secLattices}
\subsection{Bruhat--Tits buildings}\label{subsecBuidlings}
An~$o_D$-lattice sequence~$\Lambda$ in~$V$ is a function from~$\bbZ$ to the set of full~$o_D$ sub-modules of~$V$, called lattices, which is decreasing with respect to the 
inclusion and which is periodic, i.e. there is an integer~$e=e(\Lambda|D)$ such that for all integers~$i$ the lattice~$\Lambda_{i+e}$ is equal to~$\Lambda_i \pi_D$ 
for any uniformizer~$\pi_D$ of~$D$. The number~$e(\Lambda|D)$ is called the~$D$-period of~$\Lambda$. 
We call an~$o_D$-lattice sequence~$\Lambda$ {\it strict} if it is injective and we call~$\Lambda$~\emph{regular} if the~$\kappa_D$-dimension of the quotient~$\Lambda_{i}/\Lambda_{i+1}$ is independent of~$i\in\mathbb{Z}$.
A strict lattice sequence is also called~\emph{lattice chain} in the literature, see~\cite{grabitz:99}.
The square lattice sequence~$\mf{a}_\Lambda$ in~$A$ is an~$o_F$-lattice sequence  in~$A$ defined as follows: 
$\mf{a}_{\Lambda,i}$ is the set of elements of~$A$ which map~$\Lambda_j$ to~$\Lambda_{j+i}$ for all integers~$j$. Given an~$o_D$-lattice sequence~$\Lambda$ in~$V$, a positive integer~$a$ and an integer~$b$ we can define a new~$o_D$-lattice sequence~$a\Lambda+b$ in~$V$ via
\[(a\Lambda+b)_r:=\Lambda_{\lfloor \frac{r-b}{a}\rfloor}.\]
We call~$a\Lambda+b$ an \emph{affine translate} of~$\Lambda$. We call~$\Lambda-b$ an~\emph{(integral) translate} of~$\Lambda$ by~$b$.

One can generalize the notion of lattice sequence to lattice functions, i.e. where the domain is~$\bbR$, such that they can be identified with a point in the Bruhat--Tits building~$\mf{B}( G)$ of~$ G$. The square lattice function of a point is then the Moy-Prasad filtration. 

\begin{definition}
 An~$o_D$-lattice function~$\Gamma$ in~$V$ is a function from~$\bbR$ to the set of~$o_D$-lattices in~$V$ with the following properties:
\begin{enumerate}
 \item~$\Gamma$ is decreasing,
 \item For all elememts~$t\in\bbR$ the lattice~$\Gamma_{t+\frac{1}{d}}$ is 
 equal to~$\Gamma_t \pi_D$.  We say that~$\Gamma$ is~\emph{$\pi_D$-periodic} with period~$\frac{1}{d}$.
 \item~$\Gamma$ is left continuous, i.e.~$\Gamma_r$ is the intersection of all~$\Gamma_{r'}$ where~$r'$ passes through all real numbers smaller than~$r$.
\end{enumerate}
The square lattice function~$\mf{a}_\Gamma$ in~$A$ is an~$o_F$-lattice function  in~$A$ defined as follows: 
$\mf{a}_{\Gamma,t}$ is the set of elements of~$A$ which map~$\Gamma_s$ to~$\Gamma_{t+s}$ for all~$s\in\bbR$.
For an~$o_D$-lattice function~$\Gamma$ and~$s\in\bbR$ we define
the~\emph{translate of~$\Gamma$ by~$s$} to be the~$o_D$-lattice function given by
\[(\Gamma-s)_t:=\Gamma_{t+s},\ t\in\bbR.\]
\end{definition}

\begin{definition}\label{defTranslationClasses}
Two lattice functions~$\Gamma$ and~$\Gamma'$ (resp. lattice sequences~$\Lambda$ and~$\Lambda'$) are~\emph{equivalent} if they are translates of each other. We denote the equivalence class by~$[\Gamma]$, resp. by~$[\Lambda]$.  We write~$\mf{n}(\Gamma)$ for the normalizer of~$\Gamma$ in~$ G$, i.e. the set of elements~$g$  
of~$ G$ for which~$g\Gamma$ is a translate of~$\Gamma$. We denote~$(1+\mf{a}_{\Gamma,t})^\times$ by~$P_t(\Gamma)$ for non-negative~$t\in\bbR$, and we write~$P(\Gamma)$ for~$P_0(\Gamma)$. Analogously we define those objects for lattice sequences (with~$t$ a non-negative integer). If~$\mf{a}$ is the hereditary order attached to a strict lattice sequence~$\Lambda$, i.e.~$\mf{a}=\mf{a}_{\Lambda,0}$ then one can find the notation~$P_t(\mf{a})$ for~$P_t(\Lambda)$ in the literature. This includes~$P_t(o_E)$ for skew-fields~$E$.

\end{definition}

We can attach to an~$o_D$-lattice sequence an~$o_D$-lattice function~$\Gamma$ as follows:
\[\Gamma_t:=\Lambda_{\lceil te(\Lambda|F)\rceil},\ t\in\bbR.\]

We write~$C_S(T)$ for the centralizer of~$T$ in~$S$ where~$S$ is a group or an algebra acting on an ambient set of~$T$ from both sides, e.g.
the centralizer in a subgroup~$H$ of~$ G$ of an element~$\beta$ of the Lie algebra of~$ G$ is denoted by~$C_{ H}(\beta)$.
The Bruhat--Tits building of the group~$H$ in question, denoted by~$\mf{B}(H)$, is always the enlarged one.
If we refer to the reduced Bruhat--Tits building, we write~$\mf{B}_{red}(H)$. We write~$\mf{b}$ for  square lattice functions in the~$A$-centralizer~$B$ of~$E$ where~$E$ is an extension field of~$F$ in~$A$ . \black{Note that~$B$ is~$E$-algebra isomorphic to~$\End_{D_E}W$ for some central division algebra~$D_E$ over~$E$ and a right-$D_E$-vector space $W$. So the square lattice functions are attached to~$o_{D_E}$-lattice functions.}

The reduced Bruhat--Tits building~$\mf{B}_{red}( G)$ can be described using translation classes of lattice functions (or equivalently square lattice functions) and the Bruhat--Tits building~$\mf{B}( G)$ can be described using lattice functions. For more details see~\cite{broussousLemaire:02}. So we consider 
~$\mf{B}( G)$ and~$\mf{B}_{red}( G)$ as the set of lattice functions and their translation classes, respectively. 

\begin{theorem}[{{\cite[part II, Theorem 1.1]{broussousLemaire:02}}}]\label{thmBroussousLemaire1punkt1}
Let~$E|F$ be a field extension in~$A$, then there is exactly one map~$j_E$ from the set of~$E^\times$-fixed points of~$\mf{B}_{red}( G)$ to
$\mf{B}_{red}(C_{ G}(E))$
such that
\[\mf{b}_{j_E([\Lambda]),te(E|F)}=\mf{a}_{[\Lambda],t}\cap C_A(E).\]
Further,
\begin{enumerate}
 \item~$j_E$ is~$C_{ G}(E)$-equivariant,
 \item~$j_E$ is affine,
 \item~$j_E$ is bijective. 
\end{enumerate}
The map~$j_E^{-1}$ can be characterized to be the unique affine,~$C_{ G}(E)$-equivariant map from~$\mf{B}_{red}(C_{ G}(E))$ to~$\mf{B}_{red}( G)$.
\end{theorem}

The reader can find in~\cite{broussousLemaire:02} a precise description of the map~$j_E$, which we only recall for a special unramified case.  

\begin{proposition}[{{\cite[part II, Lemma 3.1]{broussousLemaire:02}}}]\label{propAlgebraicDescriptionOfBuildingEmbedding}
Let~$E|F$ be an unramified extension which has an isomorphic field extension of~$F$ in~$D$. 
The map~$j_E$ can be described as follows.~$V$ is the direct sum of~$V^i$,~$i$ passing from~$1$ to~$[E:F]$, where the idempotent~$1^i$ (projections onto~$V^i$) 
are obtained by the decomposition of~$E\otimes_F L$, where~$L$ is a maximal unramified extension of~$F$ in~$D$. Further the indexing can be adjusted such that
$V^i$ is equal to~$V^1\pi_D^{i-1}$ for all indexes~$i$.
Now, given a point~$y\in\mf{B}_{red}(C_{ G}(E))$, then the point~$x:=j_E^{-1}(y)$ is given by the translation class of the lattice function:
\[\Gamma_t:=\oplus_{i=0}^{[E:F]-1} \Sigma_{(t -\frac{i}{d})}\pi_D^i\]
where the translation class of~$\Sigma$ corresponds to the point~$y$. 
\end{proposition}

\begin{remark}\label{remjE}
There is a canonical projection map~$[\ ]$ from~$\mf{B}( G)$ to~$\mf{B}_{red}( G)$ sending~$\Lambda$ to~$[\Lambda]$. 
There is a~$C_{ G}(E)$-equivariant, affine and bijective map from~$[\ ]^{-1}(\mf{B}_{red}( G)^{E^\times})$ to~$\mf{B}(C_{ G}(E))$ which on the level of 
reduced buildings induces~$j_E$. This map is unique up to translation. We just choose one and we still call this map~$j_E$. One could take the map given by the 
formula~\cite[[part II, Lemma 3.1]{broussousLemaire:02}, a generalization of the formula in Proposition~\ref{propAlgebraicDescriptionOfBuildingEmbedding}.
\end{remark}

We will need later a more general notion of lattice sequence: 
\begin{definition}\label{defLatticeSequencOverProdOfFields}
 Let~$E|F=\prod_jE_j$ be a commutative reduced~$F$-algebra in~$A$. The idempotents split~$V$ into~$V=\oplus_j V^j$. An~$o_D$-lattice sequence~$\Lambda$ is called an~\emph{$o_E$-$o_D$-lattice sequence}\footnote{$o_E$ has not been defined here if~$E$ is not a field. We could set it as~$\oplus_jo_{E_j}$.}
if~$\Lambda$ is split by the decomposition of~$V$, i.e.~$\Lambda$ is the direct sum of the intersections~$\Lambda^j:=\Lambda\cap V^j$, 
 and~$\Lambda^j$ is an~$o_{E_j}$-lattice sequence, for every index~$j$. We define~$e(\Lambda|E)$ as the greatest common divisor of the integers~$e(\Lambda^j|E_j)$. 
\end{definition}


We need a lattice sequence version of Remark~\ref{remjE}.

\black{
For an extension field~$E$ of~$F$ in~$A$ we can use~$j_E$ from Remark~\ref{remjE} (more precisely~\cite[part II Lemma 3.1]{broussousLemaire:02}) to attach to an~$o_E$-$o_D$-lattice sequence~$\Lambda$ an~$o_{D_E}$-lattice sequence~$M$ in the following way:
Let~$\Gamma$ be the lattice function attached to~$\Lambda$ and~$\Sigma$ be~$j_E(\Gamma)$ then we define
\[M_{z}:=\Sigma_{\frac{z}{e(\Lambda|E)}},\ z\in\bbZ. \]
}

\subsection{Embedding types}

The embeddings of buildings~$j_E$ from \S\ref{subsecBuidlings} can be used to classify conjugacy classes of pairs $(E,\Lambda)$ consisting of a finite field extension of~$F$ in~$A$ and a lattice 
sequence~$\Lambda$ normalized by~$E^\times$. Such a pair is called \emph{embedding in~$A$} in analogy to the pairs~$(E,\mf{a})$ introduced and studied in~\cite{broussousGrabitz:00}. 
Let~$E_D|F$ be the unramified sub-extension of~$E|F$ whose degree is the greatest common divisor of~$d$ and~$f(E|F)$. Then two embeddings~$(E,\Lambda)$
and~$(E',\Lambda')$ in~$\End_D(V)$ and~$\End_D(V')$, respectively, are called {\it equivalent} if the field extensions~$E_D|F$
and~$E'_D|F$  are conjugate by a~$D$-isomorphism~$g$ from~$V$ to~$V'$ such that~$g\Lambda$ is equal to~$\Lambda'$ up to
integral translation. 
The equivalence classes are called  (Broussous-Grabitz) {\it embedding types}. 

\begin{proposition}[{{\cite[Proposition~3.3, Lemma~3.5]{broussousSecherreStevens:12}}}]\label{propbroussousSecherreStevens3p3u3p5}
Let~$\phi_i$,~$i=1,2$, be two~$F$-algebra embeddings of a field extension~$E|F$ into~$A$ and let~$\Lambda$ be a lattice sequence such that~$\phi_i(E)^\times$ normalizes~$\Lambda$ for both indexes~$i$. Suppose that the embeddings~$(\phi_i(E),\Lambda)$,~$i=1,2$, are equivalent. Then there is an element~$u$ of the normalizer of~$\Lambda$ in~$A^\times$ such that~$u\phi_1(x)u^{-1}$ is equal to~$\phi_2(x)$, for all~$x\in E$.
If~$\phi_1$ and~$\phi_2$ are equal on~$E_D$ then one can choose~$u$ in~$P(\Lambda)$. 
\end{proposition}


 \begin{theorem}[{see \cite[Theorem 1.2]{skodlerack:14-2}} for the strict case]\label{thmGeometricInterpretationOfEquivalenceOfEmbeddings}
Suppose~$(E,\Lambda)$ and~$(E',\Lambda')$ are two embeddings in~$A$. Suppose~$\Lambda$ and~$\Lambda'$ have the same~$\pi_D$-period. Let~$\Gamma$ and~$\Gamma'$ be the~$o_D$-lattice functions associated to~$\Lambda$ and~$\Lambda'$. Consider~$\Gamma$ and~$\Gamma'$ as points in~$\mf{B}( G)$. 
Then~$(E,\Lambda)$ and~$(E',\Lambda')$ are equivalent if and only if the 
barycentric coordinates of~$j_{E_D}([\Gamma])$ and of~$j_{E'_D}([\Gamma'])$ with respect to chambers are equal up to cyclic permutation. 
\end{theorem}

\begin{proof}
At first we prove the if-part.  There is an element~$g_1$ of~$A^\times$ which conjugates~$E_D$ to~$E'_D$. The conjugation induces an affine, simplicial isomorphism 
from~$\mf{B}_{red}( G_{E_D})$ 
 to~$\mf{B}_{red}( G_{E'_D})$ which respects barycentric coordinates up to cyclic permutation. Thus~$j_{E'_D}([g_1\Gamma])$ and~$j_{E'_D}([\Gamma'])$ have the 
 same barycentric coordinates up to cyclic permutation, and therefore  there is an element~$g_2$ of~$ G_{E'_D}$ such that~$j_{E'_D}([g_2g_1\Gamma])$ and~$j_{E'_D}([\Gamma'])$ have 
 the same barycentric coordinates and thus there is an element of~$ G_{E'_D}$ of reduced norm~$1$ which conjugates~$j_{E'_D}([g_2g_1\Gamma])$ to~$j_{E'_D}([\Gamma'])$. 
 Thus, the embeddings~$(E,\Lambda)$ and~$(E',\Lambda')$ are equivalent. 
 For the only-if-part, assume that both embeddings are equivalent, in particular there is an element~$g$ of~$A^\times$  conjugating~$E_D$ to~$E'_D$ 
 such that~$j_{E'_D}([g\Gamma])$ is equal to
~$j_{E'_D}([\Gamma'])$. In particular,~$j_{E_D}([\Gamma])$ and~$j_{E'_D}([\Gamma'])$ have the same barycentric coordinates up to cyclic permutation. 
 \end{proof}
 
 Theorem~\ref{thmBroussousLemaire1punkt1} implies that the existence of an embedding with a fixed lattice sequence only depends on the ramification index and the inertia degree: 
 
 \begin{corollary}\label{corconsructingEmbeddingsFromNumericalData}
 Suppose~$E|F$ is a field extension and~$\phi$ an injective~$F$-algebra homomorphism from~$E$ into~$A$ such that~$(\phi(E),\Lambda)$ is an embedding in~$A$. Suppose that~$E'|F$ 
 is a field extension in an algebraic closure~$\bar{E}$ of~$E$ such that~$e(E'|F)$ divides~$e(E|F)$ and~$f(E'|F)$ divides~$f(E|F)$ and~$E_D|F$ is equal to 
~$E'_D|F$. Then, there is an injective~$F$-algebra homomorphism~$\phi'$ from~$E'$ into~$A$ such that~$\phi'(E')^\times$ normalizes~$\Lambda$ and~$\phi$ is equal to~$\phi'$ on~$E\cap E'$. In particular, the embeddings~$(\phi(E),\Lambda)$ and~$(\phi'(E'),\Lambda)$ are equivalent.
 \end{corollary}

 \begin{proof}
  Without loss of generality we can assume that~$E|F$ and~$E'|F$ have the same ramification index and the same inertia degree, because we can replace~$E'$ by an extension field of~$E'$ with inertia degree~$f(E|F)$ and ramification index~$e(E|F)$. We consider an algebraic closure~$\overline{\phi(E)}$ of~$\phi(E)$ and a field isomorphism~$\psi$ from~$\bar{E}$ to~$\overline{\phi(E)}$ extending~$\phi$. We replace~$(E,\phi,\bar{E},E')$ by~$(\phi(E),\id_{\phi(E)},\overline{\phi(E)},\psi(E'))$, i.e. we can assume without loss of generality that~$E$ is an extension field of~$F$ in~$A$.
  
  We just define~$\phi'$ to be the identity on~$E\cap E'$ and we go to the centralizer~$C_A(E\cap E')$. Thus, using Theorem \ref{thmBroussousLemaire1punkt1}, we can assume that~$E\cap E'$ is~$F$. In particular both extensions are totally ramified. We extend~$E'|F$ to a purely ramified field extension
 ~$\tilde{E'}|F$ of degree~$\deg(A)$. It can be embedded into~$A$ via an injective~$F$-algebra homomorphism~$\tilde{\phi'}$. The reduced building of~$C_A(\phi'(\tilde{E'}))^\times$  only consists of one point, which is given by a mid point of a chamber   in~$\mf{B}_{red}( G)$, by Theorem \ref{thmBroussousLemaire1punkt1}. Let us call this chamber~$C$.   By conjugation we can assume that~$[\Lambda]$ is an element of the closure of~$C$. We take uniformizers~$\pi_E$ and 
 ~$\pi_{E'}$. The~$F$-valuations of the reduced norm of~$\pi_E$ and~$\tilde{\phi'}(\pi_{E'})$ coincide. Thus they rotate the Coxeter diagram of~$\mf{B}_{red}( G)$ in the same way and thus they transform barycentric coordinates the same way. In particular: The translation class of~$\Lambda$ which coincides with~$[\phi(\pi_E)\Lambda]$ has the same barycentric coordinates as~$[\tilde{\phi'}(\pi_{E'})\Lambda]$. The latter is still an element of the closure of~$C$ and must therefore coincide 
  with~$[\Lambda]$. We define~$\phi'$ to be the restriction of~$\tilde{\phi'}$ to~$E'$.  
 \end{proof}

\section{Semisimple Strata}\label{secSemisimpleStrata}

\subsection{Definitions for semisimple strata}\label{subsecSemisimpleStrataFirstDef}
Here we introduce semisimple strata for proper inner forms of the general linear groups, and we introduce a notation which makes it more convenient to work with strata.
A very good introduction for semisimple strata for the case of~$D=F$ can be foand in~\cite[\S 6]{skodlerackStevens:18} and~\cite[\S 2.4, \S 8.1]{KSS:21}.
There is a very detailed study of simple strata in~\cite{secherreI:04}. And the aim of this subsection is to state the straightforward generalizations.
A stratum is a quadruple
\[\Delta=[\Lambda,n,r,\beta]\]
such that~$\Lambda$ is an~$o_D$-lattice sequence,~$\beta$ is an element of~$\mf{a}_{\Lambda,-n}$ and~$0\leq r\leq n$.
Given a stratum~$\Delta$ and an integer~$j$ such that~$0\leq r+j\leq n$ we define~$\Delta(j+)$ to be the stratum obtained from~$\Delta$ replacing~$r$ by~$r+j$, and we define~$\Delta(j-)$ to be~$\Delta((-j)+)$ for~$0\leq r-j\leq n$. 
If~$n=r$ and~$\beta=0$ we call the stratum a~\emph{null-stratum}. 
A stratum~$\Delta$ is called~\emph{pure} if it is null or~$E:=F[\beta]$ is a field and~$\Lambda$ is an~$o_E$-lattice sequence 
and~$\nu_\Lambda(\beta)=-n$. Given a pure stratum~$\Delta$ we define the~\emph{critical exponent}~$k_0(\Lambda,\beta)$ 
 similar to \cite[Definition 2.3]{secherreI:04} in the following way: 
Let~$\mathfrak{n}_k(\beta,\Lambda)=\{x\in\mathfrak{a}_{\Lambda,0}: \beta x-x\beta\in \mathfrak{a}_{\Lambda,k}\}$ and define~$k_0(\beta,\Lambda)$ by
\[k_0(\beta,\Lambda)=\max\left\{\nu_{\Lambda}(\beta),\sup\{k\in\mathbb{Z}:\mathfrak{n}_{k}(\beta,\Lambda)\not\subseteq \mathfrak{a}_{j_E(\Lambda),0}+\mathfrak{a}_{\Lambda,1}\}\right\},\]
for non-zero~$\beta$ and~$k_0(0,\Lambda)=-\infty$. This is slightly different from the definition in \cite[before Definition 2.3]{secherreI:04}. In~\loccit\ the critical exponent for non-zero elements of~$F$ is~$-\infty$. If we consider~$\Lambda$ as an~$o_F$-lattice sequence, we write~$\Lambda_F$, then the critical exponent satisfies
\begin{equation}\label{eqCriticalExponent}
k_0(\beta,\Lambda_F)=e(\Lambda|E)k_0(\beta,\mf{p}_E^\bbZ),
 \end{equation}
 see~\cite[before Lemma 5.6]{stevens:01}.
 Given a positive integer~$a$ and an integer~$b$ then the~\emph{affine translates} of a stratum~$\Delta=[\Lambda,n,r,\beta]$ with respect to~$(a,b)$ are the strata
 \[[a\Lambda+b,rn,s,\beta],\ s\in\{ar,ar+1,\ldots,(a+1)r-1\}.\]

\begin{notation}
If we write that we are given a stratum~$\Delta'$ then we mean that the entries have the superscript~$'$, i.e.~$\Delta'$ is equal to~$[\Lambda',n',r',\beta']$. This also applies to all other superscripts. 
For subscripts we have the following rule:~$\Delta_i$ is the stratum~$[\Lambda^i,n_i,r_i,\beta_i]$. This also defines the notation~$E',V',E_i,V^i$ etc. .
\end{notation}

A pure stratum~$\Delta$ is called \emph{simple} if
~$k_0(\beta,\Lambda)<-r$, in particular if~$n=r$ then the stratum has to be null. 
Two strata~$\Delta$ and~$\Delta'$ are called equivalent if~$r=r'$ and for all integers~$s\leq -r$
the coset~$\beta+\mf{a}_{\Lambda,s}$ is equal to~$\beta'+\mf{a}_{\Lambda',s}$.
\begin{proposition}[{\cite[Theorem 2.2]{secherreIII:05}}]\label{propPureIsEquivToSimple}
 Let~$\Delta$ be a pure stratum. Then there is a simple stratum~$\Delta'$ equivalent to~$\Delta$ such that the unramified sub-extension of~$E'|F$ is contained
 in~$E|F$. 
\end{proposition}

To define semisimple strata it is convenient to use the direct sum of strata: 
We recall that the~\emph{period} of a stratum~$\Delta$ is the~$o_D$-period of~$\Lambda$.
Let~$\Delta'$ and~$\Delta''$ be two strata such that~$r'=r''$ and with the same period, then we define
\[\Delta'\oplus \Delta'':=[\Lambda'\oplus\Lambda'',\max(n',n''),r',\beta'\oplus\beta''],\]
and if a stratum~$\Delta$ decomposes in this way, it is called~\emph{split} by the decomposition~$V'\oplus V''$ with projections~$\Delta|_{V'}$ and~$\Delta|_{V''}$. 
A stratum~$\Delta$ is called~\emph{semi-pure} if~$\Delta$ is a direct sum of pure strata
\[\Delta=\bigoplus_{i\in I}\Delta_i,\ E_i:=F[\beta_i],\]
such that~$\beta$ generates over~$F$ the product of the fields~$E_i$. We call the strata~$\Delta_i$ the blocks of~$\Delta$.
We write~$B$ for the centralizer of~$E$ in~$A$. 


\begin{definition}\label{defSemisimpleStratum}
A semi-pure stratum~$\Delta$ is called~\emph{semisimple} if the direct sum~$\Delta_{i_1}\oplus\Delta_{i_2}$ is not 
equivalent to a pure stratum (or equivalently simple stratum) for all different indexes~$i_1,i_2 \in I$ and such that~$\Delta_i$ is a simple stratum for all~$i\in I$. 
\end{definition}

Given a non-null semisimple stratum~$\Delta$ with~$r=0$ we can define the critical exponent~$k_0(\beta,\Lambda)$ analogously as 
in~\cite[Definition 8.2]{KSS:21} to be 
\[k_0(\beta,\Lambda)=-\min\{s\in\bbZ\ |\ n\geq s>  0\text{ and }[\Lambda,n,s,\beta]\text{ is not semisimple}\}.\]

\begin{remark}
Analogously to~\cite[Remark 8.3]{KSS:21} this can be generalized to all pairs~$(\beta,\Lambda)$ where~$\beta$ generates over~$F$ a product of fields
which decomposes~$\Lambda$ into~$o_{E_i}$-lattice sequences.
\end{remark}

We also write~$k_0(\Delta)$ for~$k_0(\beta,\Lambda)$.
We have extension and restriction of scalars for strata: Given a stratum~$\Delta$ and a finite field extension~$\tilde{F}|L$ we define extension of scalars of~$\Delta$ to~$\ti{F}$ by 
$\Delta\otimes_F\tilde{F}=[\Lambda\otimes_{o_L} o_{\tilde{F}},n,r,\beta\otimes_F1]$
seen as a stratum with respect to~$\End_D(V)\otimes_F\ti{F}$, and we define the restriction of scalars of~$\Delta$ to a sub-skewfield~$\ti{D}$ of~$D$ to be the stratum 
$\Res_{\ti{D}}(\Delta)=[\Lambda,n,r,\beta]$ seen as a stratum of~$\End_{\ti{D}}(V)$. For example~$\Delta\otimes L$ and~$\Res_F(\Delta)$ are very important. 
The extension of scalars to~$\ti{F}$ starting from a semi-pure stratum~$\Delta$ comes equipped with an action of the group~$\Aut(\ti{F}|F)$ ($F$-linear
field automorphisms of~$\ti{F}$) on the set of  blocks of~$\Delta\otimes\ti{F}$ induced form the action on~$\End_D(V)\otimes \ti{F}$ on the second factor.  We recall that the 
intertwining from~$\Delta$ to~$\Delta'$ is the set~$I(\Delta,\Delta')$ of elements~$g\in G$ such that 
$g(\beta+\mf{a}_{-r})g^{-1}$ intersects~$(\beta'+\mf{a}'_{-r'})$. We say that an element~$g\in G$ intertwines~$\Delta$ with~$\Delta'$
if~$g$ is an element of~$I(\Delta,\Delta')$. 

The following construction attaches to a stratum a strict stratum in a canonical way. 
\begin{definition}[{\cite[\S 2.5]{broussousSecherreStevens:12}}]\label{defFirstDdagForStrata}
 Given a stratum~$\Delta$ we define $\Delta^\dag=\oplus_{i=0}^{e(\Lambda|F)-1} [\Lambda-i,n,r,\beta]$. Now~$\Delta^\dag$ is strict. 
\end{definition}

\subsection{Fundamental strata}
Let~$\Delta$ be a stratum such that~$n=r+1$. We recall the construction of the characteristic and the minimal polynomial of~$\Delta$ which we denote 
by~$\chi_\Delta$ and~$\mu_\Delta$, for more details see~\cite[Definition 2.5]{secherreStevensVI:12}. 
Let~$g$ be the greatest common divisor of~$e(\Lambda|F)$ and~$n$. We put~$\mf{y}(\Delta):=\pi_F^{\frac{n}{g}}\beta^{\frac{e(\Lambda|F)}{g}}$ and write~$\bar{\mf{y}}(\Delta)$ for  the class of~$\mf{y}(\Delta)$ in~$\mf{a}_{\Lambda_F,0}/\mf{a}_{\Lambda_F,1}$. We denote by~$\chi_\Delta$ and~$\mu_\Delta$ (elements of~$\kappa_F[X]$) the characteristic and the minimal polynomial 
of~$\bar{\mf{y}}(\Delta)$, respectively. These polynomials only depend on the equivalence class of~$\Delta$. 
The stratum~$\Delta$ is called \emph{fundamental} if~$\mu_\Delta$ is not a power of~$X$. We generalize the following 
proposition.

\begin{proposition}[{{\cite[Proposition 2.7]{secherreStevensVI:12}}}]\label{propCriteriaFundtoSimplStratum}
 Let~$\Delta$ be a stratum with~$n=r+1$. Then~$\Delta$ is equivalent to a non-null simple stratum if and only if~$\mu_\Delta$ is irreducible and different from~$X$. 
\end{proposition}

To state the analogue criterion for semisimple strata we need a multiplication map: 
For an element~$b\in \mf{a}_{-n}(\Lambda)$ and an integer~$t$, the map 
\[
m_{t,n,b}: \mf{a}_{-tn}/\mf{a}_{1-tn}\ra\mf{a}_{-(t+1)n}/\mf{a}_{1-(t+1)n}
\]
is defined via multiplication with~$b$.

\begin{proposition}\label{propCriteriaFundtoSemisimplStratum}
  Let~$\Delta$ be a stratum with~$n=r+1$. Then~$\Delta$ is equivalent to a semisimple stratum if and only if the minimal 
  polynomial $\mu_\Delta$ is  square-free and for every integer~$t$ the kernel 
of~$m_{t+1,n,\beta}$ and the image of~$m_{t,n,\beta}$ intersect trivially.  
\end{proposition}

For the proof we need two lemmas.

\begin{lemma}\label{lemCriteriaFundtoZerostratum}
 Let~$\Delta$ be a stratum with~$n=r+1$. Then~$\Delta$ is equivalent to a null-stratum if and only if~$\mu_\Delta$ is equal to~$X$ and for every integer~$t$ the kernel 
of~$m_{t+1,n,\beta}$ and the image of~$m_{t,n,\beta}$ intersect trivially. 
\end{lemma}

\begin{proof}
 We only have to prove the ``if''-part. The condition on~$\mu_\Delta$ implies that~$\beta^{e(\Lambda|F)}$ is an element 
 of~$\mf{a}_{-e(\Lambda|F)n+1}$. We apply the condition on the multiplication maps successively on
 \[m_{e(\Lambda|F)-1,n,\beta}\circ m_{e(\Lambda|F)-2,n,\beta}\circ\ldots \circ m_{1,n,\beta}\circ m_{0,n,\beta}\] successively to obtain 
 that~$ m_{0,n,\beta}$ is the zero-map. 
\end{proof}

\begin{lemma}\label{lemLiftingSplitting}
Let~$\Lambda$ be an~$o_D$-lattice sequence, and suppose that~$\Lambda_F$ is split by a decomposition~$V=V^1+V^2$ into~$F$ vector spaces such that the corresponding 
idempotents~$1^1$ and~$1^2$ satisfy~$x1^1x^{-1}-1^1\in\mf{a}_{\Lambda_F,1}$, for all~$x\in D^\times$. Then there is an element~$u\in P_1(\Lambda_F)$ such that~$uV^1$ and~$uV^2$ 
are~$D$-vector spaces.
\end{lemma}

\begin{proof}
The classes $\bar{1}^1$ and~$\bar{1}^2$ in~$\mf{a}_{\Lambda_F,0}/\mf{a}_{\Lambda_F,1}$ split the~$\kappa_D$-vector spaces~$\Lambda_i/\Lambda_{e(\Lambda|D)}$ for all~$0\leq i<e(\Lambda|D)$.
Thus by the Lemma of Nakayama we find a decomposition~$V=\tilde{V}^1\oplus\tilde{V}^2$ into~$D$-vector spaces such the idempotents satisfy~$\bar{\tilde{1}}^i=\bar{1}^i$ for~$i=1,2$. 
In particular~$\Lambda$ is split by the latter decomposition, and the map~$u$ defined via~$u(v):=\tilde{1}^1 1^1(v)+\tilde{1}^2 1^2(v)$ is an element of~$P_1(\Lambda_F)$
such that~$uV^i=\ti{V}^i$, $i=1,2$.
\end{proof}

%

\begin{proof}[Proof of Proposition~\ref{propCriteriaFundtoSemisimplStratum}]
 We prove the ``if''-part. The ``only-if''-part is easy and left for the reader.  
 As in the proof of~\cite[Proposition 6.11]{skodlerackStevens:18}~$\Res_F(\Delta)$ is equivalent to a stratum~$\Delta'_F$ which is a direct sum of strata
 corresponding to the prime decomposition of~$\mu_\Delta=\mu_{\Res_F(\Delta)}$. Those direct summands corresponding to the prime factors different from~$X$ are equivalent to a simple stratum by Proposition~\ref{propCriteriaFundtoSimplStratum}. 
 So, we can assume that those summands are already simple. Using the square-freeness of~$\mu_\Delta$ again, which forces~$\kappa_F[\bar{\mf{y}}(\Delta)]$ to be isomorphic to the product 
 of the residue fields corresponding to the summands, we obtain for every idempotent~$e$
 of the decomposition of~$V$ the existence of a polynomial~$P\in o_F[X]$ such that~$P(\mf{y}(\Delta))$ is congruent to~$e$ modulo~$\mf{a}_{\Lambda_F,1}$. 
 These idempotents satisfy the condition of Lemma~\ref{lemLiftingSplitting}, because~$\beta$ is~$D$-linear. Thus by conjugating the idempotents with an element~$u\in P_1(\Lambda_F)$ 
 we can assume that the constituents of the decomposition of~$V$ are in fact~$D$-sub-vector spaces. This decomposition also splits~$\Delta$ into a sum of strata
 which are equivalent to simple strata by Proposition~\ref{propCriteriaFundtoSimplStratum} and Lemma~\ref{lemCriteriaFundtoZerostratum}.
\end{proof}

Proposition~\ref{propCriteriaFundtoSemisimplStratum} has the following consequence:

\begin{corollary}\label{corEquivToMinSemisimpleStratum}
 Let~$\Delta$ be a stratum with~$n=r+1$ and let~$\ti{F}|L$ be an unramified field extension. Then the following assertions are equivalent:
 \begin{enumerate}
  \item The stratum~$\Delta$ is equivalent to a semisimple stratum.\label{corEquivToMinSemisimpleStratumAss1}
  \item The stratum~$\Delta\otimes \ti{F}$ is equivalent to a semisimple stratum.\label{corEquivToMinSemisimpleStratumAss2}
  \item The stratum~$\Res_F(\Delta)$ is equivalent to a semisimple stratum.\label{corEquivToMinSemisimpleStratumAss3}
 \end{enumerate}
\end{corollary}

\begin{proof}
We only need to study the equivalence to semipure strata. 
 Assertion~\ref{corEquivToMinSemisimpleStratumAss1} implies~\ref{corEquivToMinSemisimpleStratumAss2} and~\ref{corEquivToMinSemisimpleStratumAss3} by the definition of equivalence of strata. 
 All three strata have the same minimal polynomial, and we therefore have to consider the intersection condition on the multiplications maps.
 If this intersection condition is satisfied for~$\Res_F(\Delta)$ then it is also satisfied by~$\Delta\otimes L$ and if it is satisfied for~$\Delta\otimes L$
 then it is satisfied for~$\Delta$, all just by inclusion. Thus for the case~$\ti{F}=L$ all three assertions are equivalent by 
 Proposition~\ref{propCriteriaFundtoSemisimplStratum}.  In the case~$\ti{F}\neq L$ we have that if~$\Delta\otimes \ti{F}$ is equivalent to a semisimple stratum then 
 again by Proposition~\ref{propCriteriaFundtoSemisimplStratum}~$\Delta\otimes L$ is equivalent to a semisimple stratum, which finishes the proof.  
 \end{proof}
 
\begin{corollary}\label{corEquivToMinSimpleStratum}
 Let~$\Delta$ be a stratum with~$n=r+1$. Then the following assertions are equivalent:
 \begin{enumerate}
  \item The stratum~$\Delta$ is equivalent to a simple stratum.\label{corEquivToMinSimpleStratumAss1}
  \item The stratum~$\Res_F(\Delta)$ is equivalent to a simple stratum.\label{corEquivToMinSimpleStratumAss2}
 \end{enumerate}
\end{corollary}

\begin{proof}
 If~$\Delta$ is equivalent to a simple stratum then~$\Res_F(\Delta)$ is equivalent to a pure stratum and thus by Proposition~\ref{propPureIsEquivToSimple} equivalent to a 
 simple stratum. We prove now that~\ref{corEquivToMinSimpleStratumAss1} follows from~\ref{corEquivToMinSimpleStratumAss2}.
 Suppose that~$\Res_F(\Delta)$ is equivalent to a simple stratum. Thus the polynomial~$\mu_{\Res_F(\Delta)}$ is irreducible and~$\Res_F(\Delta)$ satisfies the trivial intersection property 
 for the multiplication maps by Proposition~\ref{propCriteriaFundtoSimplStratum} and Lemma~\ref{lemCriteriaFundtoZerostratum}. Thus the same is true for~$\Delta$, noting that~$\mu_{\Delta}$ is equal to~$\mu_{\Res_F(\Delta)}$. Thus~$\Delta$ is equivalent to a simple stratum by~\ref{propCriteriaFundtoSimplStratum} and~\ref{lemCriteriaFundtoZerostratum}. 
\end{proof}

One method to prove a question for all semisimple strata  is to use strata induction, see~\cite[Paragraph after Remark 7.2]{skodlerackStevens:18}. It is an induction over~$r$, where one proves at first the statement for all semisimple 
strata of type~$(n,n-1)$ followed by the induction step, where one reduces the problem for a stratum~$\Delta$ to a problem for~$\Delta(1+)$ and a derived stratum which is 
equivalent to a stratum of type~$(r+1,r)$. For that we need to recall
\begin{itemize}
 \item the tame corestriction and 
 \item the derived stratum.
\end{itemize}

We have the following natural isomorphisms:
\begin{enumerate}
 \item $A\otimes_F \End_A(V)\cong \End_F(V)$ where~$\End_A(V)\cong D^{op}$,
 \item $C_A(E)\otimes_F D^{op}\cong C_{\End_F(V)}(E)$, for every field extension~$E|F$ in~$A$.
\end{enumerate}

\begin{definition}[{{\cite[Definition 2.25]{secherreStevensIV:08} for the simple case}}]\label{defTameCor}
 Given a commutative reduced~$F$-algebra~$E$ in~$A$ a map~$s$ from~$A$ to~$B:=C_A(E)$ is called a~\emph{tame corestriction on~$A$ relative to~$E|F$} if the map 
 ~$s\otimes_F\id_{\End_A(V)}$ is a tame corestriction on~$\End_F(V)$ relative to~$E|F$ in the sense of~\cite[Paragraph after Notation 6.17]{skodlerackStevens:18}, or equivalently in the sense
  of~\cite[(1.3.3)]{bushnellKutzko:93} if~$E$ is a field. 
\end{definition}

Following the notation of Definition~\ref{defTameCor} the decomposition~$E=\prod_{i\in I}E_i$ as a product of fields gives decompositions~$1=\Sigma_{i\in I}1^i$,~$1^i:=1_{E_i}$, and~$B=\prod_{i\in I}B_i$, for~$B_i=1^iB 1^i$, and we put~$A_i:=1^i A 1^i$. We fix, for each~$i\in I$, an additive character~$\psi_{E_i}:A_i\rightarrow B_i$ of level~$1$ and define~$\psi_{B_i}$ as~$\psi_{E_i}\circ\trd_{B_i|E_i}$ using the reduced trace. 
By the remark after~\cite[Definition 2.25]{secherreStevensIV:08} (refering to~\cite[Lemma (4.2.1)]{broussous:99} and essentially to~\cite[Proposition (1.3.4)]{bushnellKutzko:93}) there exists, for each~$i\in I$, a tame corestriction~$\ti{s}_i$
on~$A_i$ relative to~$E_i|F$. By~\cite[Proposition~(1.3.4)(i)]{bushnellKutzko:93}, see also the sentence before, there is an element~$u\in o_{E_i}^\times$ such that~$u\ti{s}_i$, which we denote by~$s_i$, is a tame corestictions on~$A_i$ relative to~$E_i|F$ such that
\[\psi_{B_i}(s_i(a))=\psi_{A}(a),~a\in~A_i.\]
Define~$\psi_B$ via~$\psi_B(b)=\prod_{i\in I}
\psi_{B_i}(1^ib),$~$b\in B$, and 
$s(a):=\Sigma_{i\in I}s_i(1^i a 1^i),\ a\in A.$
Then~$s$ is a tame corestriction on~$A$ relative to~$E|F$ such that 
\[\psi_A(a)=\psi_B(s(a)),\ a\in A.\]


\begin{definition}\label{defDerivedstrata} 
Let~$\gamma\in A$ be an element which generates a field~$F[\gamma]$ and~$s$ be a tame corestriction on~$A$ relative to~$F[\gamma]|F$. Let~$\Delta$ be a stratum with~$n\geq r+1$. Suppose~$\Lambda$ (given in~$\Delta$) is an~$o_{F[\gamma]}-o_D$-lattice sequence. The stratum
\[\partial_{\gamma,s}(\Delta):=[j_{F[\gamma]}(\Lambda),\max(r,-\nu_\Lambda(\beta-\gamma)),r,s(\beta-\gamma)]\]
is called the derived stratum of~$\Delta$ with respect to~$\gamma$ and~$s$. $\partial_{\gamma,s}(\Delta)$ is a stratum in
$\End_{F[\gamma]\otimes_FD}V$.
\end{definition}

If~$\gamma$ generates a product of fields~$F[\gamma]=\prod_{j\in J}F[\gamma_j]$ and~$\Lambda$ is an~$o_{F[\gamma]}$-$o_D$-lattice sequence then we obtain several derived strata in the following way. 
Let~$s$ be a tame corestriction on~$A$ relative to~$F[\gamma]|F$. We define for~$j\in J$, 
\[\Delta_j:=[\Lambda^j,n,r,\beta_j:=1^j\beta 1^j].\]
The restriction~$s_j$ of~$s$ to~$\End_DV^j$. is a tame corestriction with image in~$\End_{F[\gamma_j]\otimes_FD}V^j$. We apply Definition~\ref{defDerivedstrata} for every~$j\in J$ with respect to~$\gamma_j$ and~$s_j$ to get the derived strata~$\partial_{\gamma_j,s_j}(\Delta_j)$, $j\in J$. 

The main property for strata induction is: 

\begin{proposition}\label{propDerived}
Let~$\Delta$ be a stratum with~$n\geq r+1$ such that~$\Delta(1+)$ is equivalent to a semisimple stratum with last entry~$\gamma$, such that~$F[\gamma]=\prod_{j\in J}F[\gamma_j]$ with field extensions~$F[\gamma_j]|F$. Suppose we are given a tame corestriction~$s$  on~$A$ relative to~$F[\gamma]|F$.
Then are equivalent:
\begin{enumerate}
 \item The stratum~$\Delta$ is equivalent to a semisimple stratum. \label{propDerivedAss1}
 \item For every~$j\in J$ the derived stratum~$\partial_{\gamma_j,s_j}(\Delta_j)$ is equivalent to a semisimple stratum.\label{propDerivedAss2} 
\end{enumerate}
\end{proposition}

We have to postpone the proof of this Proposition to~\S\ref{secMainTheoremOfStrataInduction}, because we need more theory. Nevertheless, the related proposition for simple strata induction is already proven:

\begin{proposition}[{{\cite[Proposition 2.13, Proposition 2.14]{secherreStevensVI:12}, see also~\cite[(2.2.8), (2.4.1)]{bushnellKutzko:93}}}]\label{propDerivedSimple}
 Let~$\Delta$ be a stratum with~$n\geq r+1$ such that~$\Delta(1+)$ is equivalent to a simple stratum with entry~$\gamma$ and let~$s$ be a tame corestriction
  on~$A$ relative to~$F[\gamma]|F$. Then~$\Delta$ is equivalent to a simple stratum if and only if~$\partial_{\gamma,s}(\Delta)$ is equivalent to a simple stratum.  
\end{proposition}

The property of being fundamental respects the~$\dag$-construction and the restriction. 
\begin{remark}\label{remLevelDdagRes}
Let~$\Delta$ be a stratum such that~$n=r+1$. Then
 the stratum~$\Delta$ is fundamental if and only if~$\Delta^\dag$ is fundamental if and only if~$\Res_F(\Delta)$ is fundamental.
\end{remark}

\subsection{Level of a stratum}

\begin{definition}\label{defLevelsOfstrata}
 Let~$\Delta$ be a stratum.
 The~\emph{ level~$l(\Delta)$} of~$\Delta$ is defined as~$\frac{n}{e(\Lambda|F)}$.
\end{definition}

As an immediate consequence the strata~$\Delta$,~$\Res_F(\Delta)$,~$\Delta^\dag$ and~$\Delta\otimes L$ share the level.

\begin{proposition}[{\cite[Proposition 6.9]{skodlerackStevens:18} for the split case}]\label{propIntertwiningAndLevels}
 Suppose that~$\Delta$ and~$\Delta'$ are strata with~$n=r+1$ and~$n'=r'+1$ and suppose they intertwine.
 \begin{enumerate}
  \item If both strata have the same level and~$\Delta$ is fundamental then~$\Delta'$ is fundamental. 
  \item If both strata are fundamental then both strata have the same level.  
 \end{enumerate}
\end{proposition}

\begin{proof}
By passing to affine translates of the strata, if necessary, we can assume that~$\Lambda$ and~$\Lambda'$ have the same period over~$D$. Recall that~$\Delta$ is fundamental if and only if~$\Res_F(\Delta^\dag)$ is 
fundamental. Thus by Remark~\ref{remLevelDdagRes} we can assume without loss of generality that~$\Lambda$ and~$\Lambda'$ are regular lattice sequences of the same period 
and~$D=F$. Thus we can conjugate~$\Lambda$ to~$\Lambda'$ by an element of~$A$ and we can therefore assume~$\Lambda=\Lambda'$. 

\begin{enumerate}
 \item Both strata have the same level and~$\Lambda=\Lambda'$. Thus~$n=n'$. The strata intertwine so they have the same characteristic polynomial. Thus~$\Delta'$ is fundamental because the characteristic polynomial of~$\Delta$ is not a power of~$X$. 
 \item Suppose~$l(\Delta)<l(\Delta')$. Then~$n<n'$ because~$\Lambda=\Lambda'$. Take an element~$g$ of~$A$ which intertwines~$\Delta$ with~$\Delta'$, say
 \[g(\beta+a)g^{-1}=\beta'+a',\ a\in\mf{a}_{1-n},\ a'\in\mf{a}'_{1-n'}.\]
 Then
 \[  ((\beta'+a')^{e}\pi_F^{n'})^t=g(\beta+a)^{et}\pi_F^{nt}g^{-1}\pi_F^{(n'-n)t},\ e=e(\Lambda|F), \]
 is an element of~$g\mf{a}_tg^{-1}$. If~$t$ is big enough then~$g\mf{a}_tg^{-1}$ is a subset of~$\mf{a}'_1$. Thus, the minimal polynomial of~$\bar{\mf{y}}(\Delta')$ is a power of~$X$.
 Its characteristic polynomial divides a power of the minimal polynomial and is therefore a power of~$X$. Thus~$\Delta'$ is non-fundamental. A contradiction.
\end{enumerate}
\end{proof}

\begin{corollary}\label{corSameLevelSemisimpleStrata}
 Let~$\Delta=[\Lambda,n,r,\beta]$ and~$\Delta'=[\Lambda',n',r',\beta]$ be two intertwining semisimple strata.
 \begin{enumerate} 
  \item\label{corSameLevelSemisimpleStrata.i} 
  If~$e(\Lambda|F)=e(\Lambda'|F)$ and~$r=r'$ then either both strata are null or both strata are non-null. 
  \item\label{corSameLevelSemisimpleStrata.ii} If both strata are non-null then~$l(\Delta)=l(\Delta')$.
  \item \label{corSameLevelSemisimpleStrata.iii} If both strata null,~$e(\Lambda|F)=e(\Lambda'|F)$ and~$r=r'$
 then~$l(\Delta)=l(\Delta')$.
 \end{enumerate}
\end{corollary}

\begin{proof}
\begin{enumerate}
 \item Suppose~$e(\Lambda|F)=e(\Lambda'|F)$ and~$r=r'$. Assume that~$\Delta$ is non-null 
 and~$\Delta'$ is null, in particular~$n>r=r'$. Thus~$[\Lambda,n,n-1,\beta]$ and~$[\Lambda',n',n',0]$ ($n'=r'$) intertwine and Proposition~\ref{propIntertwiningAndLevels} implies that~$[\Lambda,n.n-1,\beta]$ 
 is not fundamental. A contradiction, because~$\Delta$ is non-null and semisimple.
 \item If both strata are non-null then~$n>r$ and~$n'>r'$. Then~$[\Lambda,n,n-1,\beta]$ and~$[\Lambda',n',n'-1,\beta']$ are intertwining fundamental strata. So they have the same level by Proposition~\ref{propIntertwiningAndLevels} again. 
 \item Both strata are null. Therefore~$n=r$ and~$n'=r'$. It follows that both strata have the same level. 
  \end{enumerate}
\end{proof}

 \subsection{Simple strata}\label{secSimpleStrata}

In \cite[\S 3.2]{secherreI:04} the author studies quasi-simple characters. One goal of this subsection is to prove that the underlying strata of those characters are semisimple, see Proposition~\ref{propQuasiSimpleStrataAreSemisimple}. Further we prove two diagonalization theorems for simple strata, and we recall results about intertwining. 
So let~$\ti{F}|F$ be a finite unramified field extension such that the degree of~$D$ divides~$[\ti{F}:F]$. Without loss of generality let us assume 
that~$L|F$ is a sub-extension of~$\ti{F}|F$. 
In this subsection let~$\Delta$ be a semi-pure stratum. Then~$\Delta\otimes {\ti{F}}$ is again a semi-pure stratum.

\begin{proposition}[{{\cite[Theorem 2.23]{secherreI:04},\cite[(2.4.1)]{bushnellKutzko:93}}}]\label{propSimplePure}
 Let~$\Delta$ be a pure stratum then for every simple
 stratum~$\Delta'$ equivalent to~$\Delta$ we have that~$e(E'|F)$ divides~$e(E|F)$ and~$f(E'|F)$ divides~$f(E|F)$.
Further, the following assertions are equivalent: \begin{enumerate}
  \item $\Delta$ is simple.
  \item $\Res_F(\Delta)$ is simple. 
 \end{enumerate}
\end{proposition}

\begin{corollary}[{{see \cite[Definition 6.1, Proposition 6.4]{skodlerackStevens:18} for the split case}}]\label{corSimpleStrataFieldCriteria}
 Let~$\Delta$ be a pure stratum with~$n>r$. Then~$\Delta$ is simple if and only if there is no pure stratum~$\Delta'$ equivalent to~$\Delta$
 such that~$[E':F]$ is smaller than~$[E:F]$.
\end{corollary}

\begin{proof}
 By Proposition~\ref{propSimplePure} we only have to prove the "if"-part. By Proposition~\ref{propPureIsEquivToSimple}~$\Delta$ is equivalent to a simple stratum~$\Delta'$. The ''if''-condition and Proposition~\ref{propSimplePure} imply~$[E:F]=[E':F]$ and that~$\Res_F(\Delta')$ is simple. Thus, by the known split case,~$\Res_F(\Delta)$ is simple and therefore~$\Delta$ is simple, again by Proposition~\ref{propSimplePure}.
\end{proof}

\begin{proposition}\label{propQuasiSimpleStrataAreSemisimple}
 If~$\Delta$ is simple. Then~$\Delta\otimes\tilde{F}$ is semisimple, and both strata have the same critical exponent. 
\end{proposition}

\begin{proof}
We write~$\tilde{\Delta}$ for~$\Delta\otimes\ti{F}$ which also defines~$\ti{\beta}$ 
and~$\ti{\Lambda}$. The statement holds if~$\beta$ is an element of~$F$ because the~$\Lambda$- and the~$\Lambda\otimes_{o_L}o_{\tilde{F}}$-valuation of~$\beta$ coincide. So let us suppose~$\beta\not\in F$.
In the non-null case define~$k'_0(\tilde{\Delta})$ to be the maximum of all integers~$k$ such that~$\mf{n}_k(\tilde{\beta},\ti{\Lambda})$
is not a subset of~$(\mf{b}_{0}\otimes_{o_L} o_{\ti{F}})+(\mf{a}_{1}\otimes_{o_L} o_{\ti{F}})$. This is equal to~$k_0(\Delta)$ by \cite[Corollary 2.10]{secherreI:04}. 
The stratum~$\ti{\Delta}$ is semi-pure with blocks~$\ti{\Delta}_j$, $j\in\ti{I}$.
The element~$\ti{\beta}_i$ has over~$F$ the same minimal polynomial as~$\beta$, but not over~$\ti{F}$. 
The definition of~$k_0(\ti{\Delta}_j)$ and~$k'_0(\ti{\Delta})$ imply~$k_0(\ti{\Delta}_j)\leq k'_0(\ti{\Delta})$ for every index~$j$.  Thus, by the equality of~$k'_0(\ti{\Delta})$
with~$k_0(\Delta)$ we obtain that the blocks~$\ti{\Delta}_i$ are simple. We have to prove that~$\ti{\Delta}_{j_1}\oplus\ti{\Delta}_{j_2}$ is not equivalent to a simple stratum for~$j_1\neq j_2$. 
If there would be two blocks of~$\ti{\Delta}$ whose sum is equivalent to a simple stratum then we could view these blocks as two simple strata~$\ti{\Delta}_{j_1}$ 
and~$\ti{\Delta}_{j_2}$ 
in one vector space over~$\ti{F}$ and~$\ti{\Delta}_{i_1}^\dag$ and~$\ti{\Delta}^\dag_{i_2}$ are conjugate up to equivalence, by~\cite[Theorem 8.1]{skodlerackStevens:18}, 
because of conjugate lattice sequences and intertwining, by~\cite[Theorem 6.16]{skodlerackStevens:18}. Thus, we can assume without loss of generality that~$\ti{\Delta}_{j_1}$ 
and~$\ti{\Delta}_{j_2}$ are equivalent. Further, we have that the elements~$\beta_{j_1}$ and~$\beta_{j_2}$ have the same minimal polynomial over~$F$.
Thus by Lemma \ref{lemConjFieldExtUnramDetCriteria}~$\beta_{j_1}$ and~$\beta_{j_2}$ have the same minimal polynomial over~$\ti{F}$. A contradiction. 

By the proven first assertion we get~$k_0(\tilde{\Delta})\leq k_0(\Delta)=k'_0(\ti{\Delta})$, and by~\cite[Lemma 3.7]{stevens:05} we 
obtain~$k'_0(\ti{\Delta})\leq k_0(\ti{\Delta})$. Thus the critical exponents of~$\Delta$ and~$\Delta\otimes\ti{F}$ coincide. 
\end{proof}

\begin{lemma}\label{lemConjFieldExtUnramDetCriteria}[For~$D=F$]
Let~$F'|F$ be an unramified field extension in~$\End_F(V)$ and assume that~$\Delta'_1$ and~$\Delta'_2$ are two equivalent pure strata of~$\End_{F'}(V)$ such 
that~$\beta'_1$ and~$\beta'_2$ have the same minimal polynomial over~$F$.
Then~$\beta'_1$ and~$\beta'_2$ have the same minimal polynomial over~${F'}$.
\end{lemma}

\begin{proof}
Step 1: We prove at first that we can conjugate~$\beta'_1$ to~$\beta'_2$ by an element~$g$ of the normalizer of~$F'$ in~$\End_F(V)$. 
We write~$E'_1$ for~$F[\beta'_1]$. By Skolem-Noether there is an element~$g'$ of~$\Aut_F(V)$ which conjugates~$\beta'_1$ to~$\beta'_2$. 
But then~$E'_1 F'$ and~$E'_1 g'^{-1}F'g'$ are 
fields which are isomorphic over~$E'_1$. We apply Skolem-Noether again, to obtain the existence of an~$E'_1$-automorphism~$g''$ of~$V$ which conjugates~$E'_1g'^{-1}F'g'$ to~$E'_1F'$.
The uniqueness of unramified sub-extensions implies that~$F'$ is normalized by~$g''g'^{-1}$ whose inverse is the desired element~$g$. This proves the assertion of Step 1.

Step 2: We take the element~$g$ of Step 1. Then there is an~$E'_1F'$ automorphism of~$V$ which conjugates~$g^{-1}\Lambda'$ to~$\Lambda'$ and we can assume that~$g\Lambda'$ is equal to~$\Lambda'$. Thus~$g$ normalizes the stratum~$\Delta'_1$ up to equivalence, 
and is therefore an element of~$P(\Lambda'_{E'_1F'})(1+\mf{a}_{\Lambda',1})$. In particular, the conjugation by~$g$ induces the identity on the residue field of~$E'_1F'$. Thus,~$g$ centralizes~$F'$ which finishes the proof. 
\end{proof}

The key theorem to start a study of semisimple strata is the diagonalization theorem. Let us recall that we define for a semisimple stratum~$\Delta$ with critical exponent~$k_0=k_0(\Delta)$ the 
set~$\mf{m}(\Delta)$ to be the set~$\mf{m}_{-(r+k_0)}(\beta,\Lambda)$, where~$\mf{m}_s(\beta,\Lambda)$ is the intersection of~$\mf{n}_{s+k_0}(\beta,\Lambda)$ with~$\mf{a}_{s}$ for all non-negative integers~$s$. We sometimes omit the argument~$(\beta,\Lambda)$ and write for semisimple strata~$\Delta$ and~$\Delta'$ just~$\mf{m}_s$ and~$\mf{m}'_s$.   

\begin{theorem}\label{thmDiagonalization}
 Let~$\Delta$ be a semi-pure stratum which splits under a decomposition~$V=\oplus_j V^j$ into a direct sum of 
 pure strata and suppose that~$\Delta$ is equivalent to a simple stratum~$\Delta'$. Then, there is an element~$u$ of~$1+\mf{m}(\Delta')$ such that
 $u.\Delta'$ is split by  the above decomposition.
\end{theorem} 

For the proof we need the intertwining of a simple stratum. 

\begin{proposition}\label{propIntertwiningSimpleStratum}
 The intertwining of a simple stratum~$\Delta$ is given by the formula
 \begin{equation}\label{eqIntFormSimpleStratum}
 (1+\mf{m}(\Delta))C_{A^\times}(\beta)(1+\mf{m}(\Delta))
 \end{equation}
\end{proposition}

\begin{proof}
 We have this formula for the split case~$\Delta\otimes L$ by~\cite[Theorem 6.19]{skodlerackStevens:18}. The critical exponents of~$\Delta$ and~$\Delta\otimes L$
 coincide by Proposition~\ref{propQuasiSimpleStrataAreSemisimple}. Thus we only need to take the~$\Gal(L|F)$-fixed points. 
 Now apply~\cite[Lemma 2.35]{secherreI:04} using~\cite[\S 2.4.3, Lemma 3.5]{secherreI:04} to obtain that taking~$\Gal(L|F)$-fixed points at~\eqref{eqIntFormSimpleStratum}($\Delta\otimes L)$
 results in~\eqref{eqIntFormSimpleStratum}($\Delta$).
\end{proof}

\begin{proof}[Proof of Theorem~\ref{thmDiagonalization}]
 We write~$\Delta_j$ for~$\Delta|_{V^j}$ as usual. The $p$-adic closure of~$\prod_jI(\Delta_j)$ in~$A$ is a subset of
 the~$p$-adic closure of~$I(\Delta')$. Thus, for all~$j$, there is an element~$\alpha_j$ of the~$A$-centralizer~$B'$ of~$\beta'$ congruent to the projection~$1^j$ modulo~$\mf{m}(\Delta')$.  
 Now~\cite[Lemma 7.14]{skodlerackStevens:18} provides pairwise orthogonal idempotents~$1'^j$ in~$B'$ which are congruent to~$\alpha_j$ modulo~$\mf{m}(\Delta')$ and 
 which sum up to~$1$. The map~$u:=\oplus_j 1^j1'^j$ is an element of~$1+\mf{m}(\Delta')$ because
 \[u-1=\sum_j(1^j-1'^j)1'^j\in\mf{m}(\Delta'),\]
 and further~$u.\Delta'$ is split by the given splitting and it is equivalent to~$\Delta'$, ie. to~$\Delta$.   
\end{proof}

We need a modification of the second part of Proposition~\ref{propSimplePure} to equivalence to simple strata.

\begin{proposition}\label{propGeneralizedPureSimple}
 Let~$\Delta$ be a stratum. Then~$\Delta$ is equivalent to a simple stratum if and only if~$\Res_F(\Delta)$ is equivalent to a simple stratum. 
\end{proposition}

\begin{proof}[Proof of Proposition~\ref{propGeneralizedPureSimple}]
 We prove this statement by strata induction. For~$n=r$ both strata are null strata. Suppose now that~$n>r$. If~$\Delta$ is equivalent to a simple stratum 
 then~$\Res_F(\Delta)$ is too, by Proposition~\ref{propSimplePure}. If~$\Res_F(\Delta)$ is equivalent to a simple stratum then~$\Res_F(\Delta(1+))$ is equivalent to a 
 simple stratum and the induction hypothesis states that~$\Delta(1+)$ is equivalent to a simple stratum, say with entry~$\gamma$, and let~$s$ be a tame corestriction on~$A$ relative to~$F[\gamma]|F$. 
 The stratum~$\partial_{\gamma,s\otimes \id_D}(\Res_F(\Delta))$ is equivalent to a direct sum of translates of copies of~$\Res_{F[\gamma]}(\partial_{\gamma,s}(\Delta))$, by formula~\cite[part II, Lemma 3.1]{broussousLemaire:02}.
 Thus~$\partial_{\gamma,s}(\Delta)$ and~$\partial_{\gamma,s\otimes \id_D}(\Res_F(\Delta))$ have the same minimal polynomial which is irreducible by Proposition~\ref{propCriteriaFundtoSimplStratum}
 because~$\partial_{\gamma,s\otimes \id_D}(\Res_F(\Delta))$ is equivalent to a simple stratum by Lemma~\ref{propDerivedSimple}. If the minimal polynomial is different from~$X$ 
 then~$\partial_{\gamma,s}(\Delta)$ is equivalent to a non-null simple stratum by  Proposition~\ref{propCriteriaFundtoSimplStratum}. If the minimal polynomial is~$X$ then 
 ~$\partial_{\gamma,s\otimes \id_D}(\Res_F(\Delta))$ is equivalent to a null stratum by Proposition~\ref{propCriteriaFundtoSimplStratum}, and in this case~$\partial_{\gamma,s}(\Delta)$ is
 equivalent to a null stratum. Now Lemma~\ref{propDerivedSimple} states that~$\Delta$ is equivalent to a simple stratum.
\end{proof}

Let us recall the intertwining implies conjugacy result for simple strata:

\begin{proposition}[{{\cite[Proposition 1.9]{broussousSecherreStevens:12},\cite[(4.1.3)]{broussousGrabitz:00}}}]\label{propINtImplConSimpleD}
Let~$\Delta$ and~$\Delta'$ be simple with~$n=n'$ and~$r=r'$. Suppose~$(E,\Lambda)$ and~$(E',\Lambda')$ have the same embedding type. 
Then, granted~$I(\Delta,\Delta')\neq \emptyset$, there is an element~$g\in G$ such that~$g\Delta$ is equivalent to~$\Delta'$. 
Moreover,~$g$ can be taken, such that it conjugates the maximal unramified extension of~$F$ in~$E$ to the one in~$E'$.
\end{proposition}

\black{We say that two simple strata~$\Delta$ and~$\Delta'$ have the same embedding type if the embeddings~$(F[\beta],\Lambda)$ and~$(F[\beta'],\Lambda')$ have the same embedding type.}

Proposition~\ref{propGeneralizedPureSimple} and the Theorem~\ref{thmDiagonalization} imply a finer version of the latter theorem:

\begin{proposition}\label{propDiagonalizationFiner}
 Let~$\Delta$ be a semi-pure stratum which splits under a decomposition~$V=\oplus_{j\in J} V^j$ into a direct sum of 
 pure strata and suppose that~$\Res_F(\Delta)$ is equivalent to a simple stratum~$\Delta_F$. Then, there is an element~$u$ of~$1+\mf{m}(\Delta_F)$ such that
 $u.\Delta_F$ is equivalent to~$\Res_F(\Delta)$, split by the above decomposition and such that~$u\beta_F u^{-1}$ is~$D$-linear. 
\end{proposition}

\begin{proof}
 By Theorem~\ref{thmDiagonalization} we can assume that~$\Delta_F$ is split by~$(V^j)_{j\in J}$ without loss of generality. Thus we can reduce to~$|J|=1$ because~$\mf{m}(\Delta_{F,j})$ is a subset of~$1^j\mf{m}(\Delta_F)1^j$ by~$k_0(\Delta_{F,j})\leq k_0(\Delta_F)$. 
  We can assume that~$\Delta$ is simple by Proposition~\ref{propPureIsEquivToSimple} and thus~$\Res_F(\Delta)$ is simple by 
 Proposition~\ref{propSimplePure}. The strata~$\Delta_F$ and~$\Res_F(\Delta)$ are equivalent and thus Proposition~\ref{propSimplePure} implies that~$E|F$ 
 and~$F[\beta_F]|F$ have the same ramification index and the same inertia degree. Now 
 Corollary~\ref{corconsructingEmbeddingsFromNumericalData} provides an element~$\beta'\in A$ of the same minimal 
 polynomial as~$\beta_F$ such that~$\Lambda$ is normalized by~$F[\beta']^\times$ and such that the unramified sub-extensions of~$F[\beta']|F$ and~$E|F$ coincide. 
 The stratum~$\Delta':=[\Lambda,n,r,\beta']$ is simple by Proposition~\ref{propSimplePure} because~$\Res_F(\Delta')$ is simple. Indeed Proposition~\ref{propPureIsEquivToSimple} and 
 Proposition~\ref{propINtImplConSimpleD}, both applied over~$F$, imply that~$\Res_F(\Delta)$ is conjugate to~$\Delta_F$ up to equivalence and Corollary~\ref{corSimpleStrataFieldCriteria} then provides the simplicity of~$\Res_F(\Delta)$. 
 The direct sum~$\Delta'\oplus\Delta$ is equivalent to a simple stratum by Proposition~\ref{propGeneralizedPureSimple}
 because~$\Res_F(\Delta\oplus\Delta')$ is equivalent to a simple stratum. Now we can use Theorem~\ref{thmDiagonalization} to obtain a simple stratum~$\Delta''$ equivalent 
 to~$\Delta\oplus\Delta'$ split by~$V\oplus V$. Thus~$\Delta$ and~$\Delta'$ intertwine and are conjugate up to equivalence by Proposition~\ref{propINtImplConSimpleD}.
 (Note that both have the same embedding type.) Thus we can assume that~$\Delta$ and~$\Delta'$ are equivalent. 
 We have to show that there is an element~$u\in 1+\mf{m}(\Delta_F)$ which conjugates~$\beta_F$ to~$\beta'$. There is an element~$g\in P(\Lambda_F)$ which conjugates
 $\beta_F$ to~$\beta'$. But this element~$g$ is an element of~$I(\Delta_F)\cap \mf{a}_{\Lambda_F,0}^\times=(1+\mf{m}(\Delta_F))\mf{b}_{\Lambda_F,0}^\times$, i.e.~$g=ub$. Thus~$u$
 conjugates~$\beta_F$ to~$\beta'$.
\end{proof}

\begin{corollary}\label{corEndoEquivalenceForSimpleStrata}
Let~$\Delta$,~$\Delta'$ and~$\Delta''$ be three strata which are equivalent to simple strata such that~$r=r'=r''$ and~$n=n'=n''$ and of the same period.
\begin{enumerate}
 \item\label{corEndoEquivalenceForSimpleStrataAss.i}  Suppose~$V=V'$ and suppose that~$\Res_F(\Delta)$ and~$\Res_F(\Delta')$ intertwine. Suppose that~$(V^j)_{j\in J}$ splits~$\Delta$ and~$(V'^{j'})_{j'\in J'}$ splits~$\Delta'$.
 Then~$\Delta\oplus\Delta'$ is equivalent to a simple stratum which is split by~$\bigoplus_jV^j\oplus\bigoplus_jV^{j'}$. 
 \item\label{corEndoEquivalenceForSimpleStrataAss.ii}  Suppose that~$\Delta\oplus\Delta'$ and~$\Delta''\oplus\Delta'$
 are equivalent to a simple strata. Then~$\Delta\oplus\Delta''$ and~$\Delta\oplus\Delta'\oplus\Delta''$ are equivalent to a simple stratum. 
 \item\label{corEndoEquivalenceForSimpleStrataAss.iii}  Suppose that~$\Delta\oplus\Delta'$ is a pure stratum. Then~$\Delta$,~$\Delta'$ and their direct sum have the same critical exponent which is equal 
 to~$e(\Lambda|E)k_F(\beta)$ where~$k_F(\beta)$ is the critical exponent~$k(\beta,\mf{p}_E^\bbZ)$. 
 \item\label{corEndoEquivalenceForSimpleStrataAss.iv}  Two equivalent simple strata have the same critical exponent, coinciding inertia degrees and coinciding ramification indexes. 
\end{enumerate}
\end{corollary}

\begin{proof}
\begin{enumerate}
 \item The direct sum of~$\Res_F(\Delta)$ and~$\Res_F(\Delta')$ is equivalent to simple stratum by~\cite[ Proposition 7.1]{skodlerackStevens:18}, and 
Proposition~\ref{propDiagonalizationFiner} finishes the proof.
 \item Applying Theorem~\ref{thmDiagonalization} we obtain 
 that~$\ti{\Delta}=\Delta^{\oplus\dim_DV'}\oplus\Delta'^{\oplus\dim_DV''}$ and~$\ti{\ti{\Delta}}=\Delta'^{\oplus\dim_DV}\oplus\Delta''^{\oplus\dim_DV'}$ intertwine. Thus, 
 by~\ref{corEndoEquivalenceForSimpleStrataAss.i},~$\ti{\ti{\Delta}}\oplus\ti{\Delta}$ is equivalent to a simple stratum split by the whole splitting. 
 Thus,~$\Delta\oplus\Delta''$ and~$\Delta\oplus\Delta'\oplus\Delta''$ are equivalent to a pure strata, i.e. to a simple strata by Proposition~\ref{propPureIsEquivToSimple}. 
 \item We have~$k_0(\Delta)\leq k_0(\Delta\oplus\Delta')$ by definition and if~$\Delta$ would be simple and~$\Delta\oplus\Delta'$ not, then the latter is by 
 Corollary~\ref{corSimpleStrataFieldCriteria} and Theorem~\ref{thmDiagonalization} equivalent to a simple stratum split by~$V\oplus V'$ and of lower degree. 
 But then~$\Delta$ is equivalent to a simple stratum of lower degree which is not possible by Corollary~\ref{corSimpleStrataFieldCriteria}. 
 The stratum~$\Delta^\dag=\oplus_{i=0}^{e(\Lambda|F)-1} [\Lambda-i,n,r,\beta]$ is strict with~$k_0(\Delta)=k_0(\Delta^\dag)$ by the first part of this assertion. 
 Thus we only have to prove the formula for the critical exponent in the case where~$\Lambda$ is a lattice chain. By Proposition~\ref{propSimplePure}
 we can restrict to the split case, i.e.~$D=F$. This case can be found in~\cite[(1.4.13)(ii)]{bushnellKutzko:93}.
 \item Suppose~$\Delta$ and~$\Delta'$ are equivalent simple strata. Then~$\Res_F(\Delta^\dag)$ and~$\Res_F(\Delta'^\dag)$ are equivalent simple strata by Proposition~\ref{propSimplePure} and~\ref{corEndoEquivalenceForSimpleStrataAss.iii}, and they satisfy \ref{corEndoEquivalenceForSimpleStrataAss.iv} 
 by~\cite[(2.4.1)]{bushnellKutzko:93}. Thus~$\Delta$ and~$\Delta'$ satisfy~\ref{corEndoEquivalenceForSimpleStrataAss.iv} because~$\Delta$,~$\Delta^\dag$ and~$\Res_F(\Delta^\dag)$ have the same critical exponent by Proposition~\ref{propSimplePure} and~\ref{corEndoEquivalenceForSimpleStrataAss.iii}.
 \end{enumerate}
\end{proof}

\begin{corollary}\label{corEquivalentCriteriaForIntertwining}
Suppose~$\Delta$ and~$\Delta'$ are two simple strata such that~$V=V'$,~$r=r'$ and~$n=n'$. Then the following statements are equivalent:
\begin{enumerate}
 \item The stratum~$\Res_F(\Delta\oplus\Delta')$ intertwines with a pure stratum (or equivalently with a simple stratum). \label{corEquivalentCriteriaForIntertwiningAss1}
 \item The stratum~$\Delta\oplus\Delta'$ is equivalent to a simple stratum which is split under the decomposition~$V\oplus V$.\label{corEquivalentCriteriaForIntertwiningAss2}
 \item $I(\Delta,\Delta')\neq \emptyset$. \label{corEquivalentCriteriaForIntertwiningAss3}
 \item $I(\Res_F(\Delta),\Res_F(\Delta'))\neq \emptyset$. \label{corEquivalentCriteriaForIntertwiningAss4}
\end{enumerate}
 \end{corollary}

\begin{proof}
We leave this as an exercise for the reader. 
%
%
\end{proof}

\subsection{Semisimple strata}
In this section we apply~\S\ref{secSimpleStrata} to obtain similar results for semisimple strata.  We fix a finite unramified field extansion~$\tilde{F}|L$.

\begin{proposition}\label{propSemisimpleAndExtensionOfScalar}
Suppose~$\Delta$ is a semi-pure stratum. Then strata~$\Delta\otimes\ti{F}$ is semisimple if and only if~$\Delta$ is semisimple. 
\end{proposition}

\begin{proof}
Write~$\ti{\Delta}$ for~$\Delta\otimes\ti{F}$.
1) Suppose~$\Delta$ is semisimple. The simple case is Proposition~\ref{propQuasiSimpleStrataAreSemisimple}.  
So let us assume without loss of generality that~$\Delta$ is semisimple with exactly two simple blocks, i.e.~$\Delta=\Delta_1\oplus\Delta_2$.
Write~$\ti{\Delta}_i:=\Delta_i\otimes {\ti{F}}$. It is semi-pure with blocks~$\ti{\Delta}_{i j}$. 
From~\emph{ibid.} follows that we only need to show that there is no pair of strata~$(\tilde{\Delta}_{1,j_1},\tilde{\Delta}_{2,j_2})$ such that~$\tilde{\Delta}_{1,j_1}\oplus\tilde{\Delta}_{2,j_2}$ is equivalent to a simple stratum.  
%
 Let us assume the contrary for some pair~$(j_1,j_2)$, i.e. the stratum~$\ti{\Delta}_{1 j_1}\oplus \ti{\Delta}_{2 j_2}$ is equivalent to a simple stratum.
Then~$\Res_F(\ti{\Delta})$ is equivalent to a simple stratum by Corollary~\ref{corEndoEquivalenceForSimpleStrata}\ref{corEndoEquivalenceForSimpleStrataAss.ii}. 
But~$\Res_F(\ti{\Delta})$ is a direct sum of copies of~$\Res_F(\Delta)$. Thus by Theorem~\ref{thmDiagonalization} the latter stratum is equivalent to a simple stratum, 
and~$\Delta$ is equivalent to a simple stratum by Proposition~\ref{propGeneralizedPureSimple}. This gives a contradiction.  

2) Let us now assume that~$\ti{\Delta}$ is semisimple and non-null. If~$\Delta$ is pure then multiply~$\beta$ with a negative power of~$\pi_F$ such 
that~$[\Lambda,n+ze(\Lambda|F),0,\pi_F^{-z}\beta]$ is simple. So we can assume without loss of generality that~$\Delta(r-)$ is simple. Then the critical exponents of~$\Delta(r-)$ and~$\ti{\Delta}(r-)$
coincide by Proposition~\ref{propQuasiSimpleStrataAreSemisimple}. Thus~$\Delta$ is simple. 
Suppose~$\Delta$ is semi-pure with two blocks~$\Delta=\Delta_1\oplus\Delta_2$. Then~$\Delta_1$ and~$\Delta_2$ are simple by the simple case. 
If~$\Delta$ is equivalent to a simple stratum~$\Delta'$, then it is equivalent to a simple stratum which is split by~$V^1\oplus V^2$. So~$\Delta_1$ and~$\Delta_2$ have the same
inertia degree by Corollary~\ref{corSimpleStrataFieldCriteria}. Therefore the number of 
blocks of~$\ti{\Delta}$ is at least twice the number of blocks of~$\Delta'\otimes \ti{F}$. Thus~$\ti{\Delta}$ cannot be semisimple by Proposition 
\cite[Proposition 7.1]{skodlerackStevens:18}. A contradiction. 
\end{proof}

\begin{lemma}\label{lemSemiPureDirectSummandsOfSemisimpleStrata}
 Suppose~$\Delta\oplus\Delta'$ is a semisimple stratum and~$\nu_{\Lambda}(\beta)=-n$ and~$\nu_{\Lambda'}(\beta')=-n'$. Then~$\Delta$ is semisimple.
\end{lemma}

\begin{proof}
 The strata~$\Delta$ and~$\Delta'$ are semi-pure because~$\ti{\Delta}:=\Delta\oplus\Delta'$ is. We consider the associated splitting~$(V^i)_{i\in I}$ of~$\Delta$. The stratum~$\Delta_i$ is a direct summand of a simple block of~$\ti{\Delta}$. Thus it is simple by 
 Corollary~\ref{corEndoEquivalenceForSimpleStrata}\ref{corEndoEquivalenceForSimpleStrataAss.iii}. If~$\Delta_{i_1}\oplus\Delta_{i_2}$ is equivalent to a simple stratum 
 for different indexes~$i_1$ and~$i_2$, then the sum of the corresponding simple blocks of~$\tilde{\Delta}$ is also equivalent to a simple stratum by
 Corollary~\ref{corEndoEquivalenceForSimpleStrata}\ref{corEndoEquivalenceForSimpleStrataAss.ii}. A contradiction. 
\end{proof}

\begin{proposition}\label{propEquivSemisimpleSplittingAndRestriction}
Suppose~$\Delta$ is a semi-pure stratum.
The following assertions are equivalent:
\begin{enumerate}
 \item The stratum~$\Delta$ is semisimple.\label{thmEquivSemisimpleSplittingAndRestrictionAss1}
 \item The stratum~$\Delta\otimes \ti{F}$ is semisimple.\label{thmEquivSemisimpleSplittingAndRestrictionAss2}
 \item The stratum~$\Res_F(\Delta\otimes \ti{F})$ is semisimple.\label{thmEquivSemisimpleSplittingAndRestrictionAss3} 
 \item The stratum~$\Res_F(\Delta)$ is semisimple.\label{thmEquivSemisimpleSplittingAndRestrictionAss4}
\end{enumerate}
In particular, all mentioned strata have the same critical exponent. 
\end{proposition}

\begin{proof}
 The first two assertions are equivalent by Proposition \ref{propSemisimpleAndExtensionOfScalar}. 
 Assertion \ref{thmEquivSemisimpleSplittingAndRestrictionAss4} and \ref{thmEquivSemisimpleSplittingAndRestrictionAss1} are equivalent by 
 Proposition~\ref{propSimplePure} and Proposition~\ref{propGeneralizedPureSimple}. 
The equivalence of the third and the fourth assertion follows from Lemma~\ref{lemSemiPureDirectSummandsOfSemisimpleStrata}, because~$\Res_F(\Delta\otimes\ti{F})$ 
is a direct summand of~$\Res_F(\Delta)\otimes\ti{F}$ and~$\Res_F(\Delta)$ is a direct summand of~$\Res_F(\Delta\otimes\ti{F})$. 
The last remark of the theorem is now a direct consequence of the definition of the critical exponent. 
\end{proof}

Proposition \ref{propEquivSemisimpleSplittingAndRestriction} allows us to transfer many results from \cite{skodlerackStevens:18} about semisimple strata to the case of inner forms. 
S\'echerre's and Stevens' Cohomology argument, see \cite{secherreI:04} and \cite{stevens:01}  together using Hilbert 90 for the unramified extension~$L|F$ gives the formula 
for the intertwining for a semisimple stratum over~$D$. 
 
\begin{proposition}\label{propIntertwiningOfSemisimpleStratumOverD}
 Suppose~$\Delta$ and~$\Delta'$ with~$\beta=\beta'$ are two semisimple strata. Then the set of~$D$-automorphisms of~$V$ which intertwines the first
 with the second stratum is equal to 
 \[I(\Delta,\Delta')=(1+\mf{m}'_{-(r'+k_0(\Delta'))})C_{A^\times}(\beta) (1+\mf{m}_{-(r+k_0(\Delta))}).\] 
\end{proposition}

\begin{proof}
 By \cite[Theorem 6.22]{skodlerackStevens:18} we have a similar formula for~$I(\Delta\otimes L,\Delta'\otimes L)$:
 \[(1+\mf{m}'_{-(r'+k_0(\Delta'\otimes L))})C_{(A\otimes L)^\times }(\beta\otimes 1) (1+\mf{m}_{-(r+k_0(\Delta\otimes L))}).\]
 From Proposition~\ref{propEquivSemisimpleSplittingAndRestriction} follows that~$\Delta$ and~$\Delta\otimes L$ have the same critical exponent, and similar for $\Delta'$.  We intersect the latter formula with~$\Aut_D(V)$ and by \cite[Lemma 2.35]{secherreI:04} using Hilbert 90 this intersection is the product of the intersections of the three factors with~$\Aut_D(V)$. This finishes the proof. 
\end{proof}

\begin{corollary}\label{corDecompAreConjForEqSemisimpleStrata}
 Let~$\Delta$ and~$\Delta'$ be two equivalent semisimple strata. Then there is a bijection~$\zeta: I\ra I'$ such that~$1^i$ is congruent 
 to~$1^{\zeta(i)}$ modulo~$\mf{a}_{1}$. And there is an element~$u\in 1+\mf{m}_{-r-k_0}$ such that~$uV^i$ is equal to~$V'^{\zeta(i)}$ for all indexes~$i\in I$. 
\end{corollary}

 \begin{proof}
We consider the critical exponents~$k_0$ and~$k'_0$ of~$\Delta$ and~$\Delta'$, respectively. Suppose~$k_0\geq k'_0$, then~$\mf{m}_{-r-k'_0}$ is a subset of~$\mf{m}'_{-r-k_0}$ because~$\mf{n}_{-r}(\beta,\Lambda)$ is equal to~$\mf{n}_{-r}(\beta',\Lambda')$ by  equivalence of strata. 
By Proposition~\ref{propIntertwiningOfSemisimpleStratumOverD} we have the equality 
\[(1+\mf{m}_{-r-k_0})C_{A}(\beta) (1+\mf{m}_{-r-k_0})=(1+\mf{m}_{-r-k_0})C_{A}(\beta') (1+\mf{m}_{-r-k_0}),\]
which implies for every index~$i\in I$ the existence of an element~$e_i\in C_{A}(\beta')\cap\mf{a}_0$ which is congruent to~$1^i$ 
modulo~$\mf{m}_{-r-k_0}$. Now Lemma~\cite[Lemma 7.13]{skodlerackStevens:18} implies that there is an idempotent~$\tilde{e}_i$ in~$C_{A}(\beta')$
congruent to~$e_i$ and thus to~$1^i$ modulo~$\mf{m}_{-r-k_0}$. Lemma~\cite[Lemma 7.16]{skodlerackStevens:18} (applied to the algebra~$C_{A}(\beta')$)
shows that~$\tilde{e}_i$ is central in~$C_{A}(\beta')$. Symmetrically every primitive idempotent~$1^{i'}$ of~$F[\beta']$ is congruent to a central idempotent of~$C_{A}(\beta)$. 
Thus, there is a bijection~$\zeta$ from~$I$ to~$I'$ such that~$1^i$ is congruent to~$1^{\zeta(i)}$ modulo~$\mf{m}_{-r-k_0}$. 
The map~$u:=\sum_i1^{\zeta(i)}1^i$ is an element of~$1+\mf{m}_{-r-k_0}$ which satisfies~$uV^i=V'^{\zeta(i)}$,~$i\in I$. Therefore~$\Delta'_{\zeta(i)}$ is equivalent to~$u.\Delta_i$. 
Now~$u$ is an element of~$I(\Delta')\cap\mf{a}_0^\times$ which is~$\mf{b}'^\times_0(1+\mf{m}'_{-r-k'_0})$,
i.e.~$u$ is of the form~$u=b'u'$ with~$b'\in \mf{b}'^\times_0$ and~$u'\in1+\mf{m}'_{-r-k'_0}$. Thus~$u'$  has the desired properties.
\end{proof}

\subsection{Strata and Embedding types}
The aim of this subsection is to establish an equivalent description of embedding type, see Proposition~\ref{propequivalentToSameEmbedType}.

\begin{lemma}[{\cite[Lemma 5.2]{broussousGrabitz:00},\cite[Lemma 2.11]{secherreStevensVI:12}}]\label{lemEquivSimpleGiveSameEmbeddingType}
 If two simple strata~$\Delta$ and~$\Delta'$ with~$\Lambda=\Lambda'$ are equivalent then they have the same embedding type, and the embeddings~$(E_D,\Lambda)$ and~$(E'_D,\Lambda)$ are conjugate 
 by an element of~$P_1(\Lambda)$.
\end{lemma}

\begin{proposition}\label{propFieldEmbedingsConjugateUnderaUnit}
Let~$\phi_i$,~$i=1,2$, be to~$F$-algebra embeddings of a field extension~$E|F$ into~$A$, and let~$\Lambda$ be a lattice sequence such that~$\phi_i(E)^\times$ normalizes
$\Lambda$ for both indexes~$i$. Suppose further that 
there is an element~$g$ of~$P(\Lambda)$ such that the conjugation by~$g$ seen as an endomorphism of~$\mf{a}_{\Lambda,0}/\mf{a}_{\Lambda,1}$ 
restricts on the residue field of~$\phi_1(E_D)$ to the map induced by~$\phi_2\circ\phi_1^{-1}$. Then there is an 
element~$u\in P(\Lambda)$ such that~$u\phi_1(x)u^{-1}$ is equal to~$\phi_2(x)$, for all~$x\in E$. If~$E=E_D$ then we can choose~$u$ in~$g P_1(\Lambda)$. 
\end{proposition}

\begin{proof}
 We can restrict to the case where~$E$ is equal to~$E_D$, by Proposition \ref{propbroussousSecherreStevens3p3u3p5}. Further by conjugating with~$g$, we can assume that~$g$ is the identity. We take an element~$\beta_1$ of~$\phi_1(o_{E_D})$ whose residue class generates the residue field extension and an element 
~$\beta_2$ of~$\phi_2(o_{E_D})$ which has the same residue class as~$\beta_1$ in~$\mf{a}_{\Lambda,0}/\mf{a}_{\Lambda,1}$ and the same minimal polynomial over~$F$ as~$\beta_1$.
 The simple strata~$[\Lambda,e(\Lambda|F),e(\Lambda|F)-1,\beta_i\pi_F^{-1}]$, $i=1,2$, are equivalent and by Lemma~\ref{lemEquivSimpleGiveSameEmbeddingType} they have the same embedding type. 
 We therefore find an element~$t$ of the normalizer of~$\Lambda$ which conjugates~$\beta_1$ to~$\beta_2$, see Proposition \ref{propbroussousSecherreStevens3p3u3p5}. By Proposition~\ref{propIntertwiningOfSemisimpleStratumOverD} the element~$t$ can be written as a product of an element of~$P_1(\Lambda)$ and of~$C_A(\beta_1)$. Thus, we can choose~$t$ to be an element of~$P_1(\Lambda)$.
\end{proof}
%
We can now give an equivalent description of equivalence of embeddings. 

\begin{proposition}\label{propequivalentToSameEmbedType}
 Two embeddings~$(E,\Lambda)$ and~$(E',\Lambda')$ in~$A$ are equivalent if and only if there is a~$D$-automorphism~$g$ of~$V$ which maps~$\Lambda$ to a translate of 
~$\Lambda'$ such that the conjugation with~$g$ induces a field isomorphism between the residue fields of~$E_D$ and~$E'_D$. Given such a latter element~$g$ there is an element~$t$ of~$A^\times$ which 
 conjugates~$(E_D,\Lambda)$ to~$(E'_D,\Lambda')$ (up to translation of~$\Lambda'$) and such that~$gt^{-1}$ is an element of~$P_1(\Lambda')$. 
\end{proposition}

\begin{proof}
We can assume~$E=E_D$ and~$E'=E'_D$, in particular~$E$ and~$E'$ are isomorphic. We only have to prove the if-part.
 Applying~$g$ we can assume that both lattice sequences coincide and the residue fields coincide in~$\mf{a}_{\Lambda,0}/\mf{a}_{\Lambda,1}$. We take two injective~$F$-algebra homomorphisms~$\phi$ and~$\phi'$ from~$E$ into~$A$ such that the image of~$\phi$ is~$E$ (e.g.~$\phi=\id$) and  the image of~$\phi'$
 is equal to~$E'$, such that the maps coincide on the residue field of~$E$. Then we obtain from Proposition \ref{propFieldEmbedingsConjugateUnderaUnit} that both embeddings are equivalent and the conjugating element can be taken from~$P_1(\Lambda)$. 
\end{proof}

\subsection{Intertwining and Conjugacy for Semisimple strata}
In this subsection we prove an intertwining and conjugacy result for semisimple strata. This generalize the result for the split case, see \cite[Theorem 8.3]{skodlerackStevens:18}. For that we fix in the whole subsection two semisimple strata~$\Delta$ and~$\Delta'$ such that 
$n=n',\ r=r'$ and~$e(\Lambda|D)$ is equal to~$e(\Lambda'|D)$ and we assume that~$I(\Delta,\Delta')$ is non-empty.

%

The first important step is the statement of the existence of a matching for two intertwining semisimple strata over~$D$. 

\begin{proposition}[{{see \cite[Theorem 7.1]{skodlerackStevens:18} for the split case}}]\label{propMatchingGLDStrata}
 There is a unique bijection 
~$\zeta$ from the index set~$I$ for~$\beta$ to the index set~$I'$ for~$\beta'$ such that 
$\Delta_i\oplus\Delta'_{\zeta(i)}$ is equivalent to a simple stratum, for all~$i\in I$.
Furthermore the dimensions of~$V^i$ and~$V'^{\zeta(i)}$ coincide. 
\end{proposition}

We call the map~$\zeta$ a matching. 

\begin{proof}
If~$I(\Delta,\Delta')$ is non-empty then~$I(\Res_F(\Delta),\Res_F(\Delta'))$ is non-empty and by~\cite[Theorem 7.1]{skodlerackStevens:18}  there is a unique matching~$\zeta$ for~$\Res_F(\Delta)$ and~$\Res_F(\Delta')$. This is also a matching for~$\Delta$ and~$\Delta'$ which follows 
 from Proposition \ref{propEquivSemisimpleSplittingAndRestriction}. 
\end{proof}

\begin{corollary}\label{corEqParametersForIntertwiningSemisimpleStrata}
 Under the assumptions of this subsection we have 
 \begin{enumerate}
  \item\label{corEqParametersForIntertwiningSemisimpleStrata.i} $e(\Lambda|E)=e(\Lambda'|E')$ and~$k_0(\Delta)=k_0(\Delta')$.
  \item\label{corEqParametersForIntertwiningSemisimpleStrata.ii} $e(\Lambda^i|E_i)=e(\Lambda'^{\zeta(i)}|E'_{\zeta(i)})$ and~$e(E_i|F)=e(E'_{\zeta(i)}|F)$ and~$f(E_i|F)=f(E'_{\zeta(i)}|F)$ and~$k_0(\Delta_i)=k_0(\Delta'_{\zeta(i)})$ for all~$i\in I$.
 \end{enumerate}
\end{corollary}

\begin{proof}
 \ref{corEqParametersForIntertwiningSemisimpleStrata.ii} follows directly from Proposition~\ref{propMatchingGLDStrata}, 
 Corollary~\ref{corEquivalentCriteriaForIntertwining}(\ref{corEquivalentCriteriaForIntertwiningAss1}$\Rightarrow$\ref{corEquivalentCriteriaForIntertwiningAss2}) and 
 Corollary~\ref{corEndoEquivalenceForSimpleStrata}\ref{corEndoEquivalenceForSimpleStrataAss.iii}\ref{corEndoEquivalenceForSimpleStrataAss.iv}.
 We now prove~\ref{corEqParametersForIntertwiningSemisimpleStrata.i}. The equality~$e(\Lambda|E)=e(\Lambda'|E')$ follows directly from the second assertion. Thus we are left with the equation for the critical exponents. Assume~$k_0(\Delta)>k_0(\Delta')$ (``$\Delta'$ is semisimpler'').
 Take a positive integer~$j$ such that $r+j=-k_0(\Delta)$. Then~$\Delta(j+)$ is not semisimple but~$\Delta'(j+)$ is still semisimple and both strata are equivalent. 
 Moreover the stratum~$\Delta_i(j+)$ is still simple, because~$k_0(\Delta_i)$ is equal to~$k_0(\Delta'_{\zeta(i)})$ by the second assertion. Thus for~$\Delta(j+)$ not being semisimple there exists two indexes~$i_1,i_2\in I$
 such that~$\Delta_{i_1}(j+)\oplus\Delta_{i_2}(j+)$ is equivalent to a simple stratum. But then~$\Delta'_{\zeta(i_1)}(j+)\oplus\Delta'_{\zeta(i_2)}(j+)$, which intertwines
 with~$\Delta_{i_1}(j+)\oplus\Delta_{i_2}(j+)$, by Theorem~\ref{thmDiagonalization}, must intertwine with a simple stratum. This leads to a contradiction to Proposition~\ref{propMatchingGLDStrata}, 
 because~$\Delta'_{\zeta(i_1)}(j+)\oplus\Delta'_{\zeta(i_2)}(j+)$ has two blocks and the intertwining simple stratum only one block.  
\end{proof}

\begin{definition}\label{defGroupLevel}
 We define the~\emph{ group level} of a non-null semi-pure stratum~$\Delta$ as~$\lfloor\frac{r}{e(\Lambda|E)}\rfloor$,
 and we put the group level of a null-stratum to be infinity. 
 Caution: Equivalent semi-pure strata can have different group levels.
 The~\emph{degree} of a semi-pure stratum~$\Delta$ is defined as~$\dim_FE$.
\end{definition}
Corollary~\ref{corEqParametersForIntertwiningSemisimpleStrata} implies the first part of:  

 \begin{corollary}\label{corSameGrouplevelFromIntertwiningIfSamePeriodandSamer}
 \begin{enumerate}
 \item\label{corSameGrouplevelFromIntertwiningIfSamePeriodandSamerAssi} Granted~$e(\Lambda|F)=e(\Lambda'|F)$ and~$r=r'$, then two intertwining semisimple strata~$\Delta$ and~$\Delta'$ have the same group level if they intertwine.  
 \item\label{corSameGrouplevelFromIntertwiningIfSamePeriodandSamerAssii}  Suppose~$\Delta$ is non-null and~$e(\Lambda|E)$ does not divide~$r+1$ , i.e.~$\Delta(1+)$ and~$\Delta$ have the same group level. Then~$\Delta(1+)$ is semisimple. 
 \end{enumerate}
 \end{corollary}

\begin{proof}
 We only need to prove the second assertion. The critical exponent of a 
 non-null simple block of~$\Delta$ is a multiple of~$e(\Lambda|E)$ by Corollary~\ref{corEndoEquivalenceForSimpleStrata}\ref{corEndoEquivalenceForSimpleStrataAss.iii} and therefore all non-null pure blocks of~$\Delta$ are still simple. Thus if~$\Delta(1+)$ is not semisimple then there are two different simple blocks~$\Delta_{1}$ and~$\Delta_{2}$ of~$\Delta$ such that~$(\Delta_{1}\oplus\Delta_{2})(1+)$ is equivalent to a simple stratum~$\Delta''$ split by~$V^{1}\oplus V^{2}$, see Theorem~\ref{thmDiagonalization}. Then~$\Res_F(\Delta_{i}^\dag)$ is conjugate to~$\Res_F(\Delta''(-1)|_{V^i}^\dag),\ i=1,2,$ by~\cite[(1.9) and (1.8)]{bushnellHenniart:96}.
 Therefore~$\Delta_{1}\oplus\Delta_2$ is equivalent to a simple stratum by Corollary~\ref{corEquivalentCriteriaForIntertwining}(\ref{corEquivalentCriteriaForIntertwiningAss1}$\Rightarrow$\ref{corEquivalentCriteriaForIntertwiningAss2}). A contradiction. 
\end{proof}

We need to refine the intertwining and conjugacy theorem for simple strata, Proposition~\ref{propINtImplConSimpleD}, more precisely we need to control the valuation of the conjugating element to generalize Proposition~\ref{propINtImplConSimpleD} to intertwining semisimple strata. The solution is motivated by Proposition~\ref{propequivalentToSameEmbedType}. 
For a product of field extensions~$E=\prod_iE_i$ of~$F$ we define the residue-algebra to be the product of the residue fields 
\[\kappa_E:=\prod_i\kappa_{E_i}.\]

\begin{proposition}[{{see~\cite[(2.4.12)]{bushnellKutzko:93} for the strict simple split case}}]\label{propequivalentstratahaveequalResiduefields}
 Suppose that~$\Delta$ and~$\Delta'$ are equivalent and that~$\Lambda$ is equal to~$\Lambda'$. Then the residue algebras of~$F[\beta]$ and~$F[\beta']$ coincide in 
~$\mf{a}_{\Lambda,0}/\mf{a}_{\Lambda,1}$. 
\end{proposition}


 \begin{proof}
If~$\Delta$ and~$\Delta'$ are equivalent then~$\Res_F(\Delta)^\dag$ and~$\Res_F(\Delta')^\dag$ are equivalent, and we can therefore restrict to the strict split case, i.e.~$\Lambda$ is a lattice chain and~$D=F$.  
 By Corollary~\ref{corDecompAreConjForEqSemisimpleStrata} there is an element of~$1+\mf{m}_{-(k_0+m)}$ which conjugates the splitting of~$\beta$ to the one of~$\beta'$, and we are reduced to the simple case which is done 
 in \cite[(2.4.12)(i)]{bushnellKutzko:93}
 .
\end{proof}

\begin{lemma}\label{lemMatchingOfResidueAlg}
The conjugation by~$g\in I(\Delta,\Delta')$ induces an isomorphism form the residue algebra of~$E$ to the one of~$E'$, which  does not depend on the choice of the intertwining element. 
\end{lemma}

\begin{proof}
Let us at first define the map between the residue algebras: Take~$g\in I(\Delta,\Delta')$.
The residue algebra~$\kappa_E$ is canonically embedded into~$\mf{a}_0/\mf{a}_1$ and analogously for~$\kappa_{E'}$.
Conjugation by~$g$ maps~$\mf{a}_0/\mf{a}_1$ into~$(g\mf{a}_0g^{-1}+\mf{a}'_0)/(g\mf{a}_1g^{-1}+\mf{a}'_1)$.
There is a canonical map from~$\mf{a}'_0/\mf{a}'_1$ into~$(g\mf{a}_0g^{-1}+\mf{a}'_0)/(g\mf{a}_1g^{-1}+\mf{a}'_1)$.
We compose the inverse of the embedding of~$\kappa_{E'}$ with the~$g$-conjugation on~$\kappa_E$. We have to prove
\begin{enumerate}
\item \label{lemMatchingOfResidueAlg.1} for the existence that the~$g$-conjugate of~$\kappa_E$ is the image of~$\kappa_{E'}$ in~$(g\mf{a}_0g^{-1}+\mf{a}'_0)/(g\mf{a}_1g^{-1}+\mf{a}'_1)$, and
\item \label{lemMatchingOfResidueAlg.2} the independence of the map of the choice of the intertwining element~$g$. 
\end{enumerate}
At first we simplify the situation. Proposition \ref{propMatchingGLDStrata} allows to reduce to the case that the block decomposition of~$\beta$ and~$\beta'$ are the same 
and where the matching map~$\zeta$ does not permute the blocks. We identify the index sets, so that~$\zeta~$ is the identity map. 
By Proposition~\ref{propequivalentstratahaveequalResiduefields} we can replace~$\Delta$ and~$\Delta'$ by equivalent strata, and thus we can assume that~$\beta$ 
and~$\beta'$ have the same minimal polynomial by Theorem~\ref{thmDiagonalization}.
Thus, after conjugation, we can assume without loss of generality that~$\beta$ is equal to~$\beta'$. 

Let us prove Assertion~\ref{lemMatchingOfResidueAlg.1}: Consider the description of~$I(\Delta,\Delta')$ in Proposition \ref{propIntertwiningOfSemisimpleStratumOverD}. Property~\ref{lemMatchingOfResidueAlg.1}. is true for
 intertwining elements~$t\in C_A(\beta) (1+\mf{m}_{-(m+k-0)})$, in particular,~$t\kappa_E t^{-1}$ is in the image of~$\mf{a}'_0/\mf{a}'_1$ 
in~$(t\mf{a}_0t^{-1}+\mf{a}'_0)/(t\mf{a}_1t^{-1}+\mf{a}'_1)$. Conjugation with elements of~$1+\mf{m}'_{-(m+k_0)}$ acts as the identity on~$\mf{a}'_0/\mf{a}'_1$, and therefore we have Assertion~\ref{lemMatchingOfResidueAlg.1} for all elements of~$I(\Delta,\Delta')$.

We now prove the uniqueness: By direct calculation all elements in~$C_A(\beta) (1+\mf{m}_{-(m+k_0)})$ induce the identity on the residue algebra~$\kappa_E$. Further conjugation with elements of~$1+\mf{m}'_{-(m+k_0)}$ fixes all elements of~$\mf{a}'_0/\mf{a}'_1$. 
Thus all intertwining elements induce the identity on~$\kappa_E$.  
\end{proof}

\begin{notation}
We denote the map of Lemma \ref{lemMatchingOfResidueAlg} by~$\bar{\zeta}$ and call it the matching of the residue algebras of the 
intertwining semisimple strata in question. And we write~$(\zeta,\bar{\zeta})$ and call it the~\emph{matching pair}. 
\end{notation}

In the next theorem we introduce a new condition which together with intertwining implies conjugacy up to equivalence for semisimple strata.  We write~$E_D$ for the product of the~$(E_i)_D$.

\begin{theorem}\label{thmIntConjSesiTiG}
Let $(\zeta,\bar{\zeta})$ be the matching pair of~$\Delta$ and~$\Delta'$. Suppose there is a~$D$-automorphism~$t$ of~$V$ 
which maps~$\Lambda$ to a translate of~$\Lambda'$ such that the conjugation by~$t$ induces~$\bar{\zeta}|_{\kappa_{E_D}}$ and~$tV^i$ is equal to~$V^{\zeta(i)}$
for all indexes~$i\in I$. 
Then there is an element~$g$ of~$ G$ such that~$g\Delta$ is equivalent to~$\Delta'$,~$gV^i=V^{\zeta(i)}$ and~$gt^{-1}$ is an element of~$P(\Lambda')$.
\end{theorem}

\begin{proof}
 Conjugating with~$t$ allows us to reduce to the case where~$\Lambda=\Lambda'$,~$V^i=V^{\zeta(i)}$ and~$\bar{\zeta}|_{\kappa_{E_D}}$ is induced by conjugation with~$1$. We are now reduced to the simple case, and let us therefore assume that both strata are simple. By Corollary~\ref{corEquivalentCriteriaForIntertwining} and Theorem~\ref{thmDiagonalization}
 and Proposition~\ref{propequivalentstratahaveequalResiduefields} we can assume that~$\beta$ and~$\beta'$ have the same minimal polynomial. 
 Now, take an element of~$P_1(\Lambda)$ which conjugates~$E_D$ to~$E'_D$, see Proposition \ref{propequivalentToSameEmbedType}, and we can restrict to the case that~$E_D$ is equal to~$E'_D$. 
 We want to apply Proposition \ref{propFieldEmbedingsConjugateUnderaUnit}. For that we consider two embeddings of~$E$ with~$\phi_1(\beta)=\beta$ and~$\phi_2(\beta)=\beta'$. There is an element of~$ G$ which 
 conjugates~$\beta$ to~$\beta'$,
  i.e. which is on~$E$ the map~$\phi_2$. This conjugation induces on residue fields the identity because~$\bar{\zeta}$ is the identity. Thus,~$\phi_2\circ \phi_1^{-1}$
  is on~$E_D$ the identity, and Proposition~\ref{propFieldEmbedingsConjugateUnderaUnit} (take in ibid.~$g=1$) provides an element of~$P(\Lambda)$ which conjugates~$\beta$ to~$\beta'$. This finishes the proof. 
\end{proof}

\subsection{Proof of the main theorem of strata induction}\label{secMainTheoremOfStrataInduction}

For the proof of Proposition~\ref{propDerived} we need a lot of preparation.

\begin{lemma}\label{lemStratumPreparebeta}
 Let~$\Delta$ be a stratum such that~$\Delta(1+)$ is equivalent to a semisimple stratum with entry~$\gamma$. Let~$1^j$ be the idempotents for the 
 associated decomposition of~$V$ for~$\gamma$. Then there is an element~$u$ of~$1+\mf{m}_{-(r+1+k_0(\gamma,\Lambda))}$ such that~$u.\Delta$ is equivalent to a stratum 
 which is split by the associated decomposition of~$V$ for~$\gamma$. 
\end{lemma}

For the proof we need the map~$a_\gamma: A\ra A$ defined via~$a_\gamma(x)=\gamma x-x\gamma$.

\begin{proof}
 We have the decomposition~$A=\oplus_{j j'}A^{j j'}$ where~$A^{j j'}=1^jA1^{j'}$. We use this notation also for elements of~$A$. 
 By~\cite[Lemma 6.21]{skodlerackStevens:18} we have for non-equal indexes~$j$ and~$j'$ that the restriction of~$a_\gamma$ to~$A^{j j'}$ is a bijection and 
 that for all integers~$s\geq k_0(\gamma,\Lambda)$ the pre-image of~$\mf{a}_s$  under~$a_\gamma$ is equal to~$\mf{n}_s(\gamma,\Lambda)$. In particular 
 for~$j\neq j'$ there is an element~$a^{j j'}\in \mf{n}_{-(r+1)}(\gamma,\Lambda)^{j j'}$ such that~$a_\gamma(a^{j j'})=\beta^{j j'}$. Thus for~$x=\sum_{j\neq j'}a^{j j'}$
 we have 
 \[\beta=\sum_j \beta^{j j}+a_\gamma(x).\]
 The last equation implies that~$\sum_j\beta^{j j}$ is congruent to~$(1+x)\beta (1+x)^{-1}$  modulo~$\mf{a}_{-r}$, see~\cite[Proposition 7.6]{skodlerackStevens:18} for the calculations.
 This finishes the proof. 
\end{proof}

For a non-null semisimple stratum~$\Delta$ and a tame corestriction~$s$  on~$A$ relative to~$F[\beta]|F$ we have the following sequence 
\begin{equation}\label{eqExactSeqDelta}
\mf{n}_{-r}\cap\mf{a}_{-r-k_0}\stackrel{\a_\beta}{\ra}  \mf{a}_{-r}/\mf{a}_{-r+1} \stackrel{s}{\ra}\ (\mf{a}_{-r}\cap C_A(\beta))/(\mf{a}_{-r+1}\cap C_A(\beta)) 
\ra 0
\end{equation}

\begin{lemma}[{see~\cite[Proposition 7.6]{skodlerackStevens:18} together with Lemma~\ref{lemStratumPreparebeta} for the split case}]\label{lemExSeqDelta}
 The sequence (\ref{eqExactSeqDelta}) is exact. 
\end{lemma}

\begin{proof}
We denote the sequence (\ref{eqExactSeqDelta}) as $\Seq(\Delta)$. The sequence~$\Seq(\Delta\otimes L)$ is exact by the split case. $\Seq(\Delta)$ is obtained
from~$\Seq(\Delta\otimes L)$ by taking the~$\Gal(L|F)$-fixed points. Hilbert 90 for the trace states that the cohomology group $H^1(\Gal(L|F),\kappa_L)$ is trivial, 
which forces the exactness of~$\Seq(\Delta)$.
\end{proof}

Let~$\Lambda$ be a lattice sequence and~$1=\sum_j 1^j$ a decomposition into pairwise orthogonal idempotents~$1^j\in \mf{a}_0$. Given an integer~$r$, we say that an 
element~$a\in\mf{a}_{-r}$ is~\emph{split via~$1=\sum_j 1^j$ modulo~$\mf{a}_{1-r}$} if~$a$ is congruent to~$\sum_j1^ja1^j$ modulo~$\mf{a}_{1-r}$. 

\begin{lemma}\label{lemSimpletameCorConjSplitting}
 Let~$\Delta$ be a simple stratum with a corestriction~$s_\beta$,~$1=\sum_j1^j$ a decomposition in~$\mf{a}_0\cap C_A(\beta) $ and let~$a$ be an element of~$\mf{a}_{-r}$ 
 such that~$s_\beta(a)$ is split by~$1=\sum_j1^j$ modulo~$C_A(\beta)\cap\mf{a}_{-r+1}$. Then there is an element~$u=1+x$ of~$1+\mf{m}_{-(r+k_0)}$ such that~$u(\beta+a)u^{-1}$ 
 is split by the~$1^j$ modulo~$\mf{a}_{-r+1}$.
\end{lemma}

\begin{proof}
 This follows directly from the exactness of the sequence~(\ref{eqExactSeqDelta}), in the following way: $\ti{a}=\sum_j1^ja1^j$ and~$a$ have the same image under~$s_\beta$ in~\eqref{eqExactSeqDelta}, and thus by the exactness of sequence~(\ref{eqExactSeqDelta}) there is an element~$x$ in~$\mf{n}_{-r}\cap \mf{a}_{-(r+k_0)}$ such that~$a_\beta(x)$ is congruent 
 to~$a-\tilde{a}$ modulo~$\mf{a}_{-r+1}$ which is equivalent to the conclusion of the lemma. 
\end{proof}

\begin{lemma}\label{lemTamcorestrictionForthAndBack}
 Let~$\Delta$ be a semisimple stratum and let~$s$ be a tame corestriction on~$A$ relative to~$F[\beta]|F$. 
Given $a,a'\in\mf{a}_{-r}$, the element $s(a)$ is congruent  to~$s(a')$ modulo~$\mf{a}_{-r+1}$ if and only if there is an element~$u\in 1+\mf{m}_{-r-k_0}$ such that~$u(\beta+a)u^{-1}$ is congruent to~$\beta+a'$ modulo~$\mf{a}_{-r+1}$.       
\end{lemma}

\begin{proof}
 This follows directly of the exactness of the sequence (\ref{eqExactSeqDelta}) and standard calculations.
\end{proof}

\begin{proof}[Proof of~\ref{propDerived}]
\textbf{Part 1: \ref{propDerived}\ref{propDerivedAss1} implies \ref{propDerived}\ref{propDerivedAss2}:} Here we can assume that~$\Delta$ is a semisimple stratum.
Then there is a semisimple stratum equivalent to~$\Delta(1+)$ which is split by the splitting of~$\beta$ and 
by Corollary~\ref{corDecompAreConjForEqSemisimpleStrata} and Lemma~\ref{lemTamcorestrictionForthAndBack} we  can assume that~$\gamma$ is split by the splitting of~$\beta$. By Lemma~\ref{propDerivedSimple} each~$\partial_{\gamma_j,s_j}(\Delta_j)$ is equivalent to a semisimple strata.\\
\textbf{Part 2: \ref{propDerived}\ref{propDerivedAss2} implies \ref{propDerived}\ref{propDerivedAss1}:} The stratum~$\Delta(1+)$ is equivalent to a semisimple stratum~$\tilde{\Delta}$ with splitting~$1=\sum_j\ti{1}^j$ and~$\tilde{\beta}=\gamma$. By Lemma~\ref{lemStratumPreparebeta} the stratum~$\Delta$ 
is conjugate  to a stratum split by~$(\ti{1}^j)_j$ modulo~$\mf{a}_{-r}$ by an element of~$1+\mf{m}_{-r-1-\tilde{k}_0}$. This conjugation does not change the equivalence class of the derived strata 
by Lemma~\ref{lemTamcorestrictionForthAndBack}. So, we can assume without loss of generality that~$\Delta$ is split by~$(\ti{1}^j)_j$. Thus we can restrict 
to the case where~$F[\gamma]$ is a field. We assume by~\ref{propDerivedAss2} that~$\partial_{\gamma,s}(\Delta)$
is equivalent to a semisimple stratum which has again its own associated splitting say~$1=\sum_i1^i$, and thus~$\Delta$ is conjugate to a stratum 
split by~$(1^i)_i$ modulo~$\mf{a}_{-r}$ by an element of~$1+\mf{m}_{-r-1-\tilde{k}_0}$, by Lemma~\ref{lemSimpletameCorConjSplitting}, which allows us to assume that~$\partial_{\gamma,s}(\Delta)$ is equivalent to a simple stratum. Lemma~\ref{propDerivedSimple} states in this case that~$\Delta$ is equivalent to a simple stratum. 
\end{proof}

As a corollary we obtain the subtle modification of Proposition~\ref{propEquivSemisimpleSplittingAndRestriction}.

\begin{proposition}\label{propEquivToSemisimpleStratum}
 Let~$\Delta$ be a stratum and let~$\ti{F}|L$ be a finite unramified field extension. Then the following assertions are equivalent:
 \begin{enumerate}
  \item The stratum~$\Delta$ is equivalent to a semisimple stratum.\label{thmEquivToSemisimpleStratumAss1}
  \item The stratum~$\Delta\otimes \ti{F}$ is equivalent to a semisimple stratum.\label{thmEquivToSemisimpleStratumAss2}
  \item The stratum~$\Res_F(\Delta\otimes \ti{F})$ is equivalent to a semisimple stratum.\label{thmEquivToSemisimpleStratumAss3}
  \item The stratum~$\Res_F(\Delta)$ is equivalent to a semisimple stratum.\label{thmEquivToSemisimpleStratumAss4} 
 \end{enumerate}
\end{proposition}

We need one lemma:

\begin{lemma}\label{lemSummandsOfEquivToSemiAreEquivToSemi}
 Given two strata~$\Delta$ and~$\Delta'$ with~$r=r'$ and~$e(\Lambda)=e(\Lambda')$ and suppose that~$\Delta\oplus\Delta'$ is equivalent to a semisimple stratum.  Then~$\Delta$ is equivalent to a semisimple stratum.
\end{lemma}

\begin{proof}
Without loss of generality we can assume~$n\geq n'$. We prove by induction on~$n-r$ that~$\Delta$ and~$\Delta'$ are equivalent to semisimple strata. The case~$n=r$:
~$\Delta$ and~$\Delta'$ are equivalent to a null-stratum. The case~$n=r+1$ follows directly from Proposition~\ref{propCriteriaFundtoSemisimplStratum}. 
The case~$n>r+1$:~$\Delta\oplus\Delta'$ is equivalent to a semisimple stratum implying that by induction 
hypothesis~$\Delta(1+)$ and~$\Delta'(1+)$  are equivalent to semisimple strata~$\tilde{\Delta}$ and~$\tilde{\Delta'}$ with entries~$\gamma$ and~$\gamma'$, respectively.
By Theorem~\ref{thmDiagonalization} we can choose~$\tilde{\Delta}$ and~$\tilde{\Delta'}$ such that~$\ti{\Delta}\oplus\ti{\Delta}'$ is semisimple and split by the decomposition~$(\bigoplus_{j\in J}V^j)\oplus (\bigoplus_{j'\in J'}V'^{j'})$ where~$\bigoplus_{j\in J}V^j$ is the splitting of~$\gamma$  and~$\bigoplus_{j'\in J'}V'^{j'}$ is the splitting of~$\gamma'$. 
Let~$V\oplus V'=\bigoplus_{k\in K} (V\oplus V')^k$ be the splitting of~$\gamma+\gamma'$. We write~$V^k$ and~$V'^k$ for the intersections of~$(V\oplus V')^k$ with~$V$ and~$V'$, respectively. 

We fix a corestriction~$t$ on~$\End_D(V\oplus V')$ relative to~$F[\gamma+\gamma']|F$
\[t:\End_D(V+V')\rightarrow \End_{F[\gamma+\gamma']\otimes_FD}(V\oplus V').\]
Then the restrictions
\[s:\End_D(V)\rightarrow \End_{F[\gamma]\otimes_FD}(V).\]
and 
\[s':\End_D(V')\rightarrow \End_{F[\gamma']\otimes_FD}(V')\]
are tame corestrictions. 
By Proposition~\ref{propDerived} it is enough to prove that for all~$j\in J$ the derived stratum~$\partial_{\gamma_j,s_j}(\Delta_j)$ is equivalent to a semisimple stratum. For~$j\in J$ there exists a~$k\in K$ such that~$V^k=V^j$. If~$V'^k=0$ then we have
\[\partial_{\gamma_j,s_{j}}(\Delta_j)=\partial_{(\gamma+\gamma')_k,t_{k}}((\Delta\oplus\Delta')_k).\]
If~$V'^k\neq 0$ then
\[\partial_{\gamma_j,s_{j}}(\Delta_j)\oplus\partial_{\gamma'_{j'},s'_{j'}}(\Delta'_{j'})=\partial_{(\gamma+\gamma')_k,s_{k}}((\Delta\oplus\Delta')_k)\]
for some~$j'\in J'$. The derived stratum~$\partial_{(\gamma+\gamma')_k,s_{k}}((\Delta\oplus\Delta')_k)$ is equivalent to a semisimple stratum, by Proposition~\ref{propDerived}, because~$\Delta\oplus\Delta'$
is equivalent to a semisimple stratum. 
The base case implies that each~$\partial_{\gamma_j,s_{j}}(\Delta_j)$ is equivalent to a semisimple stratum. 
\end{proof}

\begin{proof}[Proof of Proposition~\ref{propEquivToSemisimpleStratum}]
 We have~\ref{thmEquivToSemisimpleStratumAss1}$\Rightarrow$\ref{thmEquivToSemisimpleStratumAss2}$\Rightarrow$\ref{thmEquivToSemisimpleStratumAss3} by 
 Proposition~\ref{propEquivSemisimpleSplittingAndRestriction}. The implication \ref{thmEquivToSemisimpleStratumAss3}$\Rightarrow$\ref{thmEquivToSemisimpleStratumAss4}
 follows directly from Lemma~\ref{lemSummandsOfEquivToSemiAreEquivToSemi}, because~$\Res_F(\Delta\otimes \ti{F})$ is equivalent to a direct sum of~$[\tilde{F}:L]$ many copies 
 of~$\Res_F(\Delta)$. We prove \ref{thmEquivToSemisimpleStratumAss4}$\Rightarrow$\ref{thmEquivToSemisimpleStratumAss1} by induction on~$n-r$. Case~$n=r$ is trivial and 
 Case~$n=r+1$ follows from Corollary~\ref{corEquivToMinSemisimpleStratum}. The stratum~$\Res_F(\Delta)$ is equivalent to a semisimple stratum and therefore~$\Res_F(\Delta(1+))$
 which is~$\Res_F(\Delta)(1+)$ is equivalent to a semisimple stratum too. Then~$\Delta(1+)$ is equivalent to a semisimple stratum~$\tilde{\Delta}$, 
 denote~$\gamma:=\tilde{\beta}$, and, by Lemma~\ref{lemStratumPreparebeta}, we can conjugate to the case where~$\Delta$ is equivalent to a stratum which is split by the 
 splitting of~$\gamma$, i.e. we can assume~$\tilde{\Delta}$ to be simple without loss of generality. Let~$s$ be a tame corestriction on~$A$ relative to~$F[\gamma]|F$. The stratum~$\partial_{\gamma,s\otimes\id_D}(\Res_F(\Delta))$ is equivalent to a 
 semisimple stratum by Proposition~\ref{propDerived} and it is equivalent to a direct sum of translates of copies of~$\Res_{F[\gamma]}(\partial_{\gamma,s}(\Delta))$ by 
 formula~\cite[part II, Lemma 3.1]{broussousLemaire:02}. Thus~$\Delta$ is equivalent to a semisimple stratum by Lemma~\ref{lemSummandsOfEquivToSemiAreEquivToSemi}, the base case and Proposition~\ref{propDerived}.
\end{proof}

In the next section we will need the defining sequence of a stratum:
\begin{definition}\label{defDefiningSequence}
 \begin{enumerate}
 \item A~\emph{defining sequence} for~$\Delta$ is a sequence of semisimple strata~$\Delta(j)=[\Lambda,n,r+j,\beta(j)]$, equivalent to~$\Delta(j+)$,
 ~$j=0,\ldots,n-r$, such that,~$\beta=\beta(0)$ and~$\Delta(j+1)$ is split by the associated splitting of~$\Delta(j)$. 
 Such a sequence always exists by coarsening and Theorem~\ref{thmDiagonalization}
 . And it is not unique. Thus if we write~$\Delta(j)$ there is 
 always a fixed choice behind it. 
 \item We call a defining sequence~\emph{$k_0$-controlled} if~$\beta(j)$ is equal to~$\beta(j+1)$ for the 
 indexes~$j$ satisfying~$k_0(\Delta(j))=k_0(\Delta(j+1))$.
 
 \item\label{defDefiningSequenceiii} Given a defining sequence of~$\Delta$ the strictly increasing finite sequence of integers $(r_l)=(r+j_l)_{l=0,..,s}$
 such that
 \begin{itemize}
  \item $r_0=r$ and~$r_s=n$ and
  \item ~$k_0(\Delta(j_l-1))>k_0(\Delta(j_l))$ for~$l=1,\ldots,s$
 \end{itemize} is
 called the~\emph{jump sequence} of~$\Delta$ and the integers~$r_l$,~$l=0,\ldots,s$ are called~\emph{jumps} of~$\Delta$.
 \item A~\emph{core approximation} of~$\Delta$ is a stratum~$\Delta(j)$ such that~$r+j$ is the smallest jump of~$\Delta$ which satisfies
 ~$r<\lfloor \frac{-k_0(\Delta(j))}{2}\rfloor$. Cf.~\cite[(3.5.4) and (3.5.5)]{bushnellKutzko:93}. 
 \end{enumerate}
\end{definition}

\section{Semisimple characters}\label{secSemisimpleCharacters}
In this section we generalize semisimple characters for~$\End_F(V)$, see~\cite{skodlerackStevens:18}, and simple characters for $\End_D(V)$, see~\cite{secherreI:04},
to semisimple characters for~$\End_D(V)$, and generalize the properties of semisimple strata from the last section to semisimple characters.
Initially the concept of these kind of characters is introduced in~\cite{bushnellKutzko:93} where they only considered the strict split simple case, but the generalization is
straightforward.

\subsection{Definitions}
Here we define semisimple characters for~$A=\End_D(V)$ as restrictions of semisimple characters from $\End_L(V)$ to~$A$. 
\begin{remark}
The following 
constructions depend on the choice of~$\psi_F$, see~\S\ref{secNot}. 
\end{remark}
Let~$\Delta$ be a semisimple stratum. We are going to attach to~$\Delta$ a compact open subgroup~$H(\Delta)$ of~$G$ and a set~$C(\Delta)$ of certain complex characters of~$H(\Delta)$. 

\subsubsection{Split case}
Let us recall the constructions of~$\C(\Delta)$ for the case~$\mathbf{D=F}$, see~\cite[\S 9]{skodlerackStevens:18} for details. 

\begin{definition}[{\cite{skodlerackStevens:18}~Definition~9.5}]\label{defSemisimpleCharSpliCase}
 Let~$\Delta$ be semisimple stratum. The definition of~$\C(\Delta)$ goes along an induction on~$k_0(\beta,\Lambda)$.
 \begin{enumerate}
  \item At first we define the relevant groups. If~$\Delta$ is a null-stratum then 
 one defines~$H(\beta,\Lambda)$ and~$J(\beta,\Lambda)$ as~$P(\Lambda)$. If~$\Delta$ is not a null stratum then we take a semisimple 
 stratum~$[\Lambda,n,-k_0(\beta,\Lambda),\gamma]$ which is equivalent to~$[\Lambda,n,-k_0(\beta,\Lambda),\beta]$ and one 
 defines~$H(\beta,\Lambda)$ as $P(j_E(\Lambda))(H(\gamma,\Lambda)\cap P_{\lfloor\frac{-k_0(\beta,\Lambda)}{2}\rfloor+1}(\Lambda))$ and~$J(\beta,\Lambda)$ as
 $P(j_E(\Lambda))(H(\gamma,\Lambda)\cap P_{\lfloor\frac{-k_0(\beta,\Lambda)+1}{2}\rfloor}(\Lambda))$.
 For non-negative integers~$i$ the intersections of~$H(\beta,\Lambda)$ with~$P_i(\Lambda)$ is denoted by~$H^{i}(\beta,\Lambda)$, and similar is~$J^i(\beta,\Lambda)$ 
 defined. 
 \item Here we define the set~$\C(\Delta)$, also denoted by~$\C(\Lambda,r,\beta)$, of semisimple characters for~$\Delta$. They are defined on the 
 group~$H^{r+1}(\beta,\Lambda)$ which we denote as~$H(\Delta)$. If~$\Delta$ is a null-stratum then~$\C(\Delta)$ is the singleton consisting of the trivial character 
 on~$H(\Delta)$. If~$\Delta$ is not a null stratum we take~$[\Lambda,n,-k_0(\beta,\Lambda),\gamma]$ as above. Then~$\C(\Delta)$ is defined to be the set of those~$complex$ characters~$\theta$ 
 which  satisfy the following properties:
 \begin{enumerate}
 \item~$\theta$ is normalized by~$\mf{n}(j_E(\Lambda))$.
 \item The restriction of~$\theta$ to~$P_{r+1}(j_E(\Lambda))$ factorizes through the determinant~$\det:\ C_{A^\times}(\beta)\ra E^\times$. 
 \item Put~$r_0$ to be the maximum of~$r$ and~$\lfloor\frac{-k_0(\beta,\Lambda)}{2}\rfloor$. The restriction of~$\theta$ to~$H^{1+r_0}(\gamma,\Lambda)$ coincides 
 with~$\psi_{\beta-\gamma}\theta_0$ for some element~$\theta_0$ of~$\C(\Lambda,r_0,\gamma)$. (See~\S\ref{secNot} for the definition of~$\psi_{\beta-\gamma}$.)
 \end{enumerate}
 \end{enumerate}
 \end{definition}
 
\begin{remark}
 Stevens gave a  different definition of a semisimple character for the split case, see~\cite[Definition 3.13]{stevens:05}, which is equivalent to 
 Definition~\ref{defSemisimpleCharSpliCase} by~\cite[Remark 9.6, Proposition 9.7]{skodlerackStevens:18}. 
\end{remark}

It is not surprising but a subtle statement that~$\C(\Delta)$ only depends on the equivalence class of~$\Delta$.  Following the definition one obtains that 
for~$r\geq \lfloor \frac{n}{2}\rfloor+1$ the set~$\C(\Delta)$ is a singleton consisting of the element~$\psi_\beta$.
We will later refer to this definition to give the analogue for the non-split case. A priori one can formulate Definition~\ref{defSemisimpleCharSpliCase} for this case if we replace the determinants by reduced norms. We are going to define~$\C(\Delta)$ in the non-split case differently and show that this coincides with 
 the modification of Definition~\ref{defSemisimpleCharSpliCase}. 

\subsubsection{General case}
We now generalize~$\C(\Delta)$ to an \textbf{arbitrary $D$}: We fix an additive character~$\psi_L$ of~$L$ of level~$1$ whose restriction to~$F$
is equal to~$\psi_F$. 

\begin{definition}[{see~\cite{secherreI:04} for the simple case}]
\begin{enumerate}
\item We define~$H(\Delta)$,~$H^i(\beta,\Lambda)$ and~$J^i(\beta,\Lambda)$ as the intersection of~$H(\Delta\otimes L)$,~$H^i(\beta\otimes_F1,\Lambda)$ and~$J^i(\beta\otimes_F1,\Lambda)$ with~$ G$,~$i\geq 0$. 
 The set~$\C(\Delta)$ is defined to be the set of those characters of~$H(\Delta)$ which can be extended to an element of~$\C(\Delta\otimes L)$.
 We call the elements of~$\C(\Delta)$~\emph{ semsimple characters for~$\Delta$}.
 A semisimple character~$\theta$ (for some semisimple stratum) is called~\emph{simple} if there is a simple stratum~$\Delta$ such that~$\theta\in\C(\Delta)$. 
\item We need the Lie algebra of~$H(\Delta)$: The group~$H(\Delta)$ can be written as~$1+\mf{h}(\Delta)$ with a bi-$\mf{b}_{\Lambda,0}$-order~$\mf{h}(\Delta)$ in~$A$. We also write~$\mf{h}^{j}(\beta,\Lambda)$ and~$\mf{j}^{j}(\beta,\Lambda)$ for~$H^j(\beta,\Lambda)-1$ and~$J^j(\beta,\Lambda)-1$ , for positive~$j$. 
\end{enumerate}
 \end{definition}

\begin{remark}\label{remDefSemisimpleCharacter}
  The set~$\C(\Delta)$ only depends on the equivalence class of~$\Delta$. 
\end{remark}

Given a semisimple stratum~$\Delta$ then, because of Remark~\ref{remDefSemisimpleCharacter}, we put~$\C(\Delta'):=\C(\Delta)$ for every not necessarily semisimple stratum~$\Delta'$ which is equivalent
to~$\Delta$. Analogously we use notation~$H(\Delta')$. For later  induction purposes we introduce the notation
$\theta(j+)$ to be the restriction of~$\theta$ into~$\C(\Delta(j+))$ for all non-negative integers~$j$ smaller than~$n-r+1$.

\begin{proposition}\label{propDefSemisimpleCharInductiveNonSplit}
  Let~$\Delta$ be a semisimple stratum then the inductive definition, see~\ref{defSemisimpleCharSpliCase}, of~$\C(\Delta\otimes L)$ restricts canonically to an inductive 
  definition of~$\C(\Delta)$. The same is true for the groups~$H(\beta,\Lambda)$ and~$J(\beta,\Lambda)$.  In particular,~$\C(\Delta)$ does not depend on the choice of
  the extension~$\psi_L$ of~$\psi_F$.
\end{proposition}

\begin{proof}
 We have to prove two assertions:
 \begin{enumerate}
  \item \label{propDefSemisimpleCharInductiveNonSplit.i} The assertions for~$H(\beta,\Lambda)$ and~$J(\beta,\Lambda)$, and 
  \item \label{propDefSemisimpleCharInductiveNonSplit.ii} That every character given by the inductive definition extends to a semisimple character in~$\C(\Delta\otimes L)$. 
 \end{enumerate}
 (The remaining assertions for~$\C(\Delta)$ follow directly by restriction.) 
 We start with~\ref{propDefSemisimpleCharInductiveNonSplit.i}. 
 We only prove the assertion for~$H(\beta,\Lambda)$. The proof for~$J(\beta,\Lambda)$ is similar. We consider the inductive definition of~$H(\beta\otimes 1,\Lambda)$ and 
 take the~$\Gal(L|F)$-fixed points. The only subtlety is the equality
 \[(P(j_{F[\beta\otimes 1]}(\Lambda))H^{\lfloor\frac{-k_0}{2}\rfloor+1}(\gamma\otimes 1,\Lambda))^{\Gal(L|F)}
 =P(j_{F[\beta\otimes 1]}(\Lambda))^{\Gal(L|F)}H^{\lfloor\frac{-k_0}{2}\rfloor+1}(\gamma\otimes 1,\Lambda)^{\Gal(L|F)}.\]
  This equality follows from the next lemma. 

  We now prove~\ref{propDefSemisimpleCharInductiveNonSplit.ii} by induction on the critical exponent. 
  For~$k_0=-\infty$ the trivial character extends to a trivial character. Suppose for the induction step that~$k_0\geq -n$. Suppose at first 
  that~$r\geq\lfloor \frac{-k_0}{2}\rfloor$. In this case we take a semisimple stratum~$\Delta_\gamma$ in a defining sequence for~$\Delta$ with
  $r_\gamma=-k_0$, and an inductively defined character~$\theta$ with respect to~$\Delta$ can be extended into~$\C(\Delta\otimes L)$ because 
  we can extend~$\psi_{\gamma-\beta}\theta$ by induction hypothesis. The case of~$r<\lfloor \frac{-k_0}{2}\rfloor$ is proven by induction on~$r$, using the
  case~$r=\lfloor \frac{-k_0}{2}\rfloor$ as the base case. Let~$\theta$ be an inductively defined character with respect to~$\Delta$. Then~$\theta(1+)$ is extendible to a 
  semisimple character~$\theta(1+)_L\in\C(\Delta(1+)\otimes L)$. For the sake of simplicity let us assume that~$\Delta$ is a simple stratum, i.e.~$E\otimes L$ 
  is a product of fields~$E_j$ which are over~$L$ isomorphic to~$E\ti{L}$ where~$\ti{L}|F$ is an unramified field extension isomorphic to~$L|F$ in an algebraic closure 
  of~$E$. The restriction of~$\theta(1+)_L$ to~$P_{r+2}(\Lambda_{E\otimes L})$ factorizes through the reduced norm with character~$\otimes_j\chi_j$
  on~$\prod_j P_{\lfloor \frac{r+1}{e(\Lambda^j|E_j)}\rfloor+1}(o_{E_j})$. Let further~$\chi$ be the character on~$P_{\lfloor \frac{r}{e(\Lambda|E)}\rfloor+1}(o_{E})$
  which comes from the restriction of~$\theta$ to~$P_{r+1}(j_E(\Lambda))$. Then we can extend~$\otimes_j\chi_j$ to~$\prod_j P_{\lfloor \frac{r}{e(\Lambda^j|E_j)}\rfloor+1}(o_{E_j})$
  say to~$\otimes_j\chi'_j$ (using extensions of the~$\chi_j$) such that for every element~$x$ of~$P_{\lfloor \frac{r}{e(\Lambda|E)}\rfloor+1}(o_{E})$ the product
  of the~$\chi'_j(1^jx)$ is equal to~$\chi(x)$, and we use~$\otimes_j\chi'_j$ to extend~$\theta(1+)_L$ to an element~$\theta_L$ of~$\C(\Delta\otimes L)$.  
\end{proof}

\begin{lemma}\label{lemCohom}
Let~$Q_1$ and~$Q_2$ be two subgroups of an ambient group~$Q$ such that~$Q_2$ normalizes~$Q_1$. Assume further that a group~$\Gamma$ is acting on~$Q$ preserving~$Q_1$ and~$Q_2$.
Then~$(Q_1Q_2)^\Gamma$ is equal to~$Q_1^\Gamma Q_2^\Gamma$ if the (non-abelian group) cohomology group~$H^1(\Gamma,Q_1\cap Q_2)$ is trivial. 
\end{lemma}

\begin{proof}
 The cohomology group being trivial is equivalent to the exactness of the sequence 
 \[1\ra  (Q_1\cap Q_2)^\Gamma \ra Q_2^\Gamma\ra (Q_2/(Q_1\cap Q_2))^\Gamma\ra 1.\]
 which implies the equation~$(Q_1Q_2)^\Gamma=Q_1^\Gamma Q_2^\Gamma$.
\end{proof}
Semisimple characters respect certain Iwahori decompositions. 
\begin{proposition}[\cite{stevens:08} Proposition~5.5]\label{propIwahoriSt05no5.5}
Let~$\bigoplus_{j\in J}V^j=V$ be a splitting of~$V$ such that the idempotents commute with~$\beta$. Let~$\tilde{M}$ be the Levi subgroup of~$G$ given as the stabilizer of the splitting and let~$\ti{U}_+$ and~$\ti{U}_-$ 
be the unipotent subgroups of~$G$ with respect to the splitting~$(V^j)_{j\in J}$ for a given total ordering on~$J$. Then~$H(\Delta)$ has an Iwahori decomposition with respect 
to~$\ti{U}_-\ti{M}\ti{U}_+$ and every element~$\theta$ of~$\C(\Delta)$ is trivial on the groups~$H(\Delta)\cap\ti{U}_+$ and~~$H(\Delta)\cap\ti{U}_-$.
\end{proposition}

Further, we have the following restriction maps
\[\C(\Delta)\ra \C(\Delta_i), \theta\mapsto \theta_i:=\theta|_{H^{r+1}(\beta_i,\Lambda^i)}\]
for all indexes~$i\in I$ and we call the~$\theta_i,\ i\in I$, the~\emph{block restrictions} of~$\theta$. 
Two semisimple characters~$\theta,\theta'\in\C(\Delta)$ coincide if block-wise the block restrictions coincide.



%

\subsection{Transfers, and Intertwining}
In this subsection we fix two semisimple strata~$\Delta$ and~$\Delta'$.
Take~$\theta\in\C(\Delta)$ and~$\theta'\in\C(\Delta')$. We denote by~$I(\theta,\theta')$ the set of elements~$g$ of~$A^\times$ which intertwine~$\theta$ with~$\theta'$, i.e.
$g.\theta$ and~$\theta'$ coincide on~$g.H(\Delta)\cap H(\Delta')$. We denote by~$\theta(1+)\in\C(\Delta(1+))$ the restriction of~$\theta$ to~$H(\Delta(1+))$. 
Let us recall that we denote by~$\Delta(1)$ a first member of an arbitrary fixed defining sequence for~$\Delta$ (There is a choice hidden.) The stratum~$\Delta(1)$ is 
equivalent to~$\Delta(1+)$ and~$\Delta(1)$ can be chosen to be~$\Delta(1+)$ if ~$\Delta(1+)$ is semisimple.

\begin{definition}\label{defTransfer}
Suppose that~$\Delta$ and~$\Delta'$ have the same group level and the same degree. We call~$\theta'$ a~\emph{transfer} of~$\theta$ from~$\Delta$ to~$\Delta'$ if~$I(\Delta,\Delta')$ 
is a non-empty subset of~$I(\theta,\theta')$. 
\end{definition}

\begin{notation}
 In the following of this subsection we assume that~$\Delta$ and~$\Delta'$ satisfy~$r=r'$ and~$e(\Lambda|F)=e(\Lambda'|F)$. We further assume that~$\Delta$ and~$\Delta'$ intertwine. The latter implies $n=n'$, see~\ref{corSameLevelSemisimpleStrata}, and that both strata have the same group level, see Corollary~\ref{corSameGrouplevelFromIntertwiningIfSamePeriodandSamer}\ref{corSameGrouplevelFromIntertwiningIfSamePeriodandSamerAssi}.
\end{notation}

\begin{remark}
 The following statements would all be true if we just would assume that~$\Delta$ and~$\Delta'$ have the same group level and the same degree, but many of the referred results are not 
 stated in terms of group levels and degree in the literature. 
\end{remark}

\begin{remark}\label{remDiagonalization}
 Granted~$I(\Delta,\Delta')\neq\emptyset$, then, by Theorem~\ref{thmDiagonalization}, there are semisimple strata~$\tilde{\Delta}$ and~$\tilde{\Delta}'$ 
 equivalent to~$\Delta$ and~$\Delta'$, respectively, and respecting the associated splittings, such that~$\tilde{\beta}$ and $\tilde{\beta}'$ are conjugate to each other. 
 Say, they are conjugate by a~$D$-automorphism~$u$ of~$V$. Then~$C_A(\tilde{\beta}')^\times u$ is contained in~$I(\Delta,\Delta')$. 
\end{remark}

\begin{remark}\label{remConstructionOfTransfer}
\begin{enumerate}
 \item There is a way to construct a transfer using the already understood split case~\cite[\S8.2 and Proposition A.10]{KSS:21}: Take~$\theta\in\C(\Delta)$ and an 
extension~$\theta_L\in\C(\Delta\otimes L)$. By the preceding remark we can suppose that~$\beta$ and~$\beta'$ are conjugate.
Let~$\theta'_L$ be the transfer of~$\theta_L$ from~$\Delta\otimes L$ to~$\Delta'\otimes L$. Then~$\theta'_L|_{H(\Delta')}$ is a transfer 
of~$\theta$ from~$\Delta$ to~$\Delta'$.
\item It is more subtle to show the existence in the general case of same group level and same degree. 
So let us here for this part of the remark relax to this case for~$\Delta$ and~$\Delta'$. Passing to affine translates we can assume without loss of generality that~$\Delta$ and~$\Delta'$ have the same period,~$n=n'$ and~$r'\geq r$. (The case~$r\geq r'$ is treated similarly.)
Then~$\Delta\otimes L((r'-r)+)$ is still semisimple by~\cite[Theorem 9.9(ii)(b)]{KSS:21} and by Remark~\ref{remDiagonalization} there are semisimple strata 
\[\Delta_L\approx \Delta\otimes L((r'-r)+),\ \Delta'_L\approx \Delta'\otimes L\]
respecting the associated splittings such that~$\beta_L$ and~$\beta'_L$ are conjugate in~$A\otimes_FL$, i.e.~$u\beta_Lu^{-1}=\beta'_L$ for some element~$u\in (A\otimes_FL)^\times$. 
Then by~\cite[(1.9) and (1.8)]{bushnellHenniart:96} it is possible to choose the pair of strata~$\Delta_L,\ \Delta'_{L}$ such that 
\[\Delta_L((r'-r)-)\approx \Delta\otimes L,\ \Delta'_L((r'-r)-)\approx \Delta'\otimes L((r'-r)-)\]
(in fact we use~$\dag$-construction and Theorem~\ref{thmDiagonalization} to reduce to the strict case to be able to apply~\cite[(1.9) and (1.8)]{bushnellHenniart:96}).
Now let~$\theta_L$ be an extension of~$\theta$ to~$H(\Delta_L((r'-r)-)))$ and~$\theta'_L$ be the transfer of~$\theta_L$ from~$\Delta_L((r'-r)-)$ to~$\Delta'_L((r'-r)-)$. 
Then~$uC_{(A\otimes_FL)^\times}(\beta_L)$ is a subset of~$I(\theta_L,\theta'_L)$ and therefore~$I(\Delta_L((r'-r)-),\Delta'_L)$ is a subset of~$I(\theta_L,\theta'_L|_{H(\Delta'_L)})$ by~\cite[Theorem 6.22 and Proposition 9.11]{skodlerackStevens:18}. 
Thus we have
\[I(\Delta,\Delta')\subseteq I(\theta,\theta')\]
for~$\theta':=\theta'_L|_{H(\Delta')}$ 
\end{enumerate}
\end{remark}

\begin{proposition}\label{propUniqueTransferEqualParameters}
 Suppose for~$\theta\in\C(\Delta)$ and two semisimple characters~$\theta'$ and~$\theta''$ in~$\C(\Delta')$ 
 that $I(\Delta,\Delta')\cap I(\theta,\theta')$ and~$I(\Delta,\Delta')\cap I(\theta,\theta'')$ are non-empty, then~$\theta'$ and~$\theta''$ coincide and are transfers 
 of~$\theta$ from~$\Delta$ to~$\Delta'$. In particular there is exactly one transfer of~$\theta$ from~$\Delta$ to~$\Delta'$.
\end{proposition}

In the proof we need the process of interior lifting and the process of base change which are used in~\cite{broussousSecherreStevens:12}. 

\begin{proof}
We assume~$\beta=\beta'$ without loss of generality, by Remark~\ref{remDiagonalization}. The transfer exists by Remark~\ref{remConstructionOfTransfer}, and by transitivity
we can assume that~$\theta'$ is this transfer. 
Then~$I(\theta,\theta')$ contains~$C_{A^\times}(\beta)$. Further~$C_{A^\times}(\beta)\cap I(\theta.\theta'')$ is non-empty, because~$1+\mf{m}(\Delta)$ and~$1+\mf{m}(\Delta')$
normalize~$\theta$ and~$\theta''$, respectively. So we can assume that~$1$ intertwines~$\theta$ with~$\theta''$. 
Let us start with the simple case. Using an interior lifting with respect to~$E_D$ and the base change with respect to~$L|E_D$ moves the setting to the split case, where 
 the statement is already known. (For example apply a~$\dag$-construction to reduce to the case of block-wise regular lattice sequences and 
apply~\cite[(3.6.1)]{bushnellKutzko:93}.) The semisimple case follows from the simple case because the characters respect the Iwahori-decomposition given by~$\beta$
 and are trivial on the unipotent parts, see~\ref{propIwahoriSt05no5.5}. 
\end{proof}

We denote the transfer map by~$\tau_{\Delta,\Delta'}:\C(\Delta)\ra\C(\Delta')$, i.e.~$\tau_{\Delta,\Delta'}(\theta)$ is the transfer of~$\theta$ from~$\Delta$ 
to~$\Delta'$. The map is well defined by the last proposition. 

For the next theorem we need the following subset: 
\[S(\Delta):=(1+\mf{m}(\Delta))J^{\lfloor\frac{-k_0(\Delta)+1}{2}\rfloor}(\beta,\Lambda).\]
It is the intersection of~$S(\Delta\otimes L)$ with~$ G$ because~$1+\mf{m}(\Delta\otimes L)$ normalizes the second 
factor~$J^{\lfloor\frac{-k_0+1}{2}\rfloor}(\beta\otimes_F1,\Lambda)$ and we can apply Lemma~\ref{lemCohom}. We have further the equality 
\[S(\Delta)=J^{\lfloor\frac{-k_0(\Delta)+1}{2}\rfloor}(\beta,\Lambda),\]
if~$r\leq\lfloor\frac{-k_0(\Delta)}{2}\rfloor$, because~$1+\mf{m}(\Delta)$ is contained in~$J^{\lfloor\frac{-k_0(\Delta)+1}{2}\rfloor}(\beta,\Lambda)$ 
by~\cite[Lemma 3.10(i)]{stevens:05} for~$\Delta\otimes_FL$ and restriction to~$ G$.
Further~$S(\Delta)$ is a subset of the normalizer of~$\theta$ because it is a subset of the normalizer of every extension~$\theta_L\in\C(\Delta\otimes L)$ of~$\theta$ 
by~\cite[Lemma 3.16]{stevens:05}.

\begin{proposition}[\cite{KSS:21}~Proposition A.10(i) for the split case]\label{propTransferAndInt}
Granted that~$\Delta$ and~$\Delta'$ intertwine, suppose that~$\theta'=\tau_{\Delta,\Delta'}(\theta)$. Then the set of intertwining elements from~$\theta$ to~$\theta'$ is 
the sets
\[S(\Delta') I(\Delta,\Delta') S(\Delta).\]
  If~$\beta=\beta'$ the formula simplifies to
 \begin{equation}\label{eqIntFormula}
  I(\theta,\theta')=S(\Delta') C_{A^\times}(\beta) S(\Delta).
 \end{equation}
\end{proposition}

For a subset~$S$ of~$A$ we have the duality operation:
\[S^*:=\{a\in A| aS\subseteq\ker(\psi_A)\}.\] 
As an example consider~$S=\mf{h}(\Delta)$. It is a full~$o_F$-lattice in~$A$ and the set of~$\Gal(L|F)$-fixed points of 
the~$\Gal(L|F)$-invariant~$o_L$-module~$\mf{h}(\Delta\otimes L)$. Thus~$\mf{h}(\Delta\otimes L)$ is equal to~$\mf{h}(\Delta)\otimes_{o_F}o_L$. 
And therefore~$\mf{h}(\Delta)^{*,\psi_F}$ is a subset of~$\mf{h}(\Delta\otimes L)^{*,\psi_L}$.

\begin{proof}[Proof of Proposition~\ref{propTransferAndInt}]
 The last assertion is a consequence of Proposition~\ref{propIntertwiningOfSemisimpleStratumOverD}. Let~$\theta_L\in\C(\Delta\otimes L)$ 
 and~$\theta'_L\in\C(\Delta'\otimes L)$ be extensions of~$\theta$ and~$\theta'$ and suppose that~$\theta'_L$ is a transfer of~$\theta_L$, see
 Remark~\ref{remConstructionOfTransfer}. Then~$A^\times\cap I(\theta_L,\theta'_L)$ is a subset of~$I(\theta,\theta')$. 
 Without loss of generality we can assume that~$\beta=\beta'$ by Remark~\ref{remDiagonalization}. We have formula~\eqref{eqIntFormula} for the split case in~\loccit. Taking~$\Gal(L|F)$-fixed points, a cohomology argument implies that~$A^\times\cap I(\theta_L,\theta'_L)$ 
 is equal to~$S(\Delta') C_{A^\times}(\beta) S(\Delta)$. 
 
 We need to show that~$I(\theta,\theta')$ is a subset of~$I(\theta_L,\theta'_L)$. 
 This is done by an induction on the critical exponent~$k_0(\Delta)$ which is equal to~$k_0(\Delta')$ by 
 Corollary~\ref{corEqParametersForIntertwiningSemisimpleStrata}\ref{corEqParametersForIntertwiningSemisimpleStrata.i},
because the strata are semisimple, intertwine,~$r=r'$ and~$e(\Lambda|F)=e(\Lambda'|F)$. 
 The base case is the statement for null strata which is evident. There are two cases in the induction step. 

 Case 1:~$r\geq \lfloor \frac{-k_0}{2}\rfloor$. By the end of the proof of Lemma~\cite[(3.3.5)]{bushnellKutzko:93} we have the following:
If~$x$ is an element of~$I(\theta,\theta')$ then~$\psi_{x\beta x^{-1}}$ is equal to~$\psi_\beta$ on~$x.H(\Delta)\cap H(\Delta')$. The latter holds if and only 
if~$x\beta x^{-1}-\beta$ is an element of~$(x.\mf{h}(\Delta)\cap \mf{h}(\Delta'))^*$ which is~$x.\mf{h}(\Delta)^*+\mf{h}(\Delta')^*$. 
Thus~$x$ is an element of~$I^+(\Delta,\Delta')$, the set of elements of~$A^\times$ which intertwine~$\beta+\mf{h}(\Delta)^*$ with~$\beta+\mf{h}(\Delta')^*$.
In particular,~$x$ is an element of~$I^+(\Delta\otimes L,\Delta'\otimes L)$ because~$\mf{h}(\Delta)^{*,\psi_F}$ is a subset of~$\mf{h}(\Delta\otimes L)^{*,\psi_L}$. A~$\dag$-construction, intertwining implies conjugacy~\cite[Theorem 10.2]{skodlerackStevens:18}
and~\cite[(3.3.8)]{bushnellKutzko:93} (Analogous proof of~$I(\theta_L)=I^+(\Delta_L)$ for the split semisimple case.) imply that~$I^+(\Delta_L,\Delta'_L)$ is a subset
of~$I(\theta_L,\theta'_L)$, Thus~$x$ is an element of~$I(\theta_L,\theta'_L)$.

 Case 2:~$r<\lfloor \frac{-k_0}{2}\rfloor$. In this case~$\Delta(1+)$ and~$\Delta'(1+)$ are still semisimple. Then the proof follows by induction. We have that~$I(\theta,\theta')$ is contained 
 in~$I(\theta(1+),\theta'(1+))$ which is contained in~$I(\theta_L(1+),\theta'_L(1+))$ which is the same as~$I(\theta_L,\theta'_L)$, 
 by the formula~\eqref{eqIntFormula} in the split case, see \loccit,
 using that
 $S(\Delta\otimes L)$ and~$S((\Delta\otimes L)(1+))$ coincide with $J^{\lfloor\frac{-k_0+1}{2}\rfloor}(\beta\otimes_F1,\Lambda)$ for those~$r$. 
\end{proof}

The proof of the last Proposition is very similar to the proof of the split simple counterpart in~\cite[(3.3.2)]{bushnellKutzko:93}.

\begin{corollary}\label{corConjugateAssSplittingsOfStrataWithSameSemisimpleCharacters}
 Suppose that~$C(\Delta)$ and~$\C(\Delta')$ intersect non-trivially, then there is an element of~$S(\Delta)$ which conjugates the first associated splitting to the second.
\end{corollary}

\begin{proof}
 Consider the two descriptions of the intertwining of a semisimple character of~$\C(\Delta)$ and take limits in the-$p$-adic topology. We get
 \[S(\Delta)C_A(\beta)S(\Delta)=S(\Delta')C_A(\beta')S(\Delta').\]
 As in the proof of~\cite[Proposition 9.9(iv)]{skodlerackStevens:18} we use congruence between the central simple idempotents of~$\C_A(\beta)$ and~$\C_A(\beta')$ 
 to obtain an element of~$S(\Delta)$ which conjugates the associated splitting of~$\Delta$ to the one of~$\Delta'$.
\end{proof}

\begin{proposition}\label{propRestrictionMapIsBijective}
 Suppose~$\Delta$ is a semisimple stratum and~$j$ a positive integer such that 
 \[\lfloor\frac{r}{e(\Lambda|E)}\rfloor=\lfloor\frac{r+j}{e(\Lambda|E)}\rfloor,\] i.e. both strata~$\Delta$ and~$\Delta(j+)$ have the same group level.
 Then the restriction map~$res_{\Delta,\Delta(j+)}$ from~$\C(\Delta)$ to~$\C(\Delta(j+))$ 
 is bijective. 
\end{proposition}

\newcommand{\Nrd}{\text{Nrd}}
\begin{proof}
It is enough to show the statement for~$j=1$.
 For null-strata there is nothing to prove and in the non-null case we move to a core approximation to reduce to~$r<\lfloor\frac{-k_0(\Delta)}{2}\rfloor.$ Then~$H(\Delta)=P_{r+1}(j_E(\Lambda))H(\Delta(1+))$. 
 Take an element~$\theta\in\C(\Delta(1+))$. Then we only have to prove that there is a unique extension of~$\theta|_{P_{r+2}(j_E(\Lambda))}$ to~$P_{r+1}(j_E(\Lambda))$. This
 follows immediately form having the same group level because the image of~$\Nrd_{E_i}$ on~$P_{r}(j_{E_i}(\Lambda^i))$ is equal to~$P_{\lfloor\frac{r}{e(\Lambda^i|E_i)}\rfloor}(o_{E_i})$, so it 
 is the same for~$r$ and for~$r+1$ by both strata having the same group level. Thus there is exactly one extension of~$\theta$ into~$\C(\Delta)$.
\end{proof}

The last proposition allows to describe the transfer from~$\Delta$ to~$\Delta'$ for intertwining strata in the general case (same group level and same degree).without further base change to~$L$, just in reducing to having the same period and~$r'=r$. Indeed, say~$r\leq r'$ and that both strata have the same~$F$-period. 
We claim that the transfer from~$\Delta$ to~$\Delta'$ is given by the map
\[\tau_{\Delta,\Delta'}(\theta):=\tau_{\Delta((r'-r)+),\Delta'}(\theta((r'-r)+)), \theta\in\C(\Delta).\]
Note that this map is a bijection, being composed by two bijections. Given~$\theta\in\C(\Delta)$ and~$\theta'\in\C(\Delta')$ such that~$\theta'$ is a transfer of~$\theta$,
i.e.~$I(\Delta,\Delta')\subseteq I(\theta,\theta')$, then~$\tau_{\Delta((r'-r)+),\Delta'}(\theta((r'-r)+))=\theta'$, by Proposition~\ref{propUniqueTransferEqualParameters}. Thus for~$\theta\in\C(\Delta)$ there exists exactly one transfer from~$\Delta$ to~$\Delta'$ and for~$\theta'\in\C(\Delta')$
there exists exactly one transfer from~$\Delta'$ to~$\Delta$.

\subsection{The~$\dag$-construction}\label{subsecDag}
Suppose~$\theta\in\C(\Delta)$ is a semisimple character, and fix an integer~$l$ and integers~$s_k,\ k=1,\ldots l$. Then we define~$\Delta^{\dag,(s_k)_k}$ as the direct 
sum of the strata~$[\Lambda-s_k,n,r,\beta]$ where~$k$ passes from~$1$ to~$l$. Let~$U_-MU_+$ be the Iwahori decomposition with respect to the 
direct sum~$V^\dag=\bigoplus_k V$. Then~$H(\Delta^{\dag,(s_k)_k})$ respects this Iwahori decomposition and we define a map~$\theta^{\dag,(s_k)_k}$ on~$H(\Delta^{\dag,(s_k)_k})$
via 
\[\theta^{\dag,(s_k)_k}(u_- x u_+):=\prod_k\theta(x_k).\]

\begin{proposition}\label{propDdag}
 The map~$\theta^{\dag,(s_k)_k}$ is an element of~$\C(\Delta^{\dag,(s_k)_k})$.
\end{proposition}

\begin{proof}
 The~$\dag$-construction and this proposition are already known in the split case. See for example~\cite[Lemma 4.3]{kurinczukStevens:15}. 
 Thus, for an extension~$\theta_L\in\C(\Delta\otimes L)$ the character~$\theta_L^{\dag,(s_k)_k}$ is an element of~$\C(\Delta^{\dag,(s_k)_k}\otimes L)$, and its 
 restriction to~$H(\Delta^{\dag,(s_k)_k})$ is~$\theta^{\dag^{(s_k)_k}}$. 
\end{proof}

One important application is the construction of sound strata. A simple stratum~$\Delta$ is called~\emph{sound} if~$[\Gamma]$ and~$j_E([\Gamma])$ are mid points of the facets of
regular lattice sequences ($\Gamma$ being the lattice function attached to~$\Lambda$, see below Definition~\ref{defTranslationClasses}), i.e. equivalently~$\Lambda$ is a strict~$o_E$-$o_D$-lattice sequence such that~$\mf{b}_{0}$ is the order of a regular~$o_{D_E}$-lattice sequence and such that~$\mf{n}(\mf{b}_{\Lambda,0})$ coincides
with~$\mf{n}(\mf{a}_{\Lambda,0})\cap B^\times$, see~\cite[Corollary 1.4(ii)]{grabitz:99}.

\begin{definition}[\cite{broussousSecherreStevens:12}~Lemma~2.16,~Proposition~2.17]
Suppose we are given a simple stratum~$\Delta$. 
 Let~$e$ be the~$F$-period of~$\Lambda$. This coincides with the~$F$-period of~$j_E(\Lambda)$. Take the sequence~$s_k=k$ for~$k=0,\ldots,e-1$. 
 Then~$\Delta^\ddagger:=\Delta^{\dag,(s_k)_k}$ is a sound stratum. Analogously we define~$\theta^\ddagger$.
\end{definition}

Sound simple characters have been studied in~\cite{grabitz:07}.

%

\subsection{Derived characters}
Let~$\Delta$ and~$\Delta'$ be semisimple strata of the same period and with~$r=r'$. Let~$\theta\in\C(\Delta)$ and~$\theta'\in\C(\Delta')$ be given semisimple characters. 

\begin{definition}\label{defDerivedChar}
We need an analogue to the strata induction. For that we need the derived character: 
Let~$\theta_0\in\C(\Delta(1)(1-))$ be an extension of~$\theta(1+)$. 
Then there is a~$c\in\mf{a}_{-r-1}$ such that~$\theta$ is equal to~$\psi_{\beta-\beta(1)+c}\theta_0$.
Let~$s$ be a tame corestriction on~$A$ relative to~$E(1)=F[\beta(1)]|F$. We write~$\partial_{\beta(1),s,\theta_0}\theta$ for the character~$\psi_{s(\beta-\beta(1)+c)}$ defined
 on~$P_{r+1}(j_{E(1)}(\Lambda))$. 
 \end{definition}

 \begin{proposition}\label{lemDerivedCharacters}
Suppose~$\theta$ and~$\theta'$ intertwine and that we can choose~$\Delta(1)$ and~$\Delta'(1)$ such that~$\beta(1)=\beta'(1)=:\gamma$. 
Suppose further that~$\theta'(1+)$ is a transfer of~$\theta(1+)$ from~$\Delta(1)$ to~$\Delta'(1)$. Let~$\theta_0\in\C(\Delta(1)(1-))$ be an extension of~$\theta(1+)$ 
and let~$\theta'_0$ be the transfer~$\tau_{\Delta(1)(1-),\Delta'(1)(-1)}(\theta_0)$. Let~$s$ be a tame corestriction on~$A$ relative to~$F[\gamma]|F$.
Then we have the following assertions:
\begin{enumerate}
 \item If an element~$g=uxv\in I(\theta(1+),\theta'(1+))$ (with decomposition from equation (\ref{eqIntFormula})) intertwines~$\theta$ with~$\theta'$ then~$x$
 intertwines~$\partial_{\gamma,s,\theta_0}\theta$ with~$\partial_{\gamma,s,\theta'_0}\theta'$. \label{lemDerivedCharactersAss1}
 \item For every~$x\in I(\partial_{\gamma,s,\theta_0}\theta,\partial_{\gamma,s,\theta'_0}\theta')$, there are 
 elements~$u\in S(\Delta'(1))$ and~$v\in S(\Delta(1))$ such that~$uxv$ intertwines~$\theta$ with~$\theta'$. \label{lemDerivedCharactersAss2}
\end{enumerate}
\end{proposition}

For the proof we need the following generalization of Sequence~\eqref{eqExactSeqDelta}. 
\begin{lemma}\label{lemExSeqDeltaMixed}
 Let~$\Delta$ and~$\Delta'$ be semisimple strata of the same period with~$\beta=\beta'$ and~$r=r'$ and let~$s$ be a tame corestriction on~$A$ relative to~$F[\beta]|F$. 
 Then the sequence
 \[\mf{m}(\Delta)+\mf{m}(\Delta')\stackrel{\alpha_\beta}{\longrightarrow}(\mf{a}_{-r}+\mf{a}'_{-r})/(\mf{a}_{-r+1}+\mf{a}'_{-r+1})\stackrel{s}{\longrightarrow}(\mf{b}_{-r}+\mf{b}'_{-r})/(\mf{b}_{-r+1}+\mf{b}'_{-r+1})\ra 0\]
 is exact. 
\end{lemma}

\begin{proof}
 The proof is analogous to the proof of Lemma~\ref{lemExSeqDelta}. Because the exactness of this sequence is known in the split case by~\cite[Lemma 6.21, Proposition 7.6]{skodlerackStevens:18}. 
\end{proof}

\begin{proof}[Proof of Lemma~\ref{lemDerivedCharacters}]
 \begin{enumerate}
  \item The element~$x$ intertwines~$v.\theta$ with~$u^{-1}.\theta'$. 
  Thus~$x$ intertwines the restriction of~$v.\theta$ and~$u^{-1}.\theta'$ to~$P_{r+1}(j_{F[\gamma]})$. So we have to calculate these restrictions. 

  We write~$\theta$ and~$\theta'$ as in Definition~\ref{defDerivedChar}:~$\theta=\theta_0\psi_{\beta-\gamma+c}$ and~$\theta'=\theta_0\psi_{\beta'-\gamma+c'}$  for some elements~$c\in\mf{a}_{-1-r}$ and~$c'\in\mf{a}_{-1-r}$. 
  Claim:~$v.\theta$ coincides with~$\psi_{s(\beta-\gamma+c)} \theta_0$ on $P_{r+1}(j_{F[\gamma]})$.
  The claim and the analogous statement for~$\theta'$ would imply that~$x$ intertwines~$\psi_{s(\beta-\gamma+c)}$ with~$\psi_{s(\beta'-\gamma+c')}$, because~$x$ 
  intertwines~$\theta_0$ with~$\theta'_0$. Now let us prove the claim:
  \begin{eqnarray}
   v.\theta &=& v.\psi_{\beta-\gamma+c} v.\theta_0 \\
   &=& \psi_{\beta-\gamma+c} v.\theta_0\\
   &=& \psi_{\beta-\gamma+c} \theta_0 \psi_{v\gamma v^{-1}-\gamma} 
  \end{eqnarray}
  In general for an element~$a$ of~$\mf{a}_{-r-1}$ the restriction of~$\psi_a$ to~$P_{r+1}(j_{F[\gamma]})$ is equal to~$\psi_{s(a)}$.
  So we have to show that~$s(v\gamma v^{-1})$ congruent to~$s(\gamma)$ modulo~$\mf{a}_{-r}$. 
  We have
  \begin{eqnarray}
   v\gamma v^{-1} &=& (\gamma v+v\gamma-\gamma v)v^{-1} \\
   &\equiv & \gamma+(v\gamma-\gamma v) \text{ mod~$\mf{a}_{-r}$},\\
   \end{eqnarray}
  because~$v\gamma-\gamma v$ is an element of~$\mf{a}_{-1-r}$ and~$v$ is a~$1$-unit of~$\mf{a}_0$.
  We now apply~$s$ to obtain that~$s(v\gamma v^{-1})$ is congruent to~$s(\gamma)$ modulo~$\mf{a}_{-r}$, because~$v\gamma-\gamma v$ is in the kernel of~$s$.
  \item It is the reverse of the first part. Suppose~$x\in B^\times_\gamma$ intertwines~$\psi_{s(\beta-\gamma+c)}$ with~$\psi_{s(\beta'-\gamma+c')}$. 
  Then then~$s(x(\beta-\gamma+c)x^{-1})$ is congruent to~$s(\beta'-\gamma+c')$ modulo~$x\mf{a}_{-r}x^{-1}+\mf{a}'_{-r}$. By Lemma~\ref{lemExSeqDeltaMixed} 
  there are elements~$y_v\in\mf{m}(\Delta)$ and~$y_u\in\mf{m}(\Delta')$ such that
  \[x(\beta-\gamma+c)x^{-1}-(\beta'-\gamma+c')\equiv\alpha_\gamma(x y_v x^{-1})+\alpha_\gamma(y_u)\]
  modulo~$x\mf{a}_{-r}x^{-1}+\mf{a}'_{-r}$. Define~$u$ by~$1+y_u$ and~$v$ by~$1+y_v$. Then
  \[x v(\beta+c)v^{-1} x^{-1}\equiv u^{-1}(\beta'+c')u.\]
  Thus~$v.\psi_{\beta+c}$ and~$u^{-1}.\psi_{\beta'+c'}$ intertwine by~$x$ and 
  \[v.\psi_{\beta+c}=v.\psi_{\beta+c-\gamma}\psi_{v\gamma v^{-1}-\gamma}\psi_\gamma.\]
  The element~$x$ intertwines~$\psi_\gamma$ (on~$G$) and thus~$v.\psi_{\beta+c-\gamma}\psi_{v\gamma v^{-1}-\gamma}$
  with~$u^{-1}.\psi_{\beta'+c'-\gamma}\psi_{u^{-1}\gamma u-\gamma}$.  Further~$\theta_0$ and~$\theta'_0$ are intertwined by~$x$. 
  Thus~$x$ intertwines~$v.\theta$ with~$u^{-1}.\theta'$ and thus~$uxv\in I(\theta,\theta')$. 
 \end{enumerate}
\end{proof}

\subsection{Simple characters}
Let us recall that a semisimple characters~$\theta$ is called~\emph{simple} if there exists a simple stratum~$\Delta$ such that~$\theta\in\C(\Delta)$. 
In this section we fix two simple strata~$\Delta$ and~$\Delta'$ and simple characters~$\theta\in\C(\Delta)$ and~$\theta'\in\C(\Delta')$.  

Here we collect some facts about simple characters.

\begin{proposition}[{see \cite[Proposition 9.1, Proposition 9.9]{grabitz:07}, for sound strata}]\label{propEqualDegreesForIntersectingSetsOfSimpleCharacters}
 Suppose~$\Lambda=\Lambda'$,~$r=r'$ and~$\C(\Delta)\cap\C(\Delta')\neq\emptyset$. Then both strata have coinciding intertia degrees, coinciding ramification indexes and 
 coinciding critical exponents. 
\end{proposition}

\begin{proof}
By Proposition~\ref{propIntertwiningAndLevels} either both strata are null or both are non-null, because of intertwining,~$r=r'$ and~$\Lambda=\Lambda'$. 
In the first case there is nothing to prove. So we can assume that
both simple strata are non-null. 
The intertwining of an element of~$\theta\in\C(\Delta)\cap\C(\Delta')$ can be described with~$\Delta$ and~$\Delta'$:
\[I(\theta)=S(\Delta)B^\times S(\Delta)=S(\Delta')B'^\times S(\Delta').\]
We intersect both sides with~$\mf{a}^\times$ and we factorize by~$1+\mf{a}_1$ to obtain:
\[\prod_i\GL_{s_i}(\tilde{\kappa})=\prod_{i'}\GL_{s'_i}(\tilde{\kappa}').\]
Wedderburn  implies~$\sum_is_i=\sum_{i'}s_{i'}$ and we denote this sum as~$\tilde{m}$. We have~$B=M_{\tilde{m}}(D_\beta)$ where~$D_\beta$ is the skew-field which is 
Brauer equivalent to~$B$ and~$E\otimes_FD$. In particular~$\deg D_\beta=\frac{d}{\gcd(d,[E:F])}$ 
The double centralizer theorem states
\[\deg A =[E:F]\deg B=\tilde{m}\frac{d}{\gcd(d,[E:F])}[E:F].\]
Further~$[\ti{\kappa}:\kappa_F]=f(E|F)\frac{d}{\gcd(d,[E:F])}$. 
\newcommand{\im}{\text{im}}
We have the same equations for~$\Delta'$ and we get by the quotient of the first with the second equation that both strata have the same ramification index. 
For the intertia degrees we claim that the minimum of~$\im(\nu_F\circ\Nrd_{A|F}(I(\theta)))$ in~$\mathbb{R}^{>0}$ is equal to~$f(E|F)$ and~$f(E'|F)$.
This minimum is realized if one takes an element~$x$ of minimal~$\nu_E$-positive reduced norm~$\Nrd_{B|E}$ in~$B$. Such an element~$x$ satisfies~$\Nrd_{B|E}(x)\in o_E^\times\pi_E$.
Thus~$\Nrd_{A|F}(x)\in\pi_F^{f(E|F)}o_F^\times$.
Which finishes the proof of the equality of the inertia degrees. 
The equality for the critical exponents follows now directly from Corollary~\ref{corSimpleStrataFieldCriteria}, because of if for example~$\Delta(1+)$ is still simple,
in particular~$n>r+1$, but~$\Delta'(1+)$ not, then take~$\tilde{\Delta}$ is equivalent to~$\Delta'(1+)$ and simple. Then~$[E:F]=[\tilde{E}:F]<[E':F]=[E:F]$ where the equalities 
are from Part 1 of the proof.
A contradiction. 
\end{proof}

From that Proposition follows now:

\begin{proposition}\label{propEqualDegreeFromIntertwining}
 Suppose~$\Lambda$ and~$\Lambda'$ have the same~$F$-period,~$r=r'$ and~$\theta$ and~$\theta'$ intertwine then $e(E|F)=e(E'|F)$,~$f(E|F)=f(E'|F)$ and both strata have the same critical exponent. 
\end{proposition}

\begin{proof}
 We consider at first the case where~$\Lambda=\Lambda'$. 
 By~\cite[Proposition 4.5]{broussousSecherreStevens:12} we can assume without loss of generality that~$\Delta$ and~$\Delta'$ are sound with the same Fr\"ohlich invariant.
 So we want to apply~\cite[Theorem 10.3]{grabitz:07}. But this Theorem uses~\cite[Proposition 9.1]{grabitz:07} with a gap in the proof, which is filled by 
 Proposition~\ref{propEqualDegreesForIntersectingSetsOfSimpleCharacters}. Now~\cite[Theorem 10.3]{grabitz:07}, see also~\cite[Theorem 1.16]{broussousSecherreStevens:12}, 
 implies~$e(E|F)=e(E'|F)$ and~$f(E|F)=f(E'|F)$, and~$\theta_1$ is conjugate to~$\theta_2$ by an element of the normalizer of~$\Lambda$ 
 and Proposition~\ref{propEqualDegreesForIntersectingSetsOfSimpleCharacters} finishes the proof for the case~$\Lambda=\Lambda'$. 
 Suppose now that~$\Lambda$ is different from~$\Lambda'$. Then a~$\dag$-construction, see Proposition~\ref{propDdag}, reduces to the case of conjugate regular lattice sequences. 
 See Corollary~\ref{corEndoEquivalenceForSimpleStrata}\ref{corEndoEquivalenceForSimpleStrataAss.iii} for invariance of the critical exponent under this construction.
\end{proof}

Recall that~$\left\lfloor \frac{r}{e(\Lambda|E)}\right\rfloor$ is the group level of~$\Delta$, see Definition~\ref{defGroupLevel}. 

\begin{corollary}\label{corEqualDegreesFromIntertwininSameGrouplevelSameDegree}
 Suppose two~$\Delta$ and~$\Delta'$ have the same degree and the same group level. And suppose that~$\theta$ and~$\theta'$ intertwine. 
 Then~$f(E|F)=f(E'|F)$,~$e(E|F)=e(E'|F)$ and~$k_0(\Delta)=k_0(\Delta')$. 
\end{corollary}

\begin{proof}
 By passing to affine translates of the strata we can assume that both lattice sequences have the same period. If~$r=r'$ then we can apply Proposition~\ref{propEqualDegreeFromIntertwining}. 
 Suppose now that~$r<r'$. Then there is a simple stratum~$\Delta(r'-r)$ equivalent to~$\Delta((r'-r)+)$ and~$\theta((r'-r)+)$ and~$\theta'$ still intertwine. 
 Thus~$E(r'-r)$ and~$E'$ have the same degrees and the same inertia degrees by Proposition~\ref{propEqualDegreeFromIntertwining}.
 Thus~$E|F$ and~$E(r'-r)|F$ have the same ramification index by Proposition~\ref{propSimplePure},i.e.~$E|F$ and~$E'|F$ have the same ramification index. Thus the 
 stratum~$\Delta((r'-r)+)$ is still simple because 
 \[-k_0(\Delta)\geq e(\Lambda|E)\left(1+\left\lfloor\frac{r}{e(\Lambda|E)}\right\rfloor\right)=e(\Lambda|E')\left(1+\left\lfloor\frac{r'}{e(\Lambda|E')}\right\rfloor\right)>r'.\]
 We can therefore apply Proposition~\ref{propEqualDegreeFromIntertwining} to~$\theta((r'-r)+)\in\C(\Delta((r'-r)+))$ and~$\theta'\in\C(\Delta')$. 
\end{proof}

\begin{proposition}[{\cite{broussousSecherreStevens:12}~Theorem~4.16}]
 Suppose~$\Lambda=\Lambda'$ and~$r=r'$. Then the sets~$\C(\Delta)$ and~$\C(\Delta')$ coincide if they have a non-empty intersection. 
\end{proposition}

%
%
%



\begin{proposition}[\cite{broussousSecherreStevens:12}~Theorem~1.12]\label{propIntImplConjSimpleChar}
 Suppose~$\Delta$ and~$\Delta'$ are simple with the same embedding type,~$r=r'$ and~$\Lambda=\Lambda'$. Let~$E_{ur}$ (resp.~$E'_{ur}$) be the maximal unramified field extension in~$E|F$ (resp.~$E'|F$).
 Let~$\theta\in\C(\Delta)$ and~$\theta'\in\C(\Delta')$  be two simple characters which intertwine. Then~$\theta$ is conjugate to~$\theta'$ by an element of 
 the normalizer of~$\Lambda$ which conjugates~$E_{ur}$ to~$E'_{ur}$. 
\end{proposition}

\begin{proposition}[\cite{broussousSecherreStevens:12}~Theorem~1.11,~Corollary~8.3]\label{propTranstivityOfIntertwiningSimpleCharacters}
 Intertwining is an equivalence relation on the set of all simple characters for~$G$ for simple strata of fixed group level and fixed degree.
\end{proposition}

\begin{proof}
 \cite[Theorem~1.11]{broussousSecherreStevens:12} only implies this for simple characters for the same lattice sequence. Corollary~\ref{corEqualDegreesFromIntertwininSameGrouplevelSameDegree}
 provides that the simple characters in question have the same inertia degree and the same ramification index. So, by Corollary~\ref{corconsructingEmbeddingsFromNumericalData}, we can transfer to the case where all characters are attached to 
 the same lattice sequence and all \black{underlying strata} have the same embedding type, if we have proven the following statement:
 
 \emph{ Suppose~$\Delta,\Delta'$ and~$\Delta''$ are three simple strata of the same degree and the same group level.
 Suppose that~$\beta'=\beta''$,~$\Lambda=\Lambda'$ and that~$\Delta$ and~$\Delta'$ have the same embedding type.
 Suppose further that~$\theta\in\C(\Delta),\theta'\in\C(\Delta')$ and~$\theta''\in\C(\Delta'')$ are simple characters such that~$\theta'$ and~$\theta''$ are 
 transfers of each other. Then~$\theta$ and~$\theta'$ intertwine if and only if~$\theta$ and~$\theta''$ intertwine.}
 
 Proof of this statement: Suppose at first that~$\theta$ and~$\theta'$ intertwine, say~$r\geq r'$. Then~$\theta$ and~$\theta'((r-r')+)$ intertwine and thus they are 
 conjugate by Proposition~\ref{propIntImplConjSimpleChar}. Thus~$\theta$ and~$\theta''$ intertwine because~$\theta'((r-r')+)$ is the transfer of~$\theta''$. 
 If~$r<r'$ then we take the transfer~$\ti{\theta}'$ of~$\theta''$ from~$\Delta''$ to~$\Delta'((r-r')+)$. Then~$\ti{\theta}'$ is the unique extension of~$\theta'$. 
 Now~$\theta$ and~$\ti{\theta}'$ intertwine by~\cite[Theorem~1.11]{broussousSecherreStevens:12} and we proceed as in the case~$r\geq r'$. We prove now the other direction.
 Suppose now that~$\theta$ and~$\theta''$ intertwine. We pass to affine translates of the strata~$\Delta$ and~$\Delta''$ followed by a~$\dag$-construction to obtain strata with  conjugate regular lattice sequences and the obtained characters~$\theta^\dag$ 
 and~$\theta''^\dag$ still intertwine. Then by~\cite[Theorem~1.11]{broussousSecherreStevens:12}~$\theta$ and~$\theta'$ intertwine. 
\end{proof}

\subsection{Equality of sets of semisimple characters}
Here we are given two semisimple strata~$\Delta$ and~$\Delta'$ with~$\Lambda=\Lambda'$,~$r=r'$ and~$n=n'$.

%

\begin{lemma}\label{lemJumpseq}
 The strata~$\Delta$ and~$\Delta'$ have the same jump sequence if~$\C(\Delta)\cap \C(\Delta')$ is non-empty. In particular~$\Delta$ and~$\Delta'$ have the same critical 
 exponent. 
\end{lemma}

\begin{proof}
 The second assertion follows from the first one by the definition of the jump sequence, see~Definition~\ref{defDefiningSequence}\ref{defDefiningSequenceiii}. 
 Thus we only have to prove the first assertion. 
 At first we remark that we can assume that both strata have the same associated splitting by 
 Corollary~\ref{corConjugateAssSplittingsOfStrataWithSameSemisimpleCharacters}, and we take the same indexing for both strata, i.e.~$I=I'$.
 At a jump~$r+j_t$ in the defining sequence of~$\Delta$ the number of blocks decreases or for one stratum~$\Delta_i$ there is a jump at~$r+j_t$. 
 Now,~$\Delta'(j)$ has the same number of blocks as~$\Delta(j)$, because this number is determined by the intertwining of any element of~$\C(\Delta'(j))\cap \C(\Delta(j))$.
 Thus in the first case~$\Delta'(j_t)$ has less blocks than~$\Delta'(j_t-1)$. 
 In the second case there is also a jump for the stratum~$\Delta'_i$ at~$r+j_t$, by Lemma~\ref{propEqualDegreesForIntersectingSetsOfSimpleCharacters}.
\end{proof}

For the next Lemma we need the core approximation, see Definition~\ref{defDefiningSequence}. The idea of a core approximation~$\tilde{\Delta}$ of~$\Delta$ is that we have
\[H(\Delta)=H^{r+1}(\tilde{\beta},\Lambda)=P_{r+1}(j_{\tilde{E}}(\Lambda))H^{r+2}(\tilde{\beta},\Lambda).\]

\begin{lemma}[{{\cite[Proposition~(3.5.7)]{bushnellKutzko:93}}}]\label{lemBK3.5.7}
 Let~$\theta$ and~$\theta'$ be  elements of~$\C(\Delta)$. Suppose further that~$\tilde{\Delta}$ is a core approximation of $\Delta$. Then~$\theta/\theta'$ is intertwined by 
 $S(\tilde{\Delta})C_{A^\times}(\tilde{\beta})S(\tilde{\Delta})$.
\end{lemma}

\begin{proposition}[{see~\cite[Theorem~(3.5.8)]{bushnellKutzko:93} and~\cite[Theorem 4.16]{broussousSecherreStevens:12} for the simple case}]\label{propEqualityByNonTrivalIntersection}
 $\C(\Delta)$ and~$\C(\Delta')$ coincide if they intersect non-trivially. 
\end{proposition}

The strategy of the proof is from~\cite[(3.5.8)]{bushnellKutzko:93}.

\begin{proof}
 The proof is done by induction on~$n-r$. Let us suppose that not both sets~$\C(\Delta)$ and~$\C(\Delta')$ are singletons, because otherwise the 
 statement is trivial, in particular we are in the case of 
 non-null strata and~$r<\lfloor \frac{n}{2}\rfloor$. 
 By induction hypothesis we can assume that~$\C(\Delta(1+))$ is equal to~$\C(\Delta'(1+))$. Let~$\tilde{\theta}$ be an element of~$\C(\Delta)\cap\C(\Delta')$, and take 
 core approximations~$\tilde{\Delta}$ and~$\tilde{\Delta}'$ of~$\Delta$ and~$\Delta'$, respectively. We have to show that~$\C(\Delta)$ is contained in~$\C(\Delta')$. 
 So let~$\theta$ be an element of~$\C(\Delta)$. Then there are elements~$b\in\mf{a}_{-r-1}$ and~$\theta'\in\C(\Delta')$ such that~$\theta$ is equal to~$\theta'\psi_b$.
 This is possible, because the restriction map from~$\C(\Delta')$ to~$\C(\Delta'(1+))$ is surjective. 
 Then the qoutients~$\theta/\tilde{\theta}$ and~$\theta'/\tilde{\theta}$ are intertwined by the intertwining of every element of the non-empty intersection~$C(\tilde{\Delta})\cap\C(\tilde{\Delta}')$
 by Lemma~\ref{lemBK3.5.7}. Thus the character
 \[\psi_b=\theta/\tilde{\theta} (\theta'/\tilde{\theta})^{-1}\]
 is intertwined by~$C_{A^\times}(\tilde{\beta}')$. And by~\cite[Lemma 2.10]{secherreStevensVI:12} we obtain that~$s_{\tilde{\beta}'}(b)$ is congruent to an element 
 of~$\tilde{E}'$  modulo~$\mf{a}_{-r}$ and thus the restriction of~$\psi_b$ to~$\P(j_{\tilde{E}'}(\Lambda))$ factorizes through the reduced norm. 
 Thus the restriction of~$\phi:=\theta\psi_{\tilde{\beta}'-\beta'}$ to~$\P(j_{\tilde{E}'}(\Lambda))$ factorizes through the reduced norm. 
 Secondly,~$\phi|_{H(\Delta(1+))}$ coincides with~$\theta'(1+)\psi_{\tilde{\beta}'-\beta'}$ and thus an element of~$\C(\tilde{\Delta}'((\tilde{r}'-r-1)-))$ 
 and thus~$\phi$ is an element of~$\C(\tilde{\Delta}'((\tilde{r}'-r)-))$ because~$r<\lfloor \frac{-k_0(\tilde{\Delta}')}{2}\rfloor$. 
 Thus~$\theta$ is an element of 
 \[\psi_{\beta'-\tilde{\beta}'}\C(\tilde{\Delta}'((\tilde{r}'-r)-))=\C(\Delta').\]
\end{proof}

\begin{lemma}[see \cite{broussousSecherreStevens:12}~Lemma~4.12 for the case of simple Strata, see also~\cite{bushnellKutzko:93}~(3.5.9) for the split simple case]\label{lemSemisimplecharacterDeterminesTheGroupOneLevelLower}
 Suppose that~$\C(\Delta)$ and~$\C(\Delta')$ coincide and~$r\geq 1$. Then~$H(\Delta(1-))$ is equal to~$H(\Delta'(1-))$.
\end{lemma}


\begin{proof}
 The proof is done by induction on~$r$. We can assume that~$\Delta$ and~$\Delta'$ have the same associated splitting, by 
 Corollary~\ref{corConjugateAssSplittingsOfStrataWithSameSemisimpleCharacters}, and we take for both strata the same indexing. We also write~$k_0$ for both critical
 exponents, which agree by Lemma~\ref{propEqualDegreesForIntersectingSetsOfSimpleCharacters}. At first let us assume that~$r$ is not greater than~$\lfloor \frac{-k_0}{2}\rfloor$.
 Then we have the following equalities
 \[H(\Delta(1-))=H(\Delta)(\prod_iH(\Delta_i(1-)))\text{ and }H(\Delta'(1-))=H(\Delta')(\prod_iH(\Delta'_i(1-))).\]
 From the equality of~$\C(\Delta)$ with~$\C(\Delta')$ we obtain a block wise equality, i.e.~$\C(\Delta_i)$ and~$\C(\Delta'_i)$ coincide.
 The already proven simple case, see~{\it loc.cit.}, implies that~$H(\Delta_i(1-))$ is equal to~$H(\Delta'_i(1-))$ for all indexes~$i\in I$ and thus
 we obtain the desired equality of~$H(\Delta(1-))$ with~$H(\Delta'(1-))$. 
 Secondly, we need to consider the case~$r>\lfloor \frac{-k_0}{2}\rfloor$. The proof is the same as in~\cite[(3.5.9)]{bushnellKutzko:93}
 using the jump sequences, which agree by Lemma~\ref{lemJumpseq}, and core approximations. Let us repeat the argument.
 We take core approximations~$\tilde{\Delta}$ and~$\tilde{\Delta}'$ for~$\Delta(1-)$ and~$\Delta'(1-)$, respectively. We have
 that the group~$H(\Delta(1-))$ and~$H^{r}(\tilde{\beta},\Lambda)$ coincide, by~\cite[Proposition 9.1(ii)]{skodlerackStevens:18} and restriction
 to~$G$. Then~$\tilde{r}>r$ because~$\Delta$ is semisimple. Thus, the sets of characters~$\C(\tilde{\Delta})$ and~$\C(\tilde{\Delta}')$ coincide. 
 The~$p$-adic closure of the intertwining of a semisimple character in terms of~$\tilde{\Delta}$ and~$\tilde{\Delta}'$, respectively, followed by an intersection with~$\mf{a}_{r}(\Lambda)$ and an additive closure implies 
 \[\tilde{\mf{b}}_{r}+\mf{m}(\ti{\Delta})\tilde{\mf{b}}_r+\mf{j}^{\lfloor\frac{-k_0(\tilde{\Delta})+1}{2}\rfloor}(\tilde{\beta},\Lambda)\tilde{\mf{b}}_{r}=\tilde{\mf{b}}'_{r}+\mf{m}(\ti{\Delta}')\tilde{\mf{b}}'_r+\mf{j}^{\lfloor\frac{-k_0(\tilde{\Delta}')+1}{2}\rfloor}(\tilde{\beta}',\Lambda)\tilde{\mf{b}}'_{r}.\]
 We add on both sides~$\mf{h}(\Delta)$, which is equal to~$\mf{h}(\Delta')$, and we obtain
 \[\mf{h}(\Delta(1-))=\mf{h}^{r}(\tilde{\beta},\Lambda)=\tilde{\mf{b}}_{r}+\mf{h}(\Delta)=\tilde{\mf{b}}'_{r}+\mf{h}(\Delta')=\mf{h}(\Delta'(1-)).\]
 It is worth to mention why~$\mf{m}(\ti{\Delta})\tilde{\mf{b}}_r$ is contained in~$\mf{h}(\Delta)$: We apply~\cite[Lemma 9.3]{skodlerackStevens:18} using that
 \[\tilde{r}-r<\tilde{r}-\left\lfloor\frac{\tilde{r}}{2}\right\rfloor=\left\lfloor\frac{\tilde{r+1}}{2}\right\rfloor\leq\left\lfloor\frac{-\ti{k}_0+1}{2}\right\rfloor \]
 where the first inequality is imlied by~$r-1\geq \left\lfloor\frac{\ti{r}}{2}\right\rfloor$.
 This finishes the proof.
\end{proof}

\begin{lemma}\label{lemCDeltaPsia}
 Suppose that~$V=\bigoplus_kV^k$ is a splitting which refines the associated splittings of~$\Delta$ and~$\Delta'$, suppose that~$\C(\Delta(1+))$ is equal to~$\C(\Delta'(1+))$ and that~$\C(\Delta_k)$ is equal to~$\C(\Delta'_k)$ for all indexes~$k$. Let~$a\in\mf{a}_{-r-1}\cap\prod_k\End_D(V^k)$,~$\theta\in\C(\Delta)$
 and~$\theta'\in\C(\Delta')$ be given such that~$\theta_k$ coincides with~$\theta'_k\psi_{a_k}$ for all indexes~$k$. 
 Then~$\C(\Delta)=\C(\Delta')\psi_a$. 
\end{lemma}

\begin{proof}
The group $H(\Delta)$ is the same as~$H(\Delta')$ by Proposition~\ref{lemSemisimplecharacterDeterminesTheGroupOneLevelLower}.
We show that~$\Delta''=[\Lambda,n,r,\beta'+a]$ is equivalent to a semisimple stratum. Then~$\C(\Delta'')$ is equal to~$\C(\Delta')\psi_a$ and it has a non-trivial intersection 
with~$\C(\Delta)$ and Proposition~\ref{propEqualityByNonTrivalIntersection} would finish the proof. 
Now let~$s'$ be a tame corestiction on~$A$ relative to~$E'(1)=F[\beta'(1)]|F$. Then, for every~$k$, the stratum~$[j_{1^kE'(1)}(\Lambda^k),r+1,r,1^ks'(\beta'+a-\beta'(1))1^k]$ is equivalent to a 
semisimple stratum, because~$[j_{1^k E'(1)}(\Lambda^k),r+1,r,1^k s'(\beta'-\beta'(1))1^k]$ is and~$s'(a_k)$ is congruent to an element of~$1^k E'(1)$ 
because~$\psi_{\beta'(1)_k-\beta'_k}\theta'_k\psi_{a_k}\in\C(\Delta'(1)(1-)_k)$. Thus,~$\Delta''$ is equivalent to a semisimple stratum by Proposition~\ref{propDerived}.
\end{proof}

\begin{corollary}\label{corDeltakNonTrivialIntersec}
 Suppose that~$V=\bigoplus_kV^k$ is a splitting which refines the associated splittings of~$\Delta$ and~$\Delta'$, and suppose that there are semisimple characters~$\theta\in\C(\Delta)$ and~$\theta'\in\C(\Delta')$ such that for every index~$k$ the characters~$\theta_k$ and~$\theta'_k$ coincide. Then~$\C(\Delta)$ is equal 
 to~$\C(\Delta')$. 
\end{corollary}

\begin{proof}
 This follows inductively from Lemma~\ref{lemCDeltaPsia} with~$a=0$. 
\end{proof}
%
%
%

\subsection{Matching theorem}

In this section we are given two semisimple strata~$\Delta$ and~$\Delta'$ with~$r=r'$ and~$n=n'$ and the same period. 

\begin{proposition}\label{propDiagonalizationSemisimpleCharacters}
Suppose that~$\Delta$ and~$\Delta'$ have the same associated splitting and suppose~$\theta\in\C(\Delta)$ and~$\theta'\in\C(\Delta')$ intertwine by an element 
of~$(\prod_i A^{ii})^\times$.  
Then there is a stratum~$\Delta''$ with the same splitting as~$\Delta$ such that~$\Lambda''=\Lambda$,
 $n''=n$,~$r''=r$ and~$\C(\Delta)=\C(\Delta'')$ and~$\beta''_i$ has the same minimal polynomial as~$\beta'_i$, and~$\theta=\tau_{\Delta',\Delta''}(\theta')$.  
\end{proposition}

\begin{proof}
 We underline at first that~$\Delta_i$ and~$\Delta'_i$ have the same interia degree and the same ramification index by Proposition~\ref{propEqualDegreeFromIntertwining}.
 By Corollary~\ref{corconsructingEmbeddingsFromNumericalData} there are simple strata~$\tilde{\Delta}_i$ with~$\Lambda^i=\tilde{\Lambda}^i$ and $\ti{r}_i=r$ such 
 that~$\ti{\beta}_i$ is a conjugate of~$\beta_i'$ and which has the same embedding type as~$\Delta_i$. Now, by Proposition~\ref{propTranstivityOfIntertwiningSimpleCharacters}, the characters~$\tau_{\Delta'_i,\ti{\Delta}_i}(\theta'_i)$ and~$\theta_i$ intertwine and, by Proposition~\ref{propIntImplConjSimpleChar}, there is an element~$g_i$ of the normalizer 
 of~$\Lambda$ which conjugates~$\tau_{\Delta'_i,\ti{\Delta}_i}(\theta'_i)$ to~$\theta_i$. Define~$\Delta''$ as the direct sum of the strata~$[\Lambda^i,n,r,g_i\ti{\beta}_ig_i^{-1}]$. 
 Then,~$\tau_{\Delta',\Delta''}(\theta')$ and~$\theta$ have the same block restrictions. And thus, by Corollary~\ref{corDeltakNonTrivialIntersec}, the 
 set of semisimple characters~$\C(\Delta)$ and~$\C(\Delta'')$ coincide and~$\theta$ is the transfer of~$\theta'$ from~$\Delta'$ to~$\Delta''$. 
\end{proof}

\begin{theorem}{\textbf{(Translation principle, see~\cite[Theorem 3.3]{secherreStevensVI:12} for simple strata)}}\label{thmTranslation}
 Suppose that~$\tilde{\Delta}$ is a non-null semisimple stratum which has the same associated splitting,~$V=\bigoplus_{j\in J}V^j$, and the same lattice sequence as~$\Delta(1)$, such that~$\tilde{r}=r-1$ and~$\C(\Delta(1))$ and~$\C(\tilde{\Delta})$ coincide.
 Then there is a semisimple stratum~$\Delta'$  and an element~$u\in (1+\mf{m}(\ti{\Delta}))\cap\prod_{\ti{j}}A^{\ti{j},\ti{j}}$ satisfying~$\Delta'(1)=u.\tilde{\Delta}$, such that~$\C(\Delta')$ is equal to~$\C(\Delta)$.  
\end{theorem}

Theorem~\ref{thmTranslation} has been proven for the case when~$\Delta$ and~$\Delta(1)$ are simple in~\cite[Theorem 3.3]{secherreStevensVI:12} and in the 
split case in~\cite[Theorem 9.16]{skodlerackStevens:18}.

\begin{proof}
 This proof follows the concept of~\cite[Theorem 9.16]{skodlerackStevens:18}. For the start let us firstly treat the case where~$\Delta(1)$ is simple. By 
 Theorem~\ref{thmDiagonalization} we can always replace~$\ti{\Delta}$ by an equivalent stratum. 
 By the simple version of the translation principle, see~{\it loc.cit.}, we can assume that~$\Delta(1)(1-)$ and~$\ti{\Delta}(1-)$ give the same set of simple characters. 
 Let us take tame corestrictions~$s$ and~$\tilde{s}$ on~$A$ relative to~$F[\beta(1)]|F$ and~$F[\beta]|F$, respectively, as in~\cite[Lemma 3.5]{secherreStevensVI:12}.
 The derived stratum~$\partial_{\beta(1),s}(\Delta)$ is equivalent to a semisimple stratum, by Proposition~\ref{propDerived}. Write~$c:=\beta-\beta(1)$. Thus, 
 by Proposition~\ref{propCriteriaFundtoSemisimplStratum} and~\cite[Lemma 3.5(iii)(iv)]{secherreStevensVI:12}, the stratum~$[j_{\tilde{E}}(\Lambda),r,r+1,\tilde{s}(c)]$
 is equivalent to a semisimple stratum, say with splitting,~$(1''^i)_i$, and therefore, by Proposition~\ref{propDerived}, the stratum~$[\Lambda,n,r,\tilde{\beta}+\sum_{i}1''^i c 1''^i ]$ is equivalent to a semisimple stratum~$\Delta''$ with the associated splitting~$(1''^i)_i$ and Lemma~\ref{lemSimpletameCorConjSplitting} provides an element~$v$ of~$1+\mf{m}(\ti{\Delta})$ which conjugates~$\beta''$ to~$\ti{\beta}+c$ modulo~$\mf{a}_{-r}$. We define~$\Delta'$ as~$v.\Delta''$ and we have
 \[\C(\Delta)=\C(\Delta(1)(1-))\psi_c=\C(\ti{\Delta}(1-))\psi_c=\C(\Delta').\]
  The existence of~$u$ follows now from Theorem~\ref{thmDiagonalization}.
 Suppose now that~$\ti{\Delta}$ is not simple. We apply the simple case block-wise to obtain a semisimple stratum~$\Delta'$ and a desired element~$u$ such that there is a defining sequence for~$\Delta'$ with~$u.\ti{\Delta}$ as~$\Delta'(1)$, and such that~$C(\Delta'_{\ti{j}})$ and~$C(\Delta_{\ti{j}})$ coincide. 
 We apply Corollary~\ref{corConjugateAssSplittingsOfStrataWithSameSemisimpleCharacters} block-wise for the simple blocks of~$\ti{\Delta}$ to conjugate the associated 
 splitting of~$\beta$ to the one of~$\beta'$ by an element of~$S(\Delta)\cap\prod_{\ti{j}}A^{\ti{j},\ti{j}}$. For a semisimple character~$\theta$ for~$\Delta$ and an
 extension~$\theta'\in\C(\Delta')$ of~$\theta(1+)$ there is an element~$a\in\mf{a}_{-r-1}\cap\prod_{i'}A^{i',i'}$, such that~$\theta\psi_a$ and~$\theta'$ coincide in~$\C(\Delta_{i'})$,~$i'\in I'$. Now Lemma~\ref{lemCDeltaPsia} finishes the proof. 
\end{proof}

We now state the matching theorem.  

\begin{theorem}\label{thmMatching}
 Suppose~$\theta\in\C(\Delta)$ and~$\theta'\in\C(\Delta')$  are intertwining semisimple characters. Then there is a unique bijection~$\zeta:I\ra I'$ such that 
 \begin{enumerate}
  \item $\dim_DV^i=\dim_DV^{\zeta(i)}$, and 
  \item there is a~$D$-linear isomorphism~$g_i$ from~$V^i$ to~$V^{\zeta(i)}$ such that~$g_i.\theta_i$ and~$\theta'_{\zeta(i)}$ intertwine by a~$D$-linear 
  automorphism of~$V^{\zeta(i)}$,\label{thmMatchingAss2}
 \end{enumerate}
 for all indexes~$i\in I$.
\end{theorem}

The map~$\zeta$ is called a~\emph{matching} for~$\theta$ and~$\theta'$ from~$\Delta$ to~$\Delta'$. 

\begin{proof}
 We prove the existence by an induction along~$n-r$. If~$n=r$ there is nothing to prove, and the case~$n=r+1$ follows from the theory of strata, see 
 Proposition~\ref{propMatchingGLDStrata}. Suppose now that~$n$ is greater than~$r+1$. By induction hypothesis we have a matching for~$\theta(1+)$ 
  and~$\theta'(1+)$ from~$\Delta(1)$ to~$\Delta'(1)$ and by Proposition~\ref{propDiagonalizationSemisimpleCharacters} there is a semisimple stratum~$\ti{\Delta}$ with the same associated splitting as~$\Delta'(1)$ and the same set of semisimple characters such that~$\ti{\beta}$ is conjugate to~$\beta(1)$ by an element of~$G$ such that~$\theta'(1+)$ is the transfer of~$\theta(1+)$ from~$\Delta(1)$ to~$\ti{\Delta}$. We apply the translation theorem~\ref{thmTranslation} to~$(\Delta',\Delta'(1))$ and~$\ti{\Delta}$ to reduce to the case where~$\beta(1)$ and~$\beta'(1)$ are conjugates by an
  element of~$\prod_j A^{j,j}$. Thus we can assume without loss of generality that~$\beta(1)$ and~$\beta'(1)$ coincide. 
 We apply Lemma~\ref{lemDerivedCharacters}\ref{lemDerivedCharactersAss1} then Proposition~\ref{propMatchingGLDStrata} followed by Lemma~\ref{lemDerivedCharacters}\ref{lemDerivedCharactersAss2} to obtain a matching between~$\theta$ and~$\theta'$. 

 We now prove the uniqueness: Because of the existence of a matching~$\zeta$ we can assume by conjugation that~$\zeta$ is the identity and~$V^i=V'^i$ for all indexes~$i\in I$.
 If there is a second matching then we obtain by Proposition~\ref{propTranstivityOfIntertwiningSimpleCharacters} two indexes~$i_1,i_2\in I$ such that
 $\theta_{i_1}$ and~$\theta_{i_2}$ intertwine by an isomorphism from~$V^{i_1}$ to~$V^{i_2}$. And thus the semisimple characters~$\theta|_{V^{i_1}\oplus V^{i_2}}$
 and~$\theta_{i_1}^{\dag, (0,0)}$ (see Proposition~\ref{propDdag}) intertwine. This is a contradiction because two intertwining semisimple characters have the same number of 
 blocks by the existence part. 
\end{proof}

\begin{definition}
 We call the bijection~$\zeta_{\theta,\Delta,\theta',\Delta'}: I\ra I'$ the matching from~$\theta\in\C(\Delta)$ to~$\theta'\in\C(\Delta')$.
\end{definition}

\begin{corollary}\label{corMatchingForTransfers}
 Suppose that~$\theta\in\C(\Delta)$ and~$\theta'\in\C(\Delta')$ are transfers then~$\zeta_{\Delta,\Delta'}=\zeta_{\theta,\Delta,\theta',\Delta'}$.
\end{corollary}

\begin{proof}
 By Proposition~\ref{propMatchingGLDStrata} there is an element~$g\in\prod_i\End_D(V^i,V'^{\zeta_{\Delta,\Delta'}(i)})$ which intertwines~$\Delta$ with~$\Delta'$ and thus~$\theta$ with~$\theta'$, because they are transfers. 
\end{proof}

\begin{corollary}\label{corEqualDegreesFromIntertwiningSemisimpleCharactersSamer}
 Granted~$e(\Lambda|F)=e(\Lambda'|F)$ and~$r=r'$, if there are two intertwining characters~$\theta\in\C(\Delta)$ and~$\theta'\in\C(\Delta')$ then the strata~$\Delta$ and~$\Delta'$ have the same group level, the same degree
 and the same critical exponent.
 Moreover~$e(E_i|F)=e(E_{\zeta(i)}|F)$ and~$f(E_i|F)=f(E_{\zeta(i)}|F)$ and~$k_0(\Delta_i)=k_0(\Delta'_{\zeta(i)})$ for all indexes~$i$. 
\end{corollary}

\begin{proof}
 This follows directly from Theorem~\ref{thmMatching} and Proposition~\ref{propEqualDegreeFromIntertwining}.
\end{proof}

\subsection{Intertwining semisimple characters of same degree and same group level}

At first we generalize the matching to semisimple characters of the same group level and the same degree. 

\begin{proposition}\label{propMatchingSameDegreeAndSameGroupLevel}
 Suppose that~$\theta\in\C(\Delta)$ and~$\theta'\in\C(\Delta')$ are two intertwining semisimple characters and~$\Delta$,~$\Delta'$ share group level and the degree. 
 Then there is a unique  bijection~$\zeta: I\ra I'$ and a~$D$-linear isomorphism of~$V$ such that~$gV^i=V'^{\zeta(i)}$ and such that~$g.\theta_i$ and~$\theta'_{\zeta(i)}$ intertwine for all indexes~$i\in I$. 
 Further~$e(E_i|F)=e(E_{\zeta(i)}|F)$ and~$f(E_i|F)=f(E_{\zeta(i)}|F)$. If~$\Lambda$ and~$\Lambda'$ have the same period then~$k_0(\Delta)=k_0(\Delta')$ 
 and~$k_0(\Delta_i)=k_0(\Delta'_{\zeta(i)})$ for all~$i\in I$. 
\end{proposition}

\begin{proof}
 We can replace the lattice sequences~$\Lambda$ and~$\Lambda'$ by affine translates and can therefore assume that both lattice sequences have the same period. If~$r\geq r'$, then~$\theta$ and~$\theta'((r-r')+)$ intertwine. 
 Thus there is a matching for~$\theta$ and~$\theta'((r-r')+)\in\C(\Delta'(r-r'))$ by Theorem~\ref{thmMatching}. Thus~$e(\Lambda|E)=e(\Lambda'|E(r-r'))$ and~$E'(r-r')$ and~$E$ have the same~$F$-dimension, by 
 Corollary~\ref{corEqualDegreesFromIntertwiningSemisimpleCharactersSamer}, and 
 therefore~$\Delta'(r-r')$ and~$\Delta'$ have the same degree. Thus~$\Delta'(r-r')$ has the same number of blocks as~$\Delta'$ and~$E'_{i'}|F$ has the
 same inertia degree and the same ramification index as~$E'(r-r')_{i'}|F$, by Proposition~\ref{propSimplePure}.
 Thus~$e(\Lambda|E)=e(\Lambda'|E'(r-r'))=e(\Lambda'|E')$ and thus~$\Delta'((r-r')+)$ is still semisimple, by Corollary~\ref{corSameGrouplevelFromIntertwiningIfSamePeriodandSamer}\ref{corSameGrouplevelFromIntertwiningIfSamePeriodandSamerAssii}, because~$\Delta$ and~$\Delta'$ have the same group level. 
 The matching theorem~\ref{thmMatching} provides a bijection~$\zeta:I\rightarrow I'$ and an invertible element~$g\in\prod_i\End_D(V^i,V^{\zeta(i)})$ such that~$g.\theta_i$ and~$\theta'_{\zeta(i)}((r-r')+)$ intertwine. 
 By Proposition~\ref{propTranstivityOfIntertwiningSimpleCharacters} the characters~$g.\theta_i$ and~$\theta'_{\zeta(i)}$ intertwine. 
 The uniqueness comes from the uniqueness of the matching for~$\theta$ and~$\theta'((r-r')+)$ and the respective strata. The remaining assertions follow from 
 Corollary~\ref{corEqualDegreesFromIntertwiningSemisimpleCharactersSamer} applied on~$\theta$ and~$\theta'((r-r')+)$.
\end{proof}

\begin{corollary}\label{corIntertwiningIsEquivalenceRelation}
 Intertwining is an equivalence relation among all semisimple characters given by semisimple strata of a fixed group level and fixed degree. 
\end{corollary}

\begin{proof}
 This follows directly from Proposition~\ref{propTranstivityOfIntertwiningSimpleCharacters} and Proposition~\ref{propMatchingSameDegreeAndSameGroupLevel}.
\end{proof}

\subsection{Intertwining and Conjugacy for semisimple characters}

In this subsection we fix two semisimple characters~$\theta\in\C(\Delta)$ and~$\theta'\in\C(\Delta')$ which intertwine. We further assume that~$n=n'$ and~$r=r'$. 
The Theorem~\ref{thmMatching} provides a matching~$\zeta_{\theta,\Delta,\theta',\Delta'}:I\ra I'$ from~$(\theta,\Delta)$ to~$(\theta',\Delta')$. 

\begin{lemma}[{see~\cite[Lemma 3.5]{secherreStevensVI:12} for the simple case}]\label{lemEqualResAlgebras}
 Suppose that~$\C(\Delta)=\C(\Delta')$ with~$\Lambda=\Lambda'$. Then the residue algebras~$\kappa_E$ and~$\kappa_{E'}$ coincide.  
\end{lemma}

\begin{proof}
 By Corollary~\ref{corConjugateAssSplittingsOfStrataWithSameSemisimpleCharacters} we can conjugate the associated splittings to each other by an element of~$S(\Delta)$ and thus we can assume that both associated splittings are the same. By~\cite[Lemma 3.5]{secherreStevensVI:12} the residue fields of~$E_i$ and~$E'_i$ coincide which implies 
 the assertion. 
\end{proof}

\begin{lemma}\label{lemMatchingOfResidueAlgSemisimplechar}
The conjugation by~$g\in I(\theta,\theta')$ induces a canonical isomorphism~$\bar{\zeta}$ form the residue algebra of~$E$ to the one of~$E'$, and the isomorphism does not depend on the choice of the intertwining element. 
\end{lemma}

\begin{proof}
 By Theorem~\ref{thmMatching} and Proposition~\ref{propDiagonalizationSemisimpleCharacters} and Lemma~\ref{lemEqualResAlgebras} we can assume that~$\theta$ and~$\theta'$ are transfers. The conjugation with~$S(\Delta)$  fixes~$\kappa_E$ and similar fixes the conjugation with~$S(\Delta')$ the algebra~$\kappa_{E'}$. 
 Thus we only need to show that the conjugation with an element of~$C_{A^\times}(\beta)$ defines a map from~$\kappa_E$ to~$\kappa_{E'}$ which does not depend on the 
 chioce of an element of~$C_{A^\times}(\beta)$. This is done in Lemma~\ref{lemMatchingOfResidueAlg}. 
\end{proof}

\begin{notation}
We denote the map of Lemma \ref{lemMatchingOfResidueAlgSemisimplechar} by~$\bar{\zeta}$ or~$\bar{\zeta}_{\theta,\Delta,\theta',\Delta'}$ and call it the matching of the residue algebras from~$(\theta,\Delta)$ to~$(\theta',\Delta')$. And we write~$(\zeta,\bar{\zeta})$ and call it the matching pair.
\end{notation}

We are now able to formulate and prove the first main theorem of this article about intertwining and conjugacy of semisimple characters. 

\begin{theorem}[\textbf{1st Main Theorem}]\label{thmIntertwiningImplConjugacy}
 Let~$(\zeta,\bar{\zeta})$ be the matching pair from~$(\theta,\Delta)$ to~$(\theta',\Delta')$.
 If there is a element of~$t\in G$ and an integer~$s$ such that~$t\Lambda^i_j=\Lambda^{\zeta(i)}_{j+s}$ for all~$i\in I$ and all integers~$j$ such that the conjugation
 with~$t$ induces the map~$\bar{\zeta}|_{\kappa_{E_D}}$ then there is an element~$g$ of~$ G$ such that~$g.\theta_i=\theta'_{\zeta(i)}$ and~$g^{-1}t\in P(\Lambda)$.
 In particular~$g.\theta$ and~$\theta'$ coincide. 
\end{theorem}

\begin{proof}
 We can conjugate~$\theta$ and~$\Delta$ with~$t$ and obtain the case  where~$\bar{\zeta}|_{\kappa_{E_D}}$ is induced by~$1$. 
 Thus we can assume~$t=1$ and~$\bar{\zeta}|_{\kappa_{E_D}}=\id_{\kappa_{E_D}}$. 
 By Corollary~\ref{corDeltakNonTrivialIntersec} it is enough to prove the theorem for the simple case. In this case we have to show that~$\theta$ is conjugate 
 to~$\theta'$ by an element of~$P(\Lambda)$. By Proposition~\ref{propDiagonalizationSemisimpleCharacters} we can assume (by changing~$\beta'$) that~$\theta'$ is the transfer of~$\theta$ from~$\Delta$ to~$\Delta'$. Then~$\bar{\zeta}_{\Delta,\Delta'}=\bar{\zeta}=\id$ by Corollary~\ref{corMatchingForTransfers}. 
 Thus there is an element of~$P(\Lambda)$ which conjugates~$\Delta$ to~$\Delta'$ by Theorem~\ref{thmIntConjSesiTiG} and this element conjugates~$\theta$ to the transfer, i.e. to~$\theta'$.
 \end{proof}

\section{Endo-classes for GL}\label{secEndoClass}
In this section we generalize the theory of endo-classes to semisimple characters of general linear groups. 
The theory of endo-classes for simple characters can be found in~\cite{broussousSecherreStevens:12} and
for the split case semisimple endo-classes are fully studied in~\cite{KSS:21}.
The section is structured into four parts. 
At first we study restriction maps between sets of semisimple characters, secondly we 
define transfer between  strata on different vector spaces and then between different division algebras, then we define and study semisimple endo-classes. This will be used in~\S\ref{secEndoPar}
to classify intertwining classes of full semisimple characters using so-called endo-parameters which are finite sums of weighted simple endo-classes. 

%
%
%

\subsection{Restrictions of semisimple characters}

\begin{proposition}\label{propSurjRestrictions}
 Let~$\Delta$ be a semisimple stratum with associated splitting~$\oplus_{i\in I}V^i$ (as usual) and suppose further that~$\Delta$ is split by~$V^0 \oplus V^1$ with a non-zero vector space~$V^0$, say~$\Delta=\Delta_0\oplus \Delta_1$. Then
 \begin{enumerate}
  \item the restriction map~$\C(\Delta)\stackrel{res_{\Delta,\Delta_0}}{\longrightarrow}\C(\Delta_0)$ is surjective. \label{thmSurjRestrictions.i}
  \item \label{thmSurjRestrictions.ii} if for all indexes~$i\in I$ the intersection of~$V^i$ with~$V^0$ is non-zero then~$res_{\Delta,\Delta_0}$ 
  is bijective.
 \end{enumerate}
\end{proposition}

The class of restriction maps~$res_{\Delta,\Delta_0}$ in Proposition~\ref{propSurjRestrictions}
and the class of restriction maps given in Proposition~\ref{propRestrictionMapIsBijective} only intersect in the case where~$\Delta=\Delta_0$ and~$j=0$. In this case the restriction map is the identity.   

\begin{lemma}[{\cite{secherreStevensIV:08}~Theorem~2.17(ii),~Proposition~2.10}]\label{lemBSS2.17iiBijectiveRestrictionSimpleCase}
 Under the assumptions of Proposition~\ref{propSurjRestrictions} if~$\Delta$ is simple then~$res_{\Delta,\Delta_0}$ is bijective.
\end{lemma}

\begin{lemma}[{\cite{KSS:21}~Lemma 12.4}]\label{lemTriangleToSquareExtension}
 Let~$\Delta$ be a semisimple stratum split by~$V^0\oplus V^1$ and let~$\ti{\theta}\in \C(\Delta(1+))$ and~$\theta_1\in\C(\Delta_0)$ be semisimple characters
 such that their restrictions to~$H(\Delta_0(1+))$ coincide. Then there is a semisimple character~$\theta\in\C(\Delta)$ such that~$\theta(1+)=\ti{\theta}$ and with restriction~$\theta_0$ on~$H(\Delta_0)$. 
\end{lemma}

\begin{proof}
 The proof goes along an induction on the critical exponent and is completely the same as in~\loccit, except with a small change in the 
 case~$r<\lfloor\frac{-k_0(\Delta)}{2}\rfloor$ where we use Lemma~\ref{lemBSS2.17iiBijectiveRestrictionSimpleCase} instead of transfers. In fact this is no real change, 
 because Lemma~\ref{lemBSS2.17iiBijectiveRestrictionSimpleCase} was proven by using transfers for simple characters on different vector spaces. 
\end{proof}

\begin{proof}[Proof of Proposition~\ref{propSurjRestrictions}]
 \begin{enumerate}
  \item The proof uses an induction along~$r$. If~$r=n$ then~$\C(\Delta)$ and~$\C(\Delta_0)$ are singletons. Thus the restriction map is surjective. 
  In the case~$r<n$ take a semisimple character~$\theta_0\in\C(\Delta_0)$ and use the induction hypothesis to obtain an extension~$\ti{\theta}\in\C(\Delta(1+))$ 
  of~$\theta_0(1+)$. By Lemma~\ref{lemTriangleToSquareExtension} there exists a common extension of~$\ti{\theta}$ and~$\theta_0$ into~$\C(\Delta)$. This finishes the proof of the first part.
  \item Statement~\ref{thmSurjRestrictions.i} provides the surjectivity. For the injectivity consider the following diagram 
  \[\begin{matrix}
   \C(\Delta) & \rightarrow &\C(\Delta_0)\\
   \lhookdownarrow & \circlearrowleft & \lhookdownarrow\\
   \prod_{i\in I} \C(\Delta_i) & \stackrel{\sim}{\longrightarrow} & \prod_{i\in I} \C(\Delta_{0,i})\\
  \end{matrix},\]
  where the bottom map is a bijection by Lemma~\ref{lemBSS2.17iiBijectiveRestrictionSimpleCase}. This proves the injectivity. 

 \end{enumerate}
\end{proof}

We need a notion of the direct sum of strata 
with lattice sequences of possibly different~$o_D$-periods. 
Let~$\Delta$ and~$\Delta'$ be two strata. 
Write~$e(\Lambda,\Lambda')$ for the fraction~$\frac{e(\Lambda|F)}{\gcd(e(\Lambda|F),e(\Lambda'|F))}$. Then~$e(\Lambda,\Lambda')\Lambda$ and~$e(\Lambda',\Lambda)\Lambda'$ have the same~$o_F$-period and we define the direct sum of~$\Delta$ and~$\Delta'$ via
\[\Delta\oplus\Delta':=[e'\Lambda\oplus e\Lambda',\max(ne',n'e),\max(re',r'e),\beta\oplus\beta'],\ e=e(\Lambda,\Lambda'),\ e'=e(\Lambda',\Lambda).\]
%
This is consistent with the direct sum of strata given in~\S \ref{subsecSemisimpleStrataFirstDef}.
Taking the above definition into account we have to be a little bit careful with restrictions. Suppose~$\Delta$ and~$\Delta'$ are two semisimple 
strata of the same period and suppose~$r=r'+1$. The restriction of an 
element~$\theta\in\C(\Delta\oplus\Delta')$ to~$\End_D(V')$ is an element of~$\C(\Delta'(1+))$. Thus we need to 
extend~$\theta|_{H(\Delta'(1+))}$ to~$H(\Delta')$. Proposition~\ref{propRestrictionMapIsBijective} states that this extension is unique if~$\Delta'$ 
and~$\Delta'(1+)$ have the same group level.

%

For a positive integer~$k$ we define the~$k$-multiple
~$k\Delta$ of a stratum~$\Delta=[\Lambda,n,r,\beta]$
to be the stratum~$[k\Lambda,kn,kr,\beta]$. 

\begin{definition}\label{defRestriction}
Let~$\Delta$ and~$\Delta'$  be semisimple strata such that~$(\Delta\oplus\Delta')|_V$ and~$\Delta$ have same group level. We define the restriction
\[res_{\Delta\oplus\Delta',\Delta}:\ \C(\Delta\oplus\Delta')\rightarrow\C(\Delta)\] as the map
\[res_{e'\Delta,(e'\Delta)(k-e'r)}^{-1}\circ res_{\Delta\oplus\Delta',(e'\Delta)(k-e'r).},\]
where~$e=e(\Lambda,\Lambda'),e'=e(\Lambda',\Lambda)$ and~$k=\max(e'r,er')$. 
This is the usual restriction followed by extension. We call these maps~\emph{external restrictions}, and we call the usual restrictions as~\emph{internal restrictions}. Thus every internal restriction can be interpreted as an external one, but not every external restrictions comes from an internal one. 
\end{definition}

The~$\dag$-construction has the following influence on external restrictions.

\begin{proposition}\label{propIndofRestMapFromCopyOfSimpleBlock}
 Let~$\Delta_i$,~$i\in I$ be semisimple strata such that there are two different indexes~$i_1,i_2$ such that~$\Delta_{i_1}$ is equal to~$\Delta_{i_2}$. 
 Suppose further that~$(\oplus_I\Delta_i)|_{V^{i_1}}$ has the same group level as~$\Delta_{i_1}$. Then 
 the maps~$res_{\oplus_I\Delta_i,\Delta_{i_1}}$ and~$res_{\oplus_I\Delta_i,\Delta_{i_2}}$ coincide. 
\end{proposition}

\begin{proof}
 At first the restriction maps~$res_{\Delta_{i_1}\oplus\Delta_{i_2},\Delta_{i_1}}$ and~$res_{\Delta_{i_1}\oplus\Delta_{i_2},\Delta_{i_2}}$ coincide because they have the same inverse
 which is the~$\dag$-map~$( )^{\dag,(0,0)}$, see~\S\ref{subsecDag}. The statement follows now from 
 \[res_{\oplus_I\Delta_i,\Delta_{i_j}}=res_{\Delta_{i_1}\oplus\Delta_{i_2},\Delta_{i_j}}\circ res_{\oplus_I\Delta_i,\Delta_{i_1}\oplus\Delta_{i_2}}\]
 for~$j=1,2$.
\end{proof}

\subsection{Generalization of transfer}
The idea of transfer is the following. Let us illustrate this roughly for the simple case: Say~$\Delta$ is a simple stratum with defining sequence~$(\Delta(j))_j$. 
A character~$\theta\in\C(\Delta)$ is constructed inductively using the characters~$\psi_{\beta(j)}$, determinants and characters~$\chi_j$ on~$E(j)^\times$. 
Now, given another simple stratum~$\Delta'$ of the same group level as~$\Delta$ such that~$\beta'$ has the same minimal polynomial as~$\beta$ one just takes the 
same data~$\chi_j$ to construct a character~$\theta'$. This~$\theta'$ is called the transfer of~$\theta$ from~$\Delta$ to~$\Delta'$. 
This idea  defines a bijection between~$\C(\Delta)$ and~$\C(\Delta')$. 
This construction does not need a fixed ambient vector space for both strata. And in this section we are going to generalize it to semisimple strata in a rigorous way.

\begin{definition}\label{defEndoEquivalentStrata}
 Two semisimple strata~$\Delta$ and~$\Delta'$ (possibly on different vector spaces and possibly over different skew-fields) which have the same group level and the same degree are called endo-equivalent 
 if there is a bijection~$\zeta$ from~$I$ to~$I'$ such that, for all~$i\in\I$, the direct sum~$\Res_F(\Delta_i)\oplus\Res_F(\Delta'_{\zeta(i)})$ is equivalent to a simple stratum. 
\end{definition}

\begin{proposition}\label{propEndoEquivalenceOfStrataIsEquivalenceRelation}
 \begin{enumerate}
  \item If~$\Delta$ and~$\Delta'$ are endo-equivalent then the map~$\zeta$ from Definition~\ref{defEndoEquivalentStrata} is uniquely determined, i.e. there is no other map
 ~$\zeta'$ from~$I$ to~$I'$ such that~$\Res_F(\Delta_i)\oplus\Res_F(\Delta'_{\zeta'(i)})$ is equivalent to a simple stratum, for all indexes~$i\in I$.  
  \item On semisimple strata endo-equivalence is an equivalence relation. 
 \end{enumerate}
\end{proposition}

The equivalence class of a semisimple stratum under endo-equivalence is called the endo-class of the stratum. We define the group level and the degree of an endo-class of a semisimple stratum to be the respective group level and degree of the semisimple stratum. 
We will denote the map~$\zeta$ in Definition~\ref{defEndoEquivalentStrata} by~$\zeta_{\Delta,\Delta'}$.

\begin{proof}
 This result follows directly from the transitivity of endo-equivalence for simple strata:
 Corollary~\ref{corEndoEquivalenceForSimpleStrata}\ref{corEndoEquivalenceForSimpleStrataAss.ii} and from~\cite[(1.9)]{bushnellHenniart:96}, using a~$\dag$-construction.
\end{proof}

\subsubsection{Transfer for a fixed division algebra~$D$}

\begin{definition}
 Given semisimple characters~$\theta^j\in\C(\Delta^j)$,~$j=1,\ldots,l$, we define the tensor product~$\theta^1\otimes\ldots\otimes\theta^l$ with respect to~$(\Delta^j)_{j=1,\ldots,l}$ to be the complex valued map
 on~$H(\bigoplus_j\Delta^j)$ which is defined under the Iwahori decomposition of~$\bigoplus_jV^j$ via
 \[(\otimes_j\theta^j)(u_-xu_+)=\prod_j\theta(x_j).\] 
 This definition depends on the choice of the strata.
\end{definition}

\begin{definition}\label{defTransferD}
We call a pair~$(\theta,\Delta)$ with~$\theta\in\C(\Delta)$ a parametrized semisimple characters. Given two parametrized semisimple characters~$(\theta,\Delta)$ and~$(\theta',\Delta')$ we call~$(\theta',\Delta')$ a transfer of~$(\theta,\Delta)$ if~$\Delta$ and~$\Delta'$ are endo-equivalent and~$\theta\otimes\theta'\in\C(\Delta\oplus\Delta')$
\end{definition}

\begin{proposition}\label{propFirstRemarksOnDefOftransferD}
Suppose~$\Delta$ and~$\Delta'$ are endo-equivalent semisimple strata and~$\theta\in\C(\Delta)$. 
\begin{enumerate}
 \item\label{propFirstRemarksOnDefOftransferDi} There exists exactly one semisimple character~$\theta'\in\C(\Delta')$ such that~$(\theta',\Delta')$ is a transfer of~$(\theta,\Delta)$.
 \item\label{propFirstRemarksOnDefOftransferDii} If~$\Delta''$ is a semisimple stratum equivalent to~$\Delta$ and~$\Delta'''$ is a semisimple stratum equivalent to~$\Delta'$ then~$(\theta',\Delta''')$ is the transfer of~$(\theta,\Delta'')$.
 \item\label{propFirstRemarksOnDefOftransferDiii} Suppose~$V=V'$ and~$I(\Delta,\Delta')\neq \emptyset$ then~$I (\Delta,\Delta')\subseteq I(\theta,\theta')$.
\end{enumerate}
\end{proposition}

\begin{proof}
\begin{enumerate}
 \item The desired~$\theta'$ is the image of~$\theta$ under~$res_{\Delta\oplus\Delta',\Delta'}\circ res_{\Delta\oplus\Delta',\Delta}^{-1}$. 
 \item The groups~$H(\Delta\oplus\Delta')$ and~$H(\Delta''\oplus\Delta''')$ differ in general, and to avoid confusion we write subscripts on~$\theta\otimes\theta'$.  We have to 
 show~$(\theta\otimes\theta')_{\Delta'',\Delta'''}\in\C(\Delta''\oplus\Delta''')$. 
 Consider~$\ti{\theta}:=\tau_{\Delta\oplus\Delta',\Delta''\oplus\Delta'''}((\theta\otimes\theta')_{\Delta,\Delta'})$. This is a character which decomposes under the Iwahori decomposition with respect to~$V\oplus V'$ with restrictions~$\ti{\theta}|_V=\theta$ and~$\ti{\theta}_{V'}=\theta'$. 
 Thus~$(\theta\otimes\theta')_{\Delta'',\Delta'''}$ is a character equal to~$\ti{\theta}$. 
 
 \item  We apply~$\tau_{\Delta\oplus\Delta,\Delta\oplus\Delta'}$ on~$\theta\otimes\theta$ to 
 obtain a semisimple character~$\theta\otimes\ti{\theta}'\in\C(\Delta\oplus\Delta')$ such that
 \[I(\Delta\oplus\Delta,\Delta\oplus\Delta')\subseteq I(\theta\otimes\theta,\theta\otimes\ti{\theta}').\]
 We have~$\theta\otimes\ti{\theta}'\in\C(\Delta\oplus\Delta')$ and~$\theta\otimes\theta'\in\C(\Delta\oplus\Delta')$ and the uniqueness statement in~\ref{propFirstRemarksOnDefOftransferDi} implies~$\theta'=\ti{\theta}'$ and therefore~$I(\Delta,\Delta')$ intertwines~$\theta$ with~$\theta'$. 
\end{enumerate}
\end{proof}

We call~$\theta'$ a transfer of~$\theta$ from~$\Delta$ to~$\Delta'$ if~$(\theta',\Delta')$ is a transfer of~$(\theta,\Delta)$. 
We denote the map~$res_{\Delta\oplus\Delta',\Delta'}\circ res_{\Delta\oplus\Delta',\Delta}^{-1}$ by~$\tau_{\Delta,\Delta'}$ and call it a~\emph{transfer map}.

\begin{proposition}\label{propFirstPropertiesOfGeneralizedTransfer}
 Transfer is an equivalence relation on the class of parametrized semisimple characters.  External restrictions maps are transfer maps. 
\end{proposition}

\begin{proof}
 Suppose we are given three endo-equivalent semisimple strata~$\Delta,\Delta'$ and~$\Delta''$. 
 Consider the following commutative diagram consisting of external restriction maps, which are here bijections by 
 Proposition~\ref{propSurjRestrictions}\ref{thmSurjRestrictions.ii}. 
 \[\begin{matrix}
     \C(\Delta) & \leftarrow & \C(\Delta\oplus\Delta') & \ra & \C(\Delta') \\
     \uparrow & \nwarrow & \uparrow & \nearrow & \uparrow \\
     \C(\Delta\oplus\Delta'')  & \leftarrow & \C(\Delta\oplus\Delta'\oplus\Delta'') & \rightarrow & \C(\Delta'\oplus \Delta'')\\
     & \searrow & \downarrow  & \swarrow & \\
     & & \C(\Delta'') & &  \\
   \end{matrix}
\]
The commutativity of this diagram proves that~$\tau_{\Delta',\Delta''}\circ\tau_{\Delta,\Delta'}$ is equal to~$\tau_{\Delta,\Delta''}$, and this finishes the proof of the 
first assertion.

For the second assertion, given semisimple strata~$\Delta$ and~$\Delta'$ the following map:
\[\C(\Delta\oplus\Delta')\stackrel{res^{-1}}{\longrightarrow} \C(\Delta\oplus\Delta'\oplus\Delta)\stackrel{res}{\longrightarrow}\C(\Delta),\]
is the restriction map, because the restriction map on the right does not depend on the copy of~$\Delta$ which is used for the restriction, by 
Proposition~\ref{propIndofRestMapFromCopyOfSimpleBlock}. 
\end{proof}

\begin{remark}\label{remEquivalenceOfDefintionsOfTransfer}
In~\cite{secherreStevensIV:08} the authors extended  the notion of transfer to simple strata over~$D$. It is transitive and they proved in~\emph{loc.cit.}~Theorem 2.17 that 
internal restrictions are transfer maps, according to their definition of transfer, and it extends to external restriction maps by~\emph{loc.cit.}~Theorem 2.13. 
Thus their notion of transfer and the notion of transfer in Definition~\ref{defTransferD} coincide, 
by Proposition~\ref{propFirstPropertiesOfGeneralizedTransfer}.
\end{remark}

\subsubsection{Transfers between different Brauer classes}
The key idea to move between sets of semisimple characters for different simple algebras over~$F$ with possibly different Brauer class is from~\cite{secherreI:04} where he
defined the transfer for simple characters by extension of scalers to~$L$ followed by the transfer map for the split case. We follow this procedure to generalize it to semisimple strata. 
But at first let us state an unramified extension theorem for semisimple characters whose proof is exactly as in part~\ref{propDefSemisimpleCharInductiveNonSplit.i} of 
Proposition~\ref{propDefSemisimpleCharInductiveNonSplit}. 

\begin{proposition}\label{propExtensionFromDeltaToDeltaotimesL}
 Let~$\Delta$ be a semisimple stratum and~$\ti{L}$ be an unramified extension of~$F$ then the restriction map from~$\C(\Delta\otimes \ti{L})$ into~$\C(\Delta)$ is surjective.
\end{proposition}

\begin{definition}\label{defGeneralizedTransferDifferentDandD'}
Suppose now that~$\Delta$ and~$\Delta'$ for~$D$ and~$D'$, respectively, are two endo-equivalent semisimple strata and take an unramified field extension~$\ti{L}|F$ 
which splits~$D$ and~$D'$. We define the (generalized) transfer map from~$\C(\Delta)$ to~$\C(\Delta')$ as follows: For~$\theta\in\C(\Delta)$ we choose an 
extension~$\theta_{\ti{L}}$ of~$\theta$ into~$\C(\Delta\otimes\ti{L})$. We define
\[\tau_{\Delta,\Delta'}(\theta):=\tau_{\Delta\otimes \ti{L},\Delta'\otimes\ti{L}}(\theta_{\ti{L}})|_{H(\Delta')}.\]
\end{definition}

\begin{lemma}\label{lemWellDefGeneralizedTransferDifferentDandD'}
 The map~$\tau_{\Delta,\Delta'}$ is well defined, i.e. it does not depend on the choice of~$\theta_{\tilde{L}}$ and the choice of~$\tilde{L}$ and
  it is consistent with the Definition~\ref{defTransferD}. 
\end{lemma}

\begin{proof}
 \begin{enumerate}
  \item \label{lemWellDefGeneralizedTransferDifferentDandD'.i} We start with the consistence to Definition~\ref{defTransferD}. Suppose~$D=D'$ and take~$\theta\in\C(\Delta)$ and an 
  extension~$\theta_{\ti{L}}\in\C(\Delta\otimes\ti{L})$. The transfer~$\theta'_{\ti{L}}$ of~$\theta_{\ti{L}}$ is the unique semisimple character 
  of~$\C(\Delta'\otimes\ti{L})$ such that~$\theta_{\ti{L}}\otimes\theta'_{\ti{L}}$ is an element of~$\C((\Delta\oplus\Delta')\otimes\ti{L})$. 
  Thus,~$\theta_{\ti{L}}\otimes \theta'_{\ti{L}}|_{H(\Delta\oplus\Delta')}$ is an element of~$\C(\Delta\oplus\Delta')$.
  \item By part~\ref{lemWellDefGeneralizedTransferDifferentDandD'.i} we have for~$\ti{L}\subseteq\ti{L}'$ that~$\tau_{\Delta\otimes\ti{L},\Delta'\otimes\ti{L}}$ can be 
  verified using~$\tau_{\Delta\otimes\ti{L}',\Delta'\otimes\ti{L}'}$. And thus the definition of~$\tau_{\Delta,\Delta'}$ does not depend on the choice of~$\ti{L}$. 
  \item In the case where~$\Delta$ and~$\Delta'$ are simple we can use a~$\dag$-construction to reduce to the simple strict case. The statement is then~\cite[Theorem 3.53]{secherreI:04}, more precisely it proves the independence of the choice of~$\theta_L$.
  \item In the semisimple case we have a reduction to the simple case, by the formula
  \[\tau_{\Delta\otimes \ti{L},\Delta'\otimes\ti{L}}(\theta_{\ti{L}})=\otimes_i\tau_{\Delta_i\otimes \ti{L},\Delta'_{\zeta(i)}\otimes\ti{L}}((\theta_{\ti{L}})_i)\in\C(\bigoplus_{i\in I}\Delta'^{\zeta(i)}).\]
\end{enumerate}

\end{proof}

Thus, this defines the~\emph{ generalized transfer} of~$\theta$ from~$\Delta$ to~$\Delta'$. The definition does only depend on the equivalence class of~$\Delta$ as this is
the case for~$\tau_{\Delta\otimes \ti{L},\Delta'\otimes\ti{L}}$ by Proposition~\ref{propFirstRemarksOnDefOftransferD}\ref{propFirstRemarksOnDefOftransferDii}.

\subsection{Potentially semisimple characters}

\begin{definition}\label{defPssCharacter}
For an endo-equivalence class~$\mf{E}$ of semisimple strata we denote by~$\C(\mf{E})$ the class of all semisimple characters~$\theta$ for which there is a 
stratum~$\Delta\in\mf{E}$ such that~$\theta\in\C(\Delta)$.
A~\emph{potentially semisimple character (pss-character)} is a map~$\Theta$ from~$\mf{E}$ to~$\C(\mf{E})$ such that
\begin{enumerate}
 \item~$\Theta(\Delta)$ is an element of~$\C(\Delta)$, for all~$\Delta\in\mf{E}$, and
 \item the values of~$\Theta$ are related by transfer, for all semisimple strata~$\Delta,\Delta'\in\mf{E}$, i.e.~$\tau_{\Delta,\Delta'}(\Theta(\Delta))$ is equal 
 to~$\Theta(\Delta')$.
\end{enumerate}

We write sometimes~$\Theta=\Theta_\mf{E}$ to indicate the domain~$\mf{E}$. We define the degree and group level of~$\Theta_\mf{E}$ to be the respective degree and group level of~$\mf{E}$ and we denote the degree by~$\deg(\Theta)$. Two potentially semisimple characters~$\Theta_\mf{E}$ and~$\Theta'_{\mf{E}'}$ of the same group level and the same degree are called~\emph{endo-equivalent} if there are semisimple 
strata~$\Delta\in\mf{E}$ and~$\Delta'\in\mf{E}'$ such that~$\Theta(\Delta)$ and~$\Theta'(\Delta')$ intertwine.  
\end{definition}


\newcommand{\ImfE}{I_{\mf{E}}}
\newcommand{\ImfEP}{I_{\mf{E}'}}

\begin{definition}[comparison pairs]
Let~$\mf{E}$ and~$\mf{E}'$ be endo-classes of semisimple strata of the same group level and the same degree. We identify all index sets of semisimple
strata in~$\mf{E}$ to an index set~$\ImfE$. Suppose there is a bijection~$\zeta$ from~$I_{\mf{E}}$ to~$I_{\mf{E}'}$.
We call a pair~$(\Delta,\Delta')\in\mf{E}\times\mf{E}'$ a~\emph{$\zeta$-comparison pair} if~$D=D'$ and 
for all~$i\in\ImfE$ the vector spaces~$V^i$ and~$V'^{\zeta(i)}$ have the same~$D$-dimension.
\end{definition}

\begin{theorem}\label{thmEndo}
 Suppose~$\Theta_\mf{E}$ and~$\Theta'_{\mf{E}'}$ are two pss-characters. 
 Then the following assertions are equi\-valent:
 \begin{enumerate}
  \item~$\Theta'_{\mf{E}'}$ and~$\Theta_{\mf{E}}$ are endo-equivalent. 
  \item There is a bijection~$\zeta$ from~$I_\mf{E}$ to~$I_{\mf{E}'}$ such that for every~$\zeta$-comparison pair~$(\Delta,\Delta')$  
  the semisimple characters~$\Theta(\Delta)$ and~$\Theta'(\Delta')$ intertwine. 
 \end{enumerate}
\end{theorem}

By Proposition~\ref{propMatchingSameDegreeAndSameGroupLevel} the map~$\zeta$ is uniquely determined and we call it the~\emph{matching} from~$\Theta$ to~$\Theta'$. 
From Theorem~\ref{thmEndo} follows also that endo-equivalence is an equivalence relation, and we call the equivalence classes~\emph{endo-classes}. 
%

\begin{proof}
 This theorem is known for ps-characters by~\cite[Theorem~1.11]{broussousSecherreStevens:12} and Proposition~\ref{propTranstivityOfIntertwiningSimpleCharacters}. Suppose the pss-characters~$\Theta'_{\mf{E}'}$ and~$\Theta_{\mf{E}}$ are endo-equivalent,
 then there are semisimple strata~$\Delta\in\mf{E}$ and~$\Delta'\in\mf{E}'$ such that~$\Theta(\Delta)$ and~$\Theta'(\Delta')$ intertwine. Then there is a bijection~$\zeta$ from~$\ImfE$ 
 to~$\ImfEP$ such that~$\Theta(\Delta)_i$ and~$\Theta'(\Delta')_{\zeta(i)}$ intertwine. Thus, for every~$i\in \ImfE$, the ps-characters~$\Theta_i$ and~$\Theta'_{\zeta(i)}$
 corresponding to~$\Theta(\Delta)_i$ and~$\Theta'(\Delta')_{\zeta(i)}$ are endo-equivalent and thus the case of ps-characters together with Iwahori decompositions finishes the proof.  
\end{proof}

In the next section we will give the notion of endo-equivalence for certain semisimple characters.

\begin{definition}\label{defEndoEqChar}
\begin{enumerate}
 \item A semisimple stratum~$\Delta$ is called~\emph{full} if~$r=0$, a semisimple character~$\theta$ is called~\emph{full} if there is a full semisimple stratum~$\Delta$ such that~$\theta\in\C(\Delta)$, and a pss-character is called~\emph{full} if there is a full semisimple stratum in its domain.
An endo-class of a full pss-character is called~\emph{full}.
\item An endo-class of a ps-character is also called a~\emph{simple} endo-class and we denote the set of all full simple endo-classes by~$\mc{E}$.
\item Two full semisimple characters~$\theta$ and~$\theta'$ are called endo-equivalent if there 
are full semisimple strata~$\Delta$ and~$\Delta'$ 
such that~$\theta\in\C(\Delta)$ and~$\theta'\in\C(\Delta')$ and  such that the pss-characters~$\Theta$ and~$\Theta'$ with~$\Theta(\Delta)=\theta$ and~$\Theta'(\Delta')=\theta'$ are endo-equivalent. 
We call the equivalence class of~$\theta$ the endo-class of~$\theta$ and denote it by~$\mc{E}(\theta)$.
\end{enumerate}
\end{definition}
The set of endo-classes of full semisimple characters is in canonical bijection to the set of endo-classes of full pss-characters.

\section{Endo-parameters}\label{secEndoPar}

We classify the intertwining classes of full semisimple characters of~$G$ using full simple  endo-classes and numerical invariants. This data is called an endo-parameter. 


\begin{definition}
\begin{enumerate}
 \item An~\emph{endo-parameter} is a map~$f$ from the set~$\mathcal{E}$ of endo-classes of full ps-characters to the set of non-negative integers with finite support. 
 \item We define the~\emph{degree}~$\deg(f)$ of an endo-parameter~$f$ as the number
 \[\sum_{[\Theta]\in\mc{E}}f([\Theta])\deg(\Theta).\]
 \item An endo-parameter~$f$ is called an 
 ~\emph{endo-parameter for~$G$} if
 \begin{enumerate}
  \item the degree of~$f$ is equal to the degree of~$A$, i.e.~$\dim_DV \deg(D)$, 
  \item and~$f([\Theta])$ is divisible by~$\frac{\deg(D)}{\gcd(\deg(\Theta),\deg(D))}$.
 \end{enumerate}
 \end{enumerate}
\end{definition}

The second main theorem of this article is 

\begin{theorem}[\textbf{2nd Main Theorem}]\label{thmEndoParameter}
 There is a canonical bijection from the set of intertwining classes of full semisimple characters of~$G$ to the set of endo-parameters for~$G$. 
 The map has the following form: The intertwining class of a full semisimple character~$\theta$ is mapped to the endo-parameter~$f_\theta$
 which is supported on the set of the endo-classes corresponding to the~$\theta_i$. The value of~$f_\theta$ at~$\mathcal{E}(\theta_i)$ is
 defined to be the degree of~$\End_{E_i\otimes_FD}(V^i)$ over~$E_i$, i.e. it is the square root of the~$E_i$-dimension of~$\End_{E_i\otimes_FD}(V^i)$. 
\end{theorem}

The proof needs a proposition which shows that one can add up any two given semisimple characters to obtain a new semisimple character. 

\begin{proposition}[{\cite[Lemma 12.7]{KSS:21}}]\label{propAddingSemisimpleCharacters}
 Let~$\theta\in\C(\Delta)$ and~$\theta'\in\C(\Delta')$ be two semisimple characters with~$\Delta$ and~$\Delta'$ of the same period and~$r=r'$.  
 Then there is a semisimple stratum~$\Delta''$ split by~$V\oplus V'$, such that~$\Delta''|_V=\Delta$ (after lowering the second entry from~$n''$ to~$n$),~$\C(\Delta''|_{V'})=\C(\Delta')$
 and such that~$\theta\otimes\theta'$ is an element of~$\C(\Delta'')$.
\end{proposition}

\begin{proof}
 The proof is literally the same as in \loccit. For the translation principle one takes Theorem~\ref{thmTranslation} and 
 instead of Lemma~\cite[Lemma 12.4]{KSS:21} we take Lemma~\ref{lemTriangleToSquareExtension}. For the proof of the modification of~\cite[Lemma~12.6]{KSS:21} we use Corollary~\ref{corconsructingEmbeddingsFromNumericalData}. 
\end{proof}

\newcommand{\zetathetathetaP}{\zeta_{\theta,\theta'}}

\begin{proof}[Proof of Theorem~\ref{thmEndoParameter}]
We have to show that~$f_\theta$ and~$f_{\theta'}$ coincide for intertwining full semisimple characters~$\theta$ and~$\theta'$. The characters~$\theta_i$ and~$\theta'_{\zeta_{\theta,\theta'}(i)}$
intertwine by a~$D$-linear isomorphism~$g_i$ from~$V^i$ to~$V^{\zetathetathetaP(i)}$. Thus~$\theta_i$ and~$\theta_{\zetathetathetaP(i)}$
define the same endo-class and therefore~$f_\theta$ and~$f_{\theta'}$ have the same support. The vector spaces~$V^i$ and~$V^{\zetathetathetaP(i)}$ have the same dimension 
over~$D$, and~$E_i$ and~$E'_{\zetathetathetaP(i)}$ have the same degree over~$F$ by Corollary~\ref{propMatchingSameDegreeAndSameGroupLevel}. Thus~$f_\theta$ and~$f_{\theta'}$ coincide. 
 
 Surjectivity: Let us now take an endo-parameter~$f$ for~$G$. For every simple endo-class~$c$ in the support of~$f$, we pick a full simple character~$\theta_{c}$ which
 occurs in the image of one of the ps-characters in~${c}$. Proposition~\ref{propAddingSemisimpleCharacters} provides a semisimple 
 character~$\theta\in\C(\Delta)$ with restrictions~$\theta_{c}$,~$c\in\supp(f)$. To get a semisimple character for~$G$ we need to find an appropriate lattice 
 sequence. This is done as follows: Let~$M_{c}$ be a minimal~$E_{c}\otimes D$-module. It has dimension~$\frac{\deg(c)\dim_FD}{\gcd(\deg(c),\deg(D))}$ over~$F$. Now take a 
 vertex~$\Gamma_{c}$ in the Bruhat--Tits building of~$C_{\End_D(M_{c})}(E_{c})$ and apply the inverse of~$j_{E_{c}}$ to obtain an~$o_{E_{c}}$-$o_D$-lattice sequence~$\Lambda_{c}$ in~$M_{c}$.
 The vector space~$V$ is  a~$D$-vector space isomorphic to the direct sum of~$M_{c}$s, where~$M_{c}$ occurs with 
 multiplicity~$m_c:=f(c)\frac{\gcd(\deg(c),\deg(D))}{\deg(D)}$. Let~$\Delta_f$  be the direct sum of the simple strata~$[\Lambda_{c},n_{c},0,\beta_{c}]^{\oplus m_c}$ 
 over~$c\in\supp(f)$. The stratum~$\Delta_f$ is semisimple by a small argument using 
 Corollary~\ref{corEndoEquivalenceForSimpleStrata}\ref{corEndoEquivalenceForSimpleStrataAss.ii}, because~$\Delta$ is semisimple. We take the transfer~$\theta_f$ of~$\theta$ from~$\Delta$ 
 to~$\Delta_f$. Then~$f=f_{\theta_f}$ by construction, because
 \[  f_{\theta_f}(c)\deg(c)=\dim_DV^{c}\deg(D)=f(c)\gcd(\deg(c),\deg(D))\dim_DM_{c},\]
 and~$M_{c}$ has dimension~$\frac{\deg(c)}{\gcd(\deg(c),\deg(D))}$ over~$D$ as a simple~$E_{c}\otimes_F D$-module. Thus~$f_{\theta_f}(c)=f(c)$.

Injectivity: Take two full semisimple characters~$\theta\in\C(\Delta)$ and~$\theta'\in\C(\Delta')$  such that~$f_\theta=f_{\theta'}$. Then there is a bijection~$\zeta:\ I\ra I'$ such that~$\theta_i$ 
and~$\theta_{\zeta(i)}$ define the same endo-class~$[\Theta_i]$. For~$i\in I$ the $D$-dimensions of~$V^i$ and~$V^{\zeta(i)}$ coincide
because~$f_\theta([\Theta_i])=f_{\theta'}([\Theta_i])$. Take an isomorphism~$g_i$ between them.
Thus~$(g_i.\Delta_i,\Delta'_{\zeta(i)})$ is a comparison pair for
the  ps-characters attached to~$\theta_i$ and~$\theta'_{\zeta(i)}$. Thus~$g_i.\theta_i$ and~$\theta'_{\zeta(i)}$ intertwine by~Theorem~\ref{thmEndo}. Thus~$\theta$ 
and~$\theta'$ intertwine, because~$\theta$ and~$\theta'$ decompose under the respective Iwahori decompositions. 
\end{proof}

\def\Circlearrowleft{\ensuremath{%
  \rotatebox[origin=c]{180}{$\circlearrowleft$}}}

\bibliographystyle{plain}
\bibliography{/home/zhou/LaTeX/bib/bibliography}
\end{document}